\documentclass[11pt,reqno]{amsart}

\usepackage{amsmath}
\usepackage{amsthm}
\usepackage{amssymb}
\usepackage{enumerate}
\usepackage{amscd}

\usepackage[latin1]{inputenc}

\title[]{Closed minimal surfaces of index one in Riemannian manifolds}

\author{Fernando C. Marques and  Andr\'e Neves}

\date{}

\address{Princeton University \\ Fine Hall \\ Princeton NJ 08544 \\ USA}
\email{coda@math.princeton.edu}
\address{University of Chicago \\ Department of Mathematics \\ Chicago IL 60637\\ USA}
\email{aneves@math.uchicago.edu}
\thanks{ The first  author is partly supported by NSF-DMS-2506810 and a Simons Investigator Grant. The second author is partly supported by NSF-DMS-2005468 and a Simons Investigator Grant. Part of this work was developed by the authors during their stay at SLMath for the Fall 2024 semester geometry program.}

\newtheorem{thm}{Theorem}[section]

\newtheorem{prop}{Proposition}[section]

\newtheorem{cor}{Corollary}[section]

\def\XXint#1#2#3{{\setbox0=\hbox{$#1{#2#3}{\int}$}
     \vcenter{\hbox{$#2#3$}}\kern-.5\wd0}}

\begin{document}


\begin{abstract}

In this paper we prove that an $(n+1)$-manifold, compactly $n$-enlargeable, where $3\leq (n+1)\leq 7$,   has connected, immersed Morse index one,  closed minimal hypersurfaces with unbounded volumes for bumpy  metrics.   We prove that in the three-dimensional case the hypersurfaces are geometrically distinct using cyclic coverings of manifolds with boundary. The proof extends to   $(n+1)$-fiberings. 
We   prove a     scalar curvature rigidity theorem  for area-nonincreasing maps of three-dimensional manifolds. The case of stable  surfaces  is also discussed by   using   cohomology classes and    incompressible surfaces.

\end{abstract}

\maketitle



\section{Introduction}

In this paper we are interested in the space of immersed, closed minimal hypersurfaces with Morse index one in a  closed Riemannian manifold.

Let $(M^{n+1},g)$ be a closed Riemannian manifold, $3\leq (n+1) \leq 7$. By the min-max theory for the area functional of Almgren and Pitts (\cite{almgren-varifolds}, \cite{pitts}),   there is a closed  minimal hypersurface $\Sigma$ in $(M^{n+1},g)$ which is smooth and embedded. The hypersurface $\Sigma$ is such that ${\rm index}(\Sigma)\leq 1$ (\cite{marques-neves-index}).

A metric $g$  is bumpy if any closed immersed minimal hypersurface is nondegenerate. By \cite{white-bumpy}, a smooth generic metric is bumpy. For these  metrics,   there is a two-sided, connected component  of  $\Sigma$ with index one   by \cite{marques-neves-lower-bound} and \cite{zhou-multiplicity}.

We would like to count closed minimal hypersurfaces of index one using  topological conditions on the ambient space.

 The Theorem \ref{index.area.introduction} discusses      compactly $n$-volume enlargeable manifolds.

\begin{thm}\label{index.area.introduction}
Let $M^{n+1}$ be a  compactly $n$-volume enlargeable closed manifold, $3\leq (n+1)\leq 7$. Suppose  $g$ is a bumpy metric on $M$. Then, for any $\lambda>0$, there is a connected,  smooth, two-sided, virtually embedded, closed minimal hypersurface $\psi:\Sigma\rightarrow M$, of Morse index one, with 
$$
{\rm vol}_g(\Sigma)\geq \lambda.
$$
\end{thm}

Enlargeable manifolds were introduced  by Gromov and Lawson  to find obstructions to positive scalar curvature
(\cite{gromov-lawson-scalar}, \cite{gromov-lawson}).

An immersed hypersurface $\psi:\Sigma \rightarrow M$ is virtually embedded if there is a finite cover $M'$ of $M$ such that $\psi$ lifts to an embedding $\psi':\Sigma\rightarrow M'$. We denote by ${\rm vol}_g(\Sigma)$ the volume of $\Sigma$ in the pull-back metric $\psi^*(g)$. Hence ${\rm vol}_g(\Sigma)={\rm vol}(\psi'(\Sigma))$, where $\rm vol$ is the $n$-dimensional Hausdorff measure.

A closed Riemannian manifold with nonpositive sectional curvature such that the fundamental group is  residually finite is compactly  enlargeable, and therefore compactly $n$-volume enlargeable. This includes the nonpositively curved closed locally symmetric spaces (\cite{gromov-lawson-scalar}).

By the Theorem \ref{index.area.introduction} there is a sequence of immersed closed minimal hypersurfaces $\psi_i:\Sigma_i\rightarrow M$ of Morse index one with unbounded volumes.  There is the question of whether these  hypersurfaces are geometrically distinct: $\psi_i(\Sigma_i)\neq \psi_j(\Sigma_j)$ for $i\neq j$. To prove that one would have to rule out   coverings. It could be that there is a covering map $\pi:\Sigma_j\rightarrow \Sigma_i$ such that $\psi_j=\psi_i \circ \pi$.

The Theorem \ref{nonspherical-introduction} answers this question in the three-dimensional case.

\begin{thm}\label{nonspherical-introduction}
Let $(M,g)$ be a closed orientable three-dimensional manifold. Suppose  that $M$ is not   a connected sum of spherical quotients and copies of  $S^2\times S^1$. If $g$ is bumpy, then there is a sequence of  connected,  smooth, two-sided, virtually embedded, closed minimal surfaces $\psi_i:\Sigma_i\rightarrow M$, of Morse index one, such that 
 $$
 {\rm area}_g(\psi_i(\Sigma_i))\rightarrow \infty.
 $$
\end{thm}

A closed orientable three-manifold  is compactly area enlargeable if and only if it is not a connected sum of spherical quotients and copies of  $S^2\times S^1$. By Theorem \ref{nonspherical-introduction}, since the area of the images of the index one minimal immersions is unbounded there is $\{i'\}\subset \{i\}$ such that the   surfaces $\psi_{i'}(\Sigma_{i'})$ are geometrically distinct.

If $M$ is diffeomorphic to a connected sum of spherical quotients and copies of  $S^2\times S^1$, there is a metric on $M$ of positive scalar curvature $R_g$.    If $\psi:\Sigma\rightarrow M$ is an immersed, two-sided, closed minimal surface of Morse index one for a metric $g$ such that $R_g\geq 6$, then (e.g. appendix of \cite{marques-neves-rigidity-spheres})
${\rm area}_g(\psi(\Sigma))\leq 16\pi/3.$
Therefore for these manifolds  there is an open set of metrics $\Lambda$ and a constant $C>0$   such that 
$$
{\rm area}_g(\psi(\Sigma))\leq C
$$
for any $g\in \Lambda$ and $\psi:\Sigma\rightarrow M$  an immersed, two-sided, closed minimal surface of Morse index one in the metric $g$. Since $\Lambda$ is an open set there is a bumpy metric $g'\in \Lambda$.

 This proves that the topological hypothesis of Theorem \ref{nonspherical-introduction} is sharp.

By  \cite{montezuma}, there are sequences of closed orientable three-manifolds $(M_i,g_i)$ such that the Almgren-Pitts   min-max widths are unbounded and  $R_{g_i}\geq 6$. In this case the min-max minimal surface is disconnected: there is a component of Morse index one with area bounded by $16\pi/3$ and 
$k_i$ stable components that are diffeomorphic to spheres with areas bounded by $4\pi/3$.

We do not know in the case of  the Theorem \ref{nonspherical-introduction} if there are embedded sequences of such surfaces.

The proof of Theorem \ref{index.area.introduction} uses $n$-volume nonincreasing maps    to construct connected, closed minimal hypersurfaces  of index one with volume bounds.

A Lipschitz map $f:(M,g)\rightarrow (N,h)$ of $(n+1)$-dimensional Riemannian manifolds  is $n$-volume nonincreasing if
$$
|df_x(v_1)\wedge \cdots \wedge df_x(v_n)|_h\leq |v_1\wedge \cdots \wedge v_n|_g
$$
for almost any $x\in M$ where $df_x$ is defined,  and $\{v_i\}_{i=1}^n\subset T_xM$. Then
$$
{\rm vol}_h(f(\Sigma))\leq {\rm vol}_g(\Sigma)
$$
for any hypersurface
$\Sigma^n\subset M$. 


For Riemannian metrics $g$ and $h$ on a sphere $S^{n+1}$, one could ask if there is an $n$-volume nonincreasing map $f:(S^{n+1},g)\rightarrow 
(S^{n+1},h)$
of degree ${\rm deg}(f)\neq 0$. Since $f$ is surjective, it follows that ${\rm vol}(S^{n+1},h)\leq {\rm vol}(S^{n+1},g)$.

The  Theorem \ref{homology-introduction} finds an obstruction which  involves the volumes of index one, closed, connected, embedded, minimal hypersurfaces for   metrics $g,h$  that  do not have stable degenerate embedded closed minimal hypersurfaces. 
This is the case if the metrics  $g,h$ are bumpy.

 \begin{thm}\label{homology-introduction}
 Suppose that $g$ and $h$ are smooth Riemannian metrics on $S^{n+1}$  that do  not have stable degenerate, closed, two-sided, embedded   minimal hypersurfaces, $3\leq (n+1)\leq 7$.  Let $f:(S^{n+1},g)\rightarrow (S^{n+1},h)$ be an $n$-volume nonincreasing Lipschitz map of degree $k\neq 0$. Then there is a smooth, connected, embedded, closed minimal hypersurface $\Sigma$ in $(S^{n+1},g)$,  with Morse index   one, two-sided, and there is a smooth, connected, embedded, closed minimal hypersurface $\Sigma'$ in $(S^{n+1},h)$, with Morse index   one, two-sided, such that
$$
{\rm vol}_h(\Sigma')\leq  {\rm vol}_g(\Sigma).
$$
\end{thm}

To apply $n$-volume nonincreasing maps to prove Theorem \ref{index.area.introduction},  we prove a version of Theorem \ref{homology-introduction} for maps $f:(M,g)\rightarrow (N,h)$ of $(n+1)$-manifolds. Suppose that the metrics $g$ and $h$ do not have closed, two-sided, stable, degenerate,    embeddded minimal hypersurfaces or closed, one-sided, with stable, degenerate, oriented  two-cover, embedded  minimal hypersurfaces. 
This is the case if $g$ and $h$ are bumpy metrics. Then we can find the minimal hypersurfaces $\Sigma$ and $\Sigma'$ satisfying 
${\rm vol}_h(\Sigma')\leq  {\rm vol}_g(\Sigma)$ with ${\rm index}(\Sigma)=1$ and ${\rm index}(\Sigma')\leq 1$ (Theorem \ref{connected.minimal}).

A closed manifold $M$ is compactly $n$-volume enlargeable, or compactly $n$-enlargeable,  if for any Riemannian metric $g$ on $M$ and  $r>0$, there is a finite cover $M'$ of $M$ that has an $n$-volume nonincreasing
Lipschitz map  $f:(M',g) \rightarrow (S_r^{n+1},\overline{g})$, ${\rm deg}(f)\neq 0$. 
The pull-back metric on $M'$ is denoted by $g$, and $(S_r^{n+1},\overline{g})$ is the round sphere of radius $r$. Notice that this definition does not require $M'$ to be spin.

The sphere $(S_r^{n+1},\overline{g})$ does not have stable closed minimal hypersurfaces, and the Morse index one closed minimal hypersurfaces are the equators which have volume $\sigma_n r^{n}$, where $\sigma_n$ is the volume of  the $n$-dimensional unit sphere. Hence Theorem \ref{index.area.introduction} can be proved since in this case ${\rm vol}(\Sigma')\geq \sigma_nr^n$.

A manifold that has a map of nontrivial degree to a compactly $n$-volume enlargeable manifold is compactly $n$-volume enlargeable.

Suppose $\phi:M\rightarrow M$ is a  diffeomorphism of the  closed $n$-dimensional manifold $M$. Then the mapping torus of $\phi$ is 
the manifold $M_\phi$ defined by the identification of $(1,x)$ and $(0,\phi(x))$ in the product $M \times [0,1]$ for any $x\in M$. This induces a fiber bundle $M_\phi\rightarrow S^1$ with fiber $M$. 

The Theorem \ref{fiber.bundle-introduction} discusses the case of virtual fiberings.

\begin{thm}\label{fiber.bundle-introduction}
Let $M_\phi$ be an $(n+1)$-dimensional mapping torus  with fiber $M$ such that  $b_1(M)>0$, and     $3\leq (n+1)\leq 7$.  Suppose   $M'$ is a closed manifold such that there is a finite cover $M''\rightarrow M'$  that has a continuous map $h:M''\rightarrow M_\phi$ of nontrivial degree. If $g$ is a bumpy metric on $M'$, then $(M',g)$
contains a sequence of  geometrically distinct connected,  two-sided, immersed, closed minimal hypersurfaces  of Morse index one with unbounded volumes.
 \end{thm}

 Theorem \ref{fiber.bundle-introduction} is not true for the fibering $S^2\times S^1$ as discussed previously.

The Theorem  \ref{nonspherical-introduction} and Theorem \ref{fiber.bundle-introduction} will be proved using the  Theorem  \ref{boundary.case} for compact manifolds with boundary. The idea is to argue by contradiction.  We suppose that there are finitely many index one minimal hypersurfaces. 
Then we use a sequence of cyclic coverings such that these hypersurfaces either lift  or are covered by minimal hypersurfaces with higher Morse  index. We then construct  using  min-max  theory a distinct minimal hypersurface of index one    in a cyclic covering.

We use the Theorem 1 of Brooks \cite{brooks} on spectral gaps of cyclic coverings.

In the case the manifold of Theorem \ref{nonspherical-introduction} has a hyperbolic manifold in its prime decomposition, we use the virtual fibering theorem of Agol \cite{agol}.


A closed locally symmetric space of nonpositive sectional curvature with  Kazhdan's property T (\cite{kazhdan}) does not satisfy the topological hypothesis  of the Theorem  \ref{fiber.bundle-introduction}. Since these spaces are compactly $n$-volume enlargeable  Theorem 
\ref{index.area.introduction} can be used. It would be interesting then to find sequences $\psi_i(\Sigma_i)$ of geometrically distinct such hypersurfaces.

 The proof of   Theorem \ref{cohomology.index-introduction} uses Theorem \ref{stable} for stable surfaces.
 
 \begin{thm}\label{cohomology.index-introduction}
Let $M^{n+1}$, $3\leq (n+1)\leq 7$, be a closed manifold such that the cup product 
$$
\smile\,: H^1(M,\mathbb{Z})\times H^1(M,\mathbb{Z})\rightarrow H^2(M,\mathbb{Z})
$$
is nontrivial. Then, for any smooth bumpy Riemannian metric $g$ on $M$, there is a sequence of smooth, two-sided,  connected closed embedded minimal 
hypersurfaces  $\Sigma_i\subset (M,g)$, of Morse index one, such that  ${\rm vol}_g(\Sigma_i)\rightarrow \infty$. 
\end{thm}

We discuss the case of stable surfaces in Section \ref{homology.incompressible}.

  Llarull  proved a rigidity result for scalar curvature in any dimension. 
  
   If $(M^{n+1},g)$ is a closed Riemannian spin manifold,  with $R_g\geq (n+1)n$, and $f:(M^{n+1},g)\rightarrow (S^{n+1},\overline{g})$ is a smooth area-nonincreasing map,  then $f$ is a  Riemannian isometry (\cite{llarull}). 
   
   We prove a scalar curvature rigidity theorem for the case of three-manifolds involving minimal surfaces.

   \begin{thm}\label{rigidity-introduction}
Let $(M^3,g)$ be a closed, oriented, Riemannian manifold  such that  $R_g\geq 6$. Suppose     $(N^3,h)$ is a closed,  oriented, Riemannian manifold such that if $\Sigma$ is a  closed, embedded, minimal surface in $(N,h)$,   ${\rm area}_h(\Sigma)\geq4\pi$.   Suppose $f:(M,g)\rightarrow (N,h)$ is a Lipschitz  area-nonincreasing  map such that  ${\rm deg}(f)$ is nontrivial. Then $(M,g)$ is isometric to the unit sphere and  $f$ is a smooth Riemannian  isometry.
  \end{thm}

 If $f$ is the identity map this is Theorem 2 of \cite{song}.
 The proof of the Theorem \ref{rigidity-introduction} as in the rigidity theorems of \cite{marques-neves-rigidity-spheres} and  \cite{song} uses minimal surface theory and the Ricci flow. 
 
 As in Song \cite{song-existence} we use  decompositions by stable minimal hypersurfaces to prove the main  theorems.

The paper is organized as follows. In Section 2, we discuss the notation of min-max theory. In Section 3,   we prove the Theorems \ref{index.area.introduction} and \ref{homology-introduction}. In Section 4, we prove the Theorems \ref{nonspherical-introduction} and
\ref{fiber.bundle-introduction}. In Section 5, we prove the Theorem \ref{rigidity-introduction}. In Section 6, we discuss the case of stable surfaces.

\section{Preliminaries}

Let $R$ be an $(n+1)$-dimensional compact manifold possibly with boundary, and let $G$ be the group $\mathbb{Z}$ or $\mathbb{Z}_2$. Denote by ${\bf I}_{j}(R,G)$ the set of $j$-dimensional flat chains  with support contained in $R$ and coefficients in $G$, endowed with the flat topology. In the case $G=\mathbb{Z}$, these flat chains are called integral currents. We denote by ${\rm spt}(T)$ the support of $T\in {\bf I}_{j}(R,G)$. 
Let $\tilde{\mathcal{Z}}_n(R,G)$ be the set of chains  $T\in {\bf I}_n(R,G)$ such that $\partial T=0$, and
$\mathcal{Z}_n(R,G)$ be the set of  $T\in {\bf I}_n(R,G)$ such that there is $\Omega\in {\bf I}_{n+1}(R,G)$ with $T=\partial \Omega$, where $\partial$ is the boundary map.  
Since $\partial (\partial \Omega) =0$ for any $\Omega$, $\mathcal{Z}_n\subset \tilde{\mathcal{Z}}_n$.   

Suppose that $R$ has a Riemannian metric $g'$. The flat topology on ${\bf I}_{j}(R,G)$ is induced by the flat metric
$$
 \mathcal F_{g'}(T,S) = \inf \{{\bf M}_{g'}(Q)+{\bf M}_{g'}(Y): T-S=Q+\partial Y\},
 $$
 where ${\bf M}_{g'}(T)$ denotes the $j$-dimensional mass of $T$ induced by $g'$ if $T\in {\bf I}_j(R,G)$.  We denote by $|T|$ the $j$-dimensional integral varifold induced by $T\in {\bf I}_j(R,G)$.  The  ${\bf F}$-metric on the space of varifolds  is defined in  the book of Pitts  \cite{pitts}, and the ${\bf F}$-metric on ${\bf I}_j(R,G)$ is defined by
$$ {\bf F}_{g'}(S,T)=\mathcal{F}_{g'}(S, T)+{\bf F}_{g'}(|S|,|T|).$$
The mass topology is induced by the metric ${\bf M}_{g'}(S,T)={\bf M}_{g'}(S-T)$.

Let  $\Phi:[s_1,s_2]\rightarrow \tilde{\mathcal{Z}}_n(R,G)$  be  a map   continuous in the flat topology. We say that $\Phi$ is a sweepout of $\Omega'$ if there is a map $\Omega:[s_1,s_2]\rightarrow {\bf I}_{n+1}(R,G)$, with $\Omega(s_1)=0$, $\partial \Omega(t)=\Phi(t)-\Phi(s_1)$, and continuous in the flat topology such that $\Omega'=\Omega(s_2)$. 
  The map $\Omega$, and hence $\Omega'$, is uniquely defined by the constancy theorem of  flat chains.

Let $(M,g)$ be a compact oriented Riemannian manifold, possibly with boundary, of dimension $3\leq (n+1)\leq 7$, and let $(N,h)$
be an $(n+1)$-dimensional closed oriented Riemannian manifold.

A Lipschitz map $f:M\rightarrow N$  is $n$-volume nonincreasing if
$$
|df_x(v_1)\wedge \cdots \wedge df_x(v_n)|_h\leq |v_1\wedge \cdots \wedge v_n|_g
$$
for almost any $x\in M$ where $df_x$ is defined,  and $\{v_i\}_{i=1}^n\subset T_xM$. 
Let $f_{\#}: {\bf I}_n(M,\mathbb{Z})\rightarrow  {\bf I}_n(N,\mathbb{Z})$ be the push-forward map induced by the Lipschitz map $f$ on $n$-dimensional currents. Since the map $f$ is $n$-volume nonincreasing, 
${\bf M}_h(f_{\#}(T))\leq {\bf M}_g(T)$
for $T\in  {\bf I}_n(M,\mathbb{Z})$. The map $f_{\#}$ is continuous in the flat and in the mass topologies. If $\Sigma\subset M$ is a hypersurface, we denote by ${\rm vol}_g(\Sigma)$ or by $|\Sigma|_g$ the  $n$-dimensional Hausdorff measure of $\Sigma$ induced by $g$.

Recall that by \cite{almgren} there is a natural isomorphism
$$
\Lambda: \pi_1(\mathcal{Z}_n(N,\mathbb{Z}), 0)\rightarrow H_{n+1}(N,\mathbb{Z})\simeq\mathbb{Z},
$$
where $\pi_1(X,x)$ denotes the fundamental group of $X$ based at $x\in X$ and $H_i(X,\mathbb{Z})$ denotes the $i$-dimensional integer homology group of $X$.
Denote by $\Gamma^{(k)}(N)$ the set of maps $\Phi:[0,1]\rightarrow \mathcal{Z}_n(N,\mathbb{Z})$, $\Phi(0)=\Phi(1)=0$, that are continuous in the ${\bf F}$-metric,
and  such that the homotopy class $[\Phi]$ of $\Phi$, in the flat topology, satisfies $\Lambda([\Phi])=k \cdot [N]$. 
We define the invariant
$$
W^{(k)}(N,h)= \inf_{\Phi \in \Gamma^{(k)}(N)} \sup_{t\in [0,1]} {\bf M}_h(\Phi(t)).
$$
By concatenating maps in $\Gamma^{(1)}(N)$,  it follows that $W^{(k)}(N,h) \leq W^{(1)}(N,h)$ for any $k\in \mathbb{Z}$.

Notice that $W^{(k)}(S_r^{n+1},\overline{g})=r^n W^{(k)}(S_1^{n+1},\overline{g})$ for any $k\in \mathbb{Z}$, where $(S_r^{n+1},\overline{g})$ denotes the $(n+1)$-dimensional round sphere of radius $r$.  Then 
 $$W^{(k)}(S_1^{n+1},\overline{g})\geq {\rm vol}_{\overline{g}}(S_1^n)
 $$ by using min-max theory for any $k\in \mathbb{Z}\setminus \{0\}$.  This is because if $\Sigma$ is a closed smooth minimal hypersurface of  $(S_1^{n+1},\overline{g})$ then ${\rm vol}_{\overline{g}}(\Sigma)\geq {\rm vol}_{\overline{g}}(S_1^n)$. Hence $W^{(k)}(S_1^{n+1},\overline{g})= {\rm vol}_{\overline{g}}(S_1^n)$  for $k\in 
 \mathbb{Z}\setminus \{0\}$, since
 $$
  {\rm vol}_{\overline{g}}(S_1^n)\leq W^{(k)}(S_1^{n+1},\overline{g}) \leq W^{(1)}(S_1^{n+1},\overline{g})\leq {\rm vol}_{\overline{g}}(S_1^n).
  $$

Suppose that if $\Sigma'$ is a connected component of $\partial M$, then $f_{\#}(\Sigma')=0$. Then $\partial (f_{\#}(M)) = f_{\#}(\partial M)=0$, and hence $f_{\#}(M)=k \cdot [N]$ for some $k\in \mathbb{Z}$ by the constancy theorem of integral currents. Here $M\in {\bf I}_{n+1}(M,\mathbb{Z})$ and $N\in {\bf I}_{n+1}(N,\mathbb{Z})$ are the currents induced by the oriented fundamental classes. The degree of $f$ is defined as ${\rm deg}(f)=k$.

In the case that $M$ has a boundary, $\partial M=\sum_{i=1}^tS_i \in \mathcal{Z}_n(M,\mathbb{Z})$  where $S_i\in \tilde{\mathcal{Z}}_n(M,\mathbb{Z})$ is the integral current of
a component of the boundary with the induced orientation. 
Let  $\Gamma^{(1)}(M)$  be   the set of maps $\Phi:[0,1]\rightarrow \tilde{\mathcal{Z}}_n(M,\mathbb{Z})$ that are continuous in the ${\bf F}$-metric,
$\Phi(0)=-\sum_{i\in I_1} S_i$, $\Phi(1)=\sum_{i\in I_2}S_i$, so that  $I_1$, $I_2$ are disjoint sets depending on $\Phi$ with $I_1\cup I_2 =\{i:1\leq i\leq t\}$. We define
$$
W^{(1)}(M,g)= \inf_{\Phi \in \Gamma^{(1)}(M)} \sup_{t\in [0,1]} {\bf M}_g(\Phi(t)).
$$
Notice that for $\Phi\in\Gamma^{(1)}(M)$,  if $\Omega:[0,1]\rightarrow {\bf I}_{n+1}(M,\mathbb{Z})$ is the map, continuous in the flat topology, such that $\Omega(0)=0$, and $\partial \Omega(t)=\Phi(t)-\Phi(0)$, then $\Omega(1)=M$. Since the number of partitions of the set of components of the boundary is finite, 
$W^{(1)}(M,g)$ is the min-max invariant for the homotopy class of based maps for the disjoint sets $I_1',I_2'$,  $I_1'\cup I_2' =\{i:1\leq i\leq t\}.$

We can define a min-max invariant for modulo two boundaries. Let  $\Phi:[0,1]\rightarrow \mathcal{Z}_n(M,\mathbb{Z}_2)$
be a map continuous in the ${\bf F}$-metric, such that $\Phi(t)=\partial \Omega(t)$, with $\Omega:[0,1]\rightarrow {\bf I}_{n+1}(M,\mathbb{Z}_2)$ so that  $\Omega(0)=M$,  $\Omega(1)=0$, continuous in the flat topology.  Let $\Pi$ be the $([0,1],\{0,1\})$-homotopy class of $\Phi$ in the notation of \cite{zhou-multiplicity}. Then $W_{g}(M)$ is the width of $\Pi$ for the area functional in the metric $g$: 
$$
W_g(M)= \inf_{\tilde{\Phi} \in \Pi} \sup_{t\in [0,1]} {\bf M}_g(\tilde{\Phi}(t)).
$$  


\section{Minimal hypersurfaces of Morse index one}

A smooth, generic metric 
 is bumpy, in the sense that any immersed, closed, minimal hypersurface is nondegenerate, by Theorem 2.1 of \cite{white-bumpy}. If $g$ is a bumpy metric, then there is no stable, degenerate, closed, two-sided, immersed,  minimal hypersurfaces or one-sided, closed, immersed, minimal hypersurfaces with stable, degenerate, oriented  two-cover.

\begin{prop}\label{index.one.boundary}
Let $(\Omega,g)$ be a compact Riemannian manifold, of dimension $3\leq (n+1)\leq 7$, with stable, minimal boundary $\partial \Omega$ if $\partial \Omega \neq \emptyset$. Suppose that $g$ has no stable, degenerate, closed, two-sided, embedded   minimal hypersurfaces or one-sided, closed, embedded minimal hypersurfaces with stable, degenerate, oriented  two-cover. Then 
  there is a smooth, connected, embedded, closed minimal hypersurface $\Sigma$ in $\Omega$, disjoint from $\partial \Omega$, two-sided, with Morse index   one, such that ${\rm vol}_g(\Sigma)\leq W_g(\Omega)$.
\end{prop}

\begin{proof}

Let $g_i$ be a sequence of bumpy metrics converging smoothly to $g$ such that $\partial \Omega$ is a minimal hypersurface for $g_i$. Consider a map $\Phi:[0,1]\rightarrow \mathcal{Z}_n(\Omega,\mathbb{Z}_2)$, continuous in the ${\bf F}$-metric, such that $\Phi(t)=\partial \Omega(t)$, with $\Omega:[0,1]\rightarrow {\bf I}_{n+1}(\Omega,\mathbb{Z}_2)$ so that  $\Omega(0)=\Omega$,  $\Omega(1)=0$, continuous in the flat topology.  Let $\Pi$ be the $([0,1],\{0,1\})$-homotopy class of $\Phi$ in the notation of \cite{zhou-multiplicity}. Let $W_{g_i}(\Pi)$ be the width of $\Pi$ for the area functional in the metric $g_i$.   
Then $W_{g_i}(\Pi)>{\bf M}_{g_i}(\partial \Omega)$, since $\partial \Omega$ is strictly stable. This follows for instance by applying Theorem 6.1 of \cite{marques-neves-lower-bound}.
By applying the multiplicity one theorem of Zhou (Theorem 4.1, \cite{zhou-multiplicity}), combined with the index estimates of \cite{marques-neves-lower-bound}, there is a smooth, connected, embedded, closed minimal hypersurface $\Sigma_i \subset {\rm int}(\Omega)$ for the metric $g_i$, two-sided, with Morse index one, such that ${\rm vol}_{g_i}(\Sigma_i)\leq W_{g_i}(\Omega)$.

By \cite{sharp}, the sequence $\Sigma_i$ converges in varifold sense to a smooth, connected, embedded, closed minimal hypersurface $\Sigma$ for the metric $g$ with integer multiplicity $k$.  If the surface $\Sigma$  is one-sided, then $k\geq 2$
and the oriented two-cover $\tilde{\Sigma}$ of $\Sigma$ is stable degenerate.  This is a contradiction.  Hence $\Sigma$ is two-sided. Since $\Sigma$ is not stable degenerate, the convergence $\Sigma_i\rightarrow \Sigma$ is smooth with multiplicity one. 
Then ${\rm index}(\Sigma)\leq 1$, and ${\rm vol}_{g}(\Sigma)\leq W_{g}(\Omega)$. If the surface $\Sigma$ is stable, it would be strictly stable and hence $\Sigma_i$ is strictly stable for sufficiently large $i$. This is a contradiction. Hence ${\rm index}(\Sigma)=1$. The surface $\Sigma$ is disjoint from $\partial \Omega$ by the maximum principle, since $\partial \Omega$ is stable. This finishes the proof of the proposition.
\end{proof}

\begin{prop}\label{no.stable.two-sided.case}
Let $(\Omega,g)$  be as in Proposition \ref{index.one.boundary}. Suppose that  $\Sigma\subset {\rm int}(\Omega)$ is a closed, two-sided, embedded, unstable minimal hypersurface,  and  that there is no smooth, two-sided, closed, stable minimal hypersurface $\tilde{\Sigma}$ in ${\rm int}(\Omega)$ with ${\rm vol}_g(\tilde{\Sigma})<{\rm vol}_g(\Sigma)$. Then there is a map $\tilde{\Phi}\in \Gamma^{(1)}(\Omega)$ such that
 $$
\sup_{t\in [0,1]}{\bf M}_g(\tilde{\Phi}(t))\leq {\bf M}_g(\Sigma),
 $$
 and hence  $W^{(1)}(\Omega,g)\leq  {\bf M}_g(\Sigma)$.
\end{prop}

\begin{proof}
There is a neighborhood $V$ of $\Sigma$ such that the boundary of $\Omega\setminus V$ has two strictly mean-convex components $\Sigma'$, $\Sigma''$, 
$\max\{{\rm vol}_g(\Sigma'), {\rm vol}_g(\Sigma'')\}< {\rm vol}_g(\Sigma)$, since $\Sigma$ is unstable.

Suppose that the hypersurface $\Sigma$ does not separate $\Omega$.  By minimizing the area in
the integer homology class of $\Sigma'$, inside $\Omega\setminus V$, since $\Sigma''$ is a barrier, there is a two-sided, strictly stable, closed minimal hypersurface $\tilde{\Sigma}$
in ${\rm int}(\Omega\setminus V)$, with ${\rm vol}_g(\tilde{\Sigma})< {\rm vol}_g(\Sigma)$. This is a contradiction.

Suppose then  that $\Sigma$  separates $\Omega$. Denote by  $\Omega^{(1)}$ and  $\Omega^{(2)}$  the connected components of $\Omega\setminus \Sigma$. We can suppose that   $\Sigma'\subset \Omega^{(1)}$, $\Sigma''\subset \Omega^{(2)}$ and that there is a smooth foliation  $\{\Sigma(t)\}_{t\in [-1,1]}$ of $\overline{V}$, such that $\Sigma(-1)=\Sigma'$, $\Sigma(0)=\Sigma$, $\Sigma(1)=\Sigma''$,  $f(t)={\bf M}_g(\Sigma(t))< {\bf M}_g(\Sigma)$ for  $t\neq 0$ and $f''(0)<0$. Consider a map  $\Phi:[0,1]\rightarrow \tilde{\mathcal{Z}}_n(\overline{\Omega^{(1)}}\setminus V,\mathbb{Z})$, continuous in the ${\bf F}$-metric, 
$\Phi(0)=-((\partial \Omega^{(1)})\setminus \Sigma)$, $\Phi(1)=\Sigma'$,
and so that $\Phi$ is a sweepout of $\overline{\Omega^{(1)}}\setminus V$.
Denote by $[\Phi]$ the class of such maps, and let
$$
W([\Phi])=\inf_{\Phi'\in [\Phi]} \sup_{t\in [0,1]}{\bf M}_g(\Phi'(t)).
$$
Then $W([\Phi])>{\bf M}_g((\partial \Omega^{(1)})\setminus \Sigma)$, since $(\partial \Omega^{(1)})\setminus \Sigma$ is strictly stable. 

If $W([\Phi])>{\bf M}_g(\Sigma')$, then we can apply min-max theory to $[\Phi]$ as $\partial(\Omega^{(1)}\setminus V)$ is mean-convex. We use for this  Theorem 1.7 of \cite{marques-neves-index}. We apply the result to a sequence of bumpy metrics converging to $g$, and then take a limit of the hypersurfaces as in the proof of Proposition \ref{index.one.boundary}.  This implies that there is a connected, smooth, closed, embedded, minimal hypersurface $\tilde{\Sigma}'\subset {\rm int}(\Omega^{(1)}\setminus V)$, that is either two-sided,  unstable or one-sided with unstable oriented two-cover. In any case   there is a neighborhood $\tilde{V}$ of $\tilde{\Sigma}'$ with boundary $\partial \tilde{V}$ that is strictly mean-concave.  Let $\Omega'$  be the component of  $\Omega^{(1)}\setminus (V\cup \tilde{V})$ such that $\Sigma'\subset \overline{\Omega'}$.  As $\partial \tilde{V}\cap \overline{\Omega'}$ is a barrier, by minimizing the area in the homology class of $\Sigma'$ with integer coefficients, inside
$\overline{\Omega'}$, we find a two-sided, strictly stable, embedded minimal hypersurface 
$\tilde{\Sigma}\subset {\rm int}(\Omega')$ with ${\rm vol}_g(\tilde{\Sigma})<{\rm vol}_g(\Sigma)$. As before  this implies   a contradiction. 

Hence  $W([\Phi])\leq {\bf M}_g(\Sigma')$. Let $\eta>0$ be such that  ${\rm vol}_g(\Sigma')\leq {\bf M}_g(\Sigma)-2\eta$ and ${\rm vol}_g(\Sigma'')\leq {\bf M}_g(\Sigma)-2\eta$.  It follows by the definition of $W([\Phi])$ that we can find  a map $\Phi':[0,1]\rightarrow \tilde{\mathcal{Z}}_n(\overline{\Omega^{(1)}}\setminus V,\mathbb{Z})$,  continuous in the ${\bf F}$-metric, 
$\Phi'(0)=-((\partial \Omega^{(1)})\setminus \Sigma)$, $\Phi'(1)=\Sigma'$, $\Phi'$ a sweepout of $\overline{\Omega^{(1)}}\setminus V$
  and such that
 $\sup_{t\in [0,1]}{\bf M}_g(\Phi'(t))\leq {\bf M}_g(\Sigma)-\eta$.   Similarly, there is a map $\Phi'':[0,1]\rightarrow \tilde{\mathcal{Z}}_n(\overline{\Omega^{(2)}}\setminus V,\mathbb{Z})$, continuous in the ${\bf F}$-metric, 
$\Phi''(0)=(\partial \Omega^{(2)}\setminus \Sigma)$, $\Phi''(1)=\Sigma''$,
$\Phi''$  a sweepout of $-\overline{\Omega^{(2)}}\setminus V$ and such that
 $\sup_{t\in [0,1]}{\bf M}_g(\Phi''(t))\leq {\bf M}_g(\Sigma)-\eta$.  By concatenating $\Phi'$, $\Phi''$, and $\{\Sigma(t)\}_{t\in [-1,1]}$, after reparametrizing, we obtain a map $\tilde{\Phi}\in \Gamma^{(1)}(\Omega)$ such that
 $$
 W^{(1)}(\Omega,g)\leq \sup_{t\in [0,1]}{\bf M}_g(\tilde{\Phi}(t))\leq {\bf M}_g(\Sigma).
 $$
 This proves the proposition.
\end{proof}

\begin{prop}\label{no.stable.one-sided.case}
Let $(\Omega,g)$  be as in Proposition \ref{index.one.boundary}. Suppose that  $\Sigma\subset {\rm int}(\Omega)$ is a closed, one-sided, embedded, minimal hypersurface with unstable oriented two-cover,  and  that there is no smooth, two-sided, closed, stable minimal hypersurface $\tilde{\Sigma}$ in ${\rm int}(\Omega)$, disjoint from $\Sigma$, with ${\rm vol}_g(\tilde{\Sigma})<2\, {\rm vol}_g(\Sigma)$. Then  
$W_g(\Omega)< 2\, {\bf M}_g(\Sigma),$ and $W^{(1)}(\Omega,g)< 2\, {\bf M}_g(\Sigma).$
\end{prop}

\begin{proof}
There is a neighborhood $V$ of $\Sigma$ such that the boundary of $\Omega\setminus V$ has a connected, strictly mean-convex component $\Sigma'=\partial V$, ${\rm vol}_g(\Sigma')<2 {\rm vol}_g(\Sigma)$, since the oriented two-cover  of $\Sigma$ is unstable.


The region $\Omega'=\Omega \setminus \overline{V}$ is connected since  $\Sigma$  is one-sided.  We can suppose that there is a smooth foliation  $\{\Sigma(t)\}_{t\in [-1,0)}$ of $\overline{V}\setminus \Sigma$, such that $\Sigma(-1)=\Sigma'$, $\Sigma(t)\rightarrow 2\Sigma$ as $t\rightarrow 0$ in varifold sense,   $f(t)={\bf M}_g(\Sigma(t))< 2\,{\bf M}_g(\Sigma)$ for  $t\neq 0$. and $f''(0)<0$. By the catenoid estimate (\cite{ketover-marques-neves}), the foliation $\{\Sigma(t)\}_{t\in [-1,0)}$ can be modified into a sweepout $\{\Sigma'(t)\}_{t\in [-1,0]}$ of $\overline{V}$, continuous in the ${\bf F}$-metric, such that  $\Sigma'(t)\subset \overline{V}$, $\Sigma'(-1)=\Sigma'$, $\Sigma'(0)=0$, and ${\bf M}_g(\Sigma'(t))\leq 2\,{\bf M}_g(\Sigma)-\eta$ for $t\in [-1,0]$, where $\eta>0$.  

Consider a map  $\Phi:[0,1]\rightarrow \tilde{\mathcal{Z}}_n(\Omega\setminus V,\mathbb{Z}_2)$, continuous in the ${\bf F}$-metric, 
$\Phi(0)=\partial \Omega$, $\Phi(1)=\Sigma'$,
and so that $\Phi$ is a sweepout of $\Omega\setminus V$.
Denote by $[\Phi]$ the class of such maps, and let
$$
W([\Phi])=\inf_{\Phi'\in [\Phi]} \sup_{t\in [0,1]}{\bf M}_g(\Phi'(t)).
$$
Then $W([\Phi])>{\bf M}_g(\partial \Omega)$, since $\partial \Omega$ is strictly stable. 
If $W([\Phi])>{\bf M}_g(\Sigma')$, then we can apply min-max theory to $[\Phi]$ as $\partial(\Omega\setminus V)$ is mean-convex. 
This implies a contradiction as in the proof of Proposition \ref{no.stable.two-sided.case}, since there is no two-sided, closed, stable minimal hypersurface $\tilde{\Sigma}$ in ${\rm int}(\Omega)$ with ${\rm vol}_g(\tilde{\Sigma})<{\rm vol}_g(\Sigma')<2\, {\rm vol}_g(\Sigma)$. Hence $W([\Phi])={\bf M}_g(\Sigma')$.



It follows by the definition of $W([\Phi])$ that we can find  a map $\Phi':[0,1]\rightarrow \tilde{\mathcal{Z}}_n(\Omega\setminus V,\mathbb{Z}_2)$,  continuous in the ${\bf F}$-metric, $\Phi'(0)=\partial \Omega$, $\Phi'(1)=\Sigma'$,
 $\Phi'$ a sweepout of $\Omega\setminus V$
  and such that
 $\sup_{t\in [0,1]}{\bf M}_g(\Phi'(t))\leq 2\, {\bf M}_g(\Sigma)-\eta/2$.  By concatenating $\Phi'$  and $\{\Sigma'(t)\}_{t\in [-1,0]}$, after reparametrizing, we obtain a map $\tilde{\Phi}$
 such that
 $$
 W_g(\Omega)\leq \sup_{t\in [0,1]}{\bf M}_g(\tilde{\Phi}(t))<2\,{\bf M}_g(\Sigma).
 $$
 The proof of $W^{(1)}(\Omega,g)< 2\, {\bf M}_g(\Sigma)$ is analogous. 
\end{proof}

\begin{prop}\label{stable.index}
Let $(\Omega,g)$ be a compact Riemannian manifold as in Proposition \ref{index.one.boundary}. Suppose that $\Sigma'$ is a  smooth, two-sided, closed, connected, embedded stable  minimal hypersurface in $\Omega$. Then 
  there is a smooth, connected, embedded, closed minimal hypersurface $\Sigma$ in $\Omega$, disjoint from $\partial \Omega$ and $\Sigma'$, two-sided, with Morse index   one, such that ${\rm vol}_g(\Sigma')<{\rm vol}_g(\Sigma)$.
\end{prop}

\begin{proof}
Let $\Lambda$ be the set of    smooth, closed, embedded, connected, two-sided, strictly stable minimal hypersurfaces  $\tilde{\Sigma}$,  ${\rm vol}_g(\tilde{\Sigma})\leq c$,   in the interior ${\rm int}(\Omega)\subset \Omega$, for $c>0$. The set $\Lambda$ of such hypersurfaces  is finite.  To prove this we suppose by contradiction that there is a sequence $\{\tilde{\Sigma}_i\in \Lambda\}$ such that
$\tilde{\Sigma}_i\neq \tilde{\Sigma}_j$ for any $i\neq j$.  By the curvature estimates of  \cite{schoen-simon},  there is a sequence $\{j\}\subset \{i\}$ such that $\tilde{\Sigma}_j$ converges smoothly locally graphically  to a closed, connected, embedded minimal hypersurface $\tilde{\Sigma}$ with  multiplicity. If $\tilde{\Sigma}$  is two-sided, 
 the convergence is smooth with multiplicity one since $\tilde{\Sigma}$ is not stable degenerate. Hence $\tilde{\Sigma}$ is stable,  and by nondegeneracy  $\tilde{\Sigma}_j=\tilde{\Sigma}$ for sufficiently large $j$. This is a contradiction. In the case $\tilde{\Sigma}$ is one-sided, the convergence is with multiplicity two. By lifting to the two-sided double cover of $\tilde{\Sigma}$, we get that the two-cover is  stable degenerate which is a contradiction. This proves that the set $\Lambda$ is finite.
 
Let  $c={\rm vol}_g(\Sigma')$.   There is  $\tilde{\Lambda} \subset \Lambda$, possibly empty, of mutually disjoint hypersurfaces that is not contained in any other set with this property.  We can define $\tilde{\Lambda}$ so that $\Sigma'\in \tilde{\Lambda}$ in case $\Sigma'$ is not one of the components of $\partial \Omega$.
Let $\Omega_1,\dots, \Omega_p$ be the connected components of the complement of $\bigcup_{\Sigma\in \tilde{\Lambda}} \Sigma$ in $\Omega$. It follows that, for any $1\leq j\leq p$,  there
is no  stable, two-sided,  closed minimal hypersurface  $\tilde{\Sigma}$ in ${\rm int}(\Omega_j)$, with ${\rm vol}_g(\tilde{\Sigma})\leq {\rm vol}_g(\Sigma').$

Let $\tilde{\Omega}_j$ be the metric closure of $\Omega_j$. Then there is $1\leq j\leq p$ such that  $\Sigma'\subset  \partial \tilde{\Omega}_j$. By Proposition \ref{index.one.boundary}, there is  a smooth, connected, embedded, closed minimal hypersurface $\Sigma$ in ${\rm int}(\tilde{\Omega}_j)$,  two-sided, with Morse index   one, such that ${\rm vol}_g(\Sigma)\leq W_g(\tilde{\Omega}_j)$.

Suppose that there is a smooth, two-sided, closed, stable minimal hypersurface $\tilde{\Sigma}$ in ${\rm int}(\tilde{\Omega}_j)$ with ${\rm vol}_g(\tilde{\Sigma})<{\rm vol}_g(\Sigma)$. Then ${\rm vol}_g(\Sigma')< {\rm vol}_g(\tilde{\Sigma})$, by definition of $\tilde{\Lambda}$ and $\Omega_j$. Hence ${\rm vol}_g(\Sigma')<{\rm vol}_g(\Sigma)$, which proves the proposition in this case. 

If there is no two-sided, stable minimal hypersurface $\tilde{\Sigma}$ in ${\rm int}(\tilde{\Omega}_j)$ with ${\rm vol}_g(\tilde{\Sigma})<{\rm vol}_g(\Sigma)$, then by Proposition \ref{no.stable.two-sided.case} it follows that 
$W^{(1)}(\tilde{\Omega}_j,g)\leq {\rm vol}_g(\Sigma)$. Since the hypersurface $\Sigma'$ is one of the components of 
$\partial \tilde{\Omega}_j$, and $\partial\tilde{\Omega}_j$ is strictly stable, it follows that ${\rm vol}_g(\Sigma')< W^{(1)}(\tilde{\Omega}_j,g)$. Therefore ${\rm vol}_g(\Sigma')<{\rm vol}_g(\Sigma)$, which finishes the proof of the proposition. 
\end{proof}

\begin{prop}\label{stable.one-sided.index}
Let $(\Omega,g)$ be a compact Riemannian manifold as in Proposition \ref{index.one.boundary}. Suppose that $\Sigma'$ is a  smooth, one-sided, closed, connected, embedded minimal hypersurface in $\Omega$ with stable oriented two-cover. Then 
  there is a smooth, connected, embedded, closed minimal hypersurface $\Sigma$ in $\Omega$, disjoint from $\partial \Omega$ and $\Sigma'$, two-sided, with Morse index   one, such that $2\,{\rm vol}_g(\Sigma')<{\rm vol}_g(\Sigma)$.
\end{prop}

\begin{proof}
This follows by applying the proof of Proposition \ref{stable.index} to the metric closure $\tilde{\Omega}'$ of $\Omega'=\Omega \setminus \Sigma'$. This is because the oriented two-cover $\tilde{\Sigma}'\subset \tilde{\Omega}'$
of $\Sigma'$ is a stable, two-sided, embedded, closed minimal hypersurface with ${\rm vol}_g(\tilde{\Sigma}')=2\, {\rm vol}_g(\Sigma')$.
\end{proof}

Let $(M^{n+1},g)$ be a compact oriented Riemannian $(n+1)$-dimensional manifold, possibly with boundary $\partial M$,  and  let $(N^{n+1},h)$
be a closed oriented Riemannian $(n+1)$-manifold, such that $3\leq (n+1)\leq 7$.

Suppose that  $h$ has no stable, degenerate, closed, two-sided, embedded   minimal hypersurfaces or one-sided, closed, embedded minimal hypersurfaces with stable, degenerate, oriented  two-cover. 

\begin{prop}\label{index.width}
There is a smooth, connected, embedded, closed minimal hypersurface $\Sigma$ in $N$,  with Morse index   one, two-sided, such that
$$
|\Sigma|_h\leq W^{(k)}(N,h).
$$
\end{prop}

\begin{proof}
Let $h_i$ be a sequence of bumpy metrics converging smoothly to $h$. By min-max theory applied to $W^{(k)}(N,h)$, using the Theorem 1.7 of \cite{marques-neves-index} and taking a limit as in the proof of Proposition \ref{index.one.boundary}, there is a closed, embedded, minimal hypersurface $\Sigma' \subset N$ that is either two-sided, of Morse index one, and $|\Sigma'|_h\leq W^{(k)}(N,h)$, or one-sided, with unstable oriented two-cover, and $2\, |\Sigma'|_h\leq W^{(k)}(N,h)$. We can suppose $\Sigma'$ is one-sided since in the two-sided case the proposition follows with $\Sigma=\Sigma'$.

Let  $c=2\, {\rm vol}_h(\Sigma')$, and define  $\Lambda$ to  be the set of    smooth, closed, embedded, connected, two-sided, strictly stable minimal hypersurfaces  $\tilde{\Sigma}\subset N$,  ${\rm vol}_h(\tilde{\Sigma})\leq c$,  disjoint from $\Sigma'$.  
It follows  by the proof of Proposition \ref{stable.index} that  the set $\Lambda$ is finite, and hence there is a  set $\tilde{\Lambda} \subset \Lambda$, possibly empty, of mutually disjoint hypersurfaces that is not contained in any other set with this property.

Let $\Omega$ be the connected component of the complement of $\cup_{\tilde{\Sigma}\in \tilde{\Lambda}}\tilde{\Sigma}$ such that  
$\Sigma'\subset {\rm int}(\Omega)$. By Proposition \ref{no.stable.one-sided.case}, $W_h(\Omega)< 2|\Sigma'|_h$. Then, by Proposition \ref{index.one.boundary}, there is a closed, two-sided, embedded, minimal hypersurface $\Sigma\subset {\rm int}(\Omega)$, with Morse index one, such that $|\Sigma|_h\leq W_h(\Omega)$. This finishes the proof of the proposition since
$W_h(\Omega)< 2|\Sigma'|_h\leq W^{(k)}(N,h)$.

\end{proof}

 Suppose that the metric $g$ also has no stable, degenerate, closed, two-sided, embedded   minimal hypersurfaces or one-sided, closed, embedded minimal hypersurfaces with stable, degenerate, oriented  two-cover.  In the case $M$ has boundary,  suppose that $\partial M$ is a strictly stable minimal hypersurface.
 
 We prove that an $n$-volume nonincreasing Lipschitz map can be used to construct connected, Morse index one, two-sided minimal hypersurfaces with volume bounds.
 
 \begin{thm}\label{connected.minimal}
Suppose that  $f:(M,g)\rightarrow (N,h)$ is an   $n$-volume nonincreasing Lipschitz map such that $f_{\#}(\Sigma')=0$ for any connected component $\Sigma'$ of $\partial M$. If ${\rm deg}(f)\neq 0$, then there is a smooth, connected, embedded, closed minimal hypersurface $\Sigma$ in $M$, disjoint from $\partial M$, with Morse index   one, two-sided, such that 
\begin{itemize}
\item[(i)] either there is a smooth, connected, embedded, closed minimal hypersurface $\Sigma'$ in $N$, with Morse index   one, two-sided, and
$$|\Sigma'|_h\leq W^{(k)}(N,h) \leq |\Sigma|_g,$$
\item[(ii)] or there are smooth, connected, embedded, two-sided, stable minimal hypersurfaces $\tilde{\Sigma}$ in $M$ and $\tilde{\Sigma}'$ in $N$, and
$$
|\tilde{\Sigma}'|_h\leq |\tilde{\Sigma}|_g \leq |\Sigma|_g.
$$
\end{itemize}
\end{thm}

\begin{proof}
Let $\Phi$ be a map in $\Gamma^{(1)}(M)$, continuous in the mass topology. Then $f_{\#}\circ \Phi$ is continuous in the mass topology and $f_{\#}\circ \Phi\in \Gamma^{(k)}(N)$, where  ${\rm deg}(f)=k$. Since ${\bf M}_h((f_{\#}\circ \Phi)(x))\leq {\bf M}_g(\Phi(x))$ for any $\Phi$ and $x\in [0,1]$, it follows that
$$
W^{(k)}(N,h)\leq W^{(1)}(M,g).
$$
We are using that maps that are continuous in the ${\bf F}$-metric can be approximated by maps continuous in the mass topology (see appendix of \cite{marques-neves-index}).

Let $c=W^{(k)}(N,h)$, and define $\Lambda$ to be the set of    smooth, closed, embedded, connected, two-sided, strictly stable minimal hypersurfaces  $\tilde{\Sigma}$,  ${\rm vol}_g(\tilde{\Sigma})\leq c$,   in the interior ${\rm int}(M)\subset M$. The set $\Lambda$ of such hypersurfaces  is finite, as in the proof of Proposition \ref{stable.index}. 
 Hence there is a  set $\tilde{\Lambda} \subset \Lambda$, possibly empty, of mutually disjoint hypersurfaces, that is not contained in any other set with this property.  
Let $\Omega_1,\dots, \Omega_p$ be the connected components of the complement of $\bigcup_{\Sigma\in \tilde{\Lambda}} \Sigma$. It follows that, for any $1\leq j\leq p$,  there
is no  stable, two-sided,  closed minimal hypersurface  $\tilde{\Sigma}$ in ${\rm int}(\Omega_j)$, with ${\rm vol}_g(\tilde{\Sigma})\leq W^{(k)}(N,h).$

If  $\tilde{\Omega}_j$ is  the metric closure of $\Omega_j$, then $\tilde{\Omega}_j$ is an $(n+1)$-dimensional compact manifold, possibly with boundary, any  connected component of $\partial \tilde{\Omega}_j$ can be identified with one of the components of $\partial M$ or with a hypersurface in $\tilde{\Lambda}$, 
  and ${\rm int}(\tilde{\Omega}_j)=\Omega_j$.  By Proposition \ref{index.one.boundary}, since $\partial \tilde{\Omega}_j$ is strictly stable, there is a smooth, connected, embedded, closed minimal hypersurface $\Sigma_j' \subset {\rm int}(\tilde{\Omega}_j)$, two-sided, with Morse index one, and ${\rm vol}_g(\Sigma_j')\leq W_g(\tilde{\Omega}_j)$.

 Suppose that there is $\eta>0$ such that, for any $1\leq j\leq p$,  ${\rm vol}_g(\Sigma_j')\leq W^{(k)}(N,h)-2\eta$.
  Since there is no closed, two-sided, stable minimal hypersurface $\tilde{\Sigma}\subset {\rm int}(\tilde{\Omega}_j)$ with 
  ${\rm vol}_g(\tilde{\Sigma})< {\rm vol}_g(\Sigma_j')$, by Proposition \ref{no.stable.two-sided.case} 
   there is 
a map $\tilde{\Psi}_j:[-1,1]\rightarrow \tilde{\mathcal{Z}}_n(\tilde{\Omega}_j,\mathbb{Z})$, continuous in the ${\bf F}$-metric, such that
$\tilde{\Psi}_j(-1)=-\partial' \tilde{\Omega}_j$,  $\tilde{\Psi}_j(1)=\partial'' \tilde{\Omega}_j$, the map $\tilde{\Psi}_j$ is a sweepout of $\tilde{\Omega}_j$, and 
$$
{\bf M}_g(\tilde{\Psi}_j(t))\leq {\rm vol}_g(\Sigma_j') 
$$ 
for any $t\in [-1,1]$. Hence $\sup_{t\in [-1,1]}{\bf M}_g(\tilde{\Psi}_j(t)) \leq W^{(k)}(N,h)-2\eta$.  The hypersurfaces $\partial'\tilde{\Omega}_j$, $\partial''\tilde{\Omega}_j$, are unions of components of $\partial \tilde{\Omega}_j$ such that $\partial'\tilde{\Omega}_j$ and  $\partial''\tilde{\Omega}_j$ are disjoint, with 
$\partial \tilde{\Omega}_j=\partial'\tilde{\Omega}_j \cup \partial''\tilde{\Omega}_j$. The orientation of $\partial \tilde{\Omega}_j$ is induced by that of $\tilde{\Omega}_j$.
Notice that $|\tilde{\Sigma}|_g\leq |\Sigma_j'|_g$ for any component $\tilde{\Sigma}$ of $\partial \tilde{\Omega}_j$.

It follows by the appendix of \cite{marques-neves-index} that the map $\tilde{\Psi}_j$ can be approximated by maps that are continuous in the mass topology. Hence  there is 
a map $\tilde{\Psi}'_j:[-1,1]\rightarrow \tilde{\mathcal{Z}}_n(\tilde{\Omega}_j,\mathbb{Z})$, continuous in the mass topology, such that
$\tilde{\Psi}'_j(-1)=-\partial' \tilde{\Omega}_j$,  $\tilde{\Psi}'_j(1)=\partial'' \tilde{\Omega}_j$, the map $\tilde{\Psi}'_j$ is a sweepout of $\tilde{\Omega}_j$, and 
$$
{\bf M}_g(\tilde{\Psi}'_j(t))\leq {\rm vol}_g(\Sigma_j') +\eta
$$ 
for any $t\in [-1,1]$.

Let $\delta>0$ be such that $\delta \cdot  (\#\tilde{\Lambda})\leq \eta$. Let $\tilde{\Sigma}\in \tilde{\Lambda}$ with an orientation. We   are going to use a constrained 
minimization in the homology class of $f_{\#}(\tilde{\Sigma})\in \tilde{\mathcal{Z}}_n(N,\mathbb{Z}).$ Denote by  $\Gamma_{\tilde{\Sigma}}$
 the set of maps continuous in the ${\bf F}$-metric $\Phi:[0,1]\rightarrow \tilde{\mathcal{Z}}_n(N,\mathbb{Z})$, such that $\Phi(0)=-f_{\#}(\tilde{\Sigma})$ and ${\bf M}_h(\Phi(t))\leq {\bf M}_h(f_{\#}(\tilde{\Sigma}))+\delta$ for any $t\in [0,1]$. Let
$$
\lambda_{\tilde{\Sigma}}= \inf_{\Phi\in \Gamma_{\tilde{\Sigma}}}{\bf M}_h(\Phi(1)).
$$
Then $\lambda_{\tilde{\Sigma}}\leq {\bf M}_h(f_{\#}(\tilde{\Sigma}))$.

Let $\{\Phi_i\in \Gamma_{\tilde{\Sigma}}\}_i$ be a sequence such that ${\bf M}_h(\Phi_i(1))\rightarrow \lambda_{\tilde{\Sigma}}$. By compactness of the space of integral currents with mass bounds and lower semicontinuity of mass, we can suppose by passing to a subsequence that there is $T\in \tilde{\mathcal{Z}}_n(N,\mathbb{Z})$, ${\bf M}_h(T)\leq \lambda_{\tilde{\Sigma}}$, such that 
$\mathcal{F}(\Phi_i(1),T)\rightarrow 0$. The Proposition A.2 of \cite{marques-neves-index} implies by concatenation that there is $\Phi\in \Gamma_{\tilde{\Sigma}}$ such that $\Phi(1)=T$, and hence ${\bf M}_h(T)=\lambda_{\tilde{\Sigma}}$. 

If $\lambda_{\tilde{\Sigma}}>0$, the support of $T$ is a smooth, embedded, closed minimal hypersurface of $(N,h)$, since for any
$p\in {\rm support}(T)$, there is $r>0$ such that $T$ is area-minimizing in $B_r(p)$. By construction, in this case ${\rm support}(T)$ is two-sided, stable and $|T|_h\leq {\bf M}_h(f_{\#}(\tilde{\Sigma}))$.  Since the map $f$ is $n$-volume nonincreasing, 
${\bf M}_h(f_{\#}(\tilde{\Sigma}))\leq {\bf M}_g(\tilde{\Sigma})$. In this case the theorem is proved using case (ii) with $\tilde{\Sigma}'$ a connected component of ${\rm spt}(T)$ and 
$\Sigma=\Sigma_j'$, where $1\leq j\leq p$ is such that $\tilde{\Sigma}\subset \partial \tilde{\Omega}_j$.

Hence we can suppose that there is $\Phi_{\tilde{\Sigma}}\in \Gamma_{\tilde{\Sigma}}$ such that $\Phi_{\tilde{\Sigma}}(1)=0$. The map $\Phi_{\tilde{\Sigma}}$ is a sweepout of a chain $R_{\tilde{\Sigma}}\in  {\bf I}_{n+1}(N,\mathbb{Z})$ such that
$\partial R_{\tilde{\Sigma}}=f_{\#}(\tilde{\Sigma})$. If $-\tilde{\Sigma}$ denotes $\tilde{\Sigma}$ with the opposite orientation,
we define $\Phi_{-\tilde{\Sigma}}(t)=-\Phi_{\tilde{\Sigma}}(t)$ and $R_{-\tilde{\Sigma}}=-R_{\tilde{\Sigma}}$.

If $\tilde{\Sigma}$ is a component of $\partial M$, then  $f_{\#}(\tilde{\Sigma})=0$. In this case, we can define the map
$\Phi_{\tilde{\Sigma}}:[0,1]\rightarrow \tilde{\mathcal{Z}}_n(N,\mathbb{Z})$, $\Phi_{\tilde{\Sigma}}(t)=0$ for $t\in [0,1]$, and the chain $R_{\tilde{\Sigma}}=0$ in ${\bf I}_{n+1}(N,\mathbb{Z})$.

Let $1\leq j\leq p$.   The map $f_{|\Omega_j}$ induces a map $f:\tilde{\Omega}_j\rightarrow N$ that is Lipschitz and $n$-volume nonincreasing.  Hence the  map $\tilde{\Phi}_j=f_{\#} \circ \tilde{\Psi}'_j$ is such that  $\tilde{\Phi}_j:[-1,1]\rightarrow \tilde{\mathcal{Z}}_n(N,\mathbb{Z})$ is continuous in the mass topology, 
$\tilde{\Phi}_j(-1)=-f_{\#}(\partial' \tilde{\Omega}_j)$,  $\tilde{\Phi}_j(1)=f_{\#}(\partial'' \tilde{\Omega}_j)$, the map $\tilde{\Phi}_j$ is a sweepout of $f_{\#}(\tilde{\Omega}_j)$, and 
$$
{\bf M}_h(\tilde{\Phi}_j(t))\leq {\bf M}_g(\tilde{\Psi}'_j(t))\leq  {\rm vol}_g(\Sigma_j') +\eta
$$ 
for any $t\in [-1,1]$. Hence $\sup_{t\in [-1,1]}{\bf M}_h(\tilde{\Phi}_j(t)) \leq W^{(k)}(N,h)-\eta$.  The chain $f_{\#}(\partial' \tilde{\Omega}_j)$ is a boundary, since $f_{\#}(\partial' \tilde{\Omega}_j)=\partial(\sum_{\tilde{\Sigma}\subset \partial' \tilde{\Omega}_j} R_{\tilde{\Sigma}}).$ Therefore $\tilde{\Phi}_j(t)\in \mathcal{Z}_n(N,\mathbb{Z})$ for any $t\in [0,1]$.

We orient the components $\tilde{\Sigma}$ of $\partial \tilde{\Omega}_j$ with the orientation induced by that of $\tilde{\Omega}_j$. Hence
$\partial'\tilde{\Omega}_j=\sum_{\tilde{\Sigma}\subset \partial' \tilde{\Omega}_j} \tilde{\Sigma}$ and $\partial''\tilde{\Omega}_j=\sum_{\tilde{\Sigma}\subset \partial'' \tilde{\Omega}_j} \tilde{\Sigma}$.
By concatenating the maps  $\sum_{\tilde{\Sigma}\subset \partial' \tilde{\Omega}_j}  \Phi_{\tilde{\Sigma}}$, $\tilde{\Phi}_j$, and 
$-\sum_{\tilde{\Sigma}\subset \partial'' \tilde{\Omega}_j}  \Phi_{\tilde{\Sigma}}$, after reparametrizing, we obtain a map $\Phi_j:[0,1]\rightarrow \mathcal{Z}_n(N,\mathbb{Z})$, continuous in the {\bf F}-metric, such that $\Phi_j(0)=0$, $\Phi_j(1)=0$, 
$\Phi_j$ is a sweepout of 
\begin{eqnarray*}
\Omega(j)&=&-\sum_{\tilde{\Sigma}\subset \partial' \tilde{\Omega}_j} R_{\tilde{\Sigma}}+f_{\#}(\tilde{\Omega}_j)-\sum_{\tilde{\Sigma}\subset \partial'' \tilde{\Omega}_j} R_{\tilde{\Sigma}}
=f_{\#}(\tilde{\Omega}_j)-\sum_{\tilde{\Sigma}\subset \partial \tilde{\Omega}_j} R_{\tilde{\Sigma}},
\end{eqnarray*}
and
$$
\sup_{t\in [0,1]} {\bf M}(\Phi_j(t))\leq {\rm vol}_g(\Sigma_j')+\eta.
$$
We are using that 
\begin{eqnarray*}
\sum_{\tilde{\Sigma}\subset \partial'\tilde{\Omega}_j}  ({\bf M}_h(f_{\#}(\tilde{\Sigma}))+\delta)&\leq& \sum_{\tilde{\Sigma}\subset \partial'\tilde{\Omega}_j}  ({\bf M}_g(\tilde{\Sigma})+\delta)\leq {\bf M}_g(\partial'\tilde{\Omega}_j)+\eta\\
&=& {\bf M}_g(\tilde{\Psi}_j(-1))+\eta\leq {\rm vol}_g(\Sigma_j')+\eta,
\end{eqnarray*}
and that
\begin{eqnarray*}
\sum_{\tilde{\Sigma}\subset \partial''\tilde{\Omega}_j}  ({\bf M}_h(f_{\#}(\tilde{\Sigma}))+\delta)&\leq& \sum_{\tilde{\Sigma}\subset \partial''\tilde{\Omega}_j}  ({\bf M}_g(\tilde{\Sigma})+\delta)\leq {\bf M}_g(\partial''\tilde{\Omega}_j)+\eta\\
&=& {\bf M}_g(\tilde{\Psi}_j(1))+\eta\leq {\rm vol}_g(\Sigma_j')+\eta.
\end{eqnarray*}

By concatenating the maps $\Phi_1, \dots, \Phi_p$, we obtain a map $\Phi:[0,1]\rightarrow \mathcal{Z}_n(N,\mathbb{Z})$, continuous in the {\bf F}-metric, such that $\Phi(0)=0$, $\Phi(1)=0$, 
$\Phi$ is a sweepout of 
$$
\sum_{j=1}^p \Omega(j)=\sum_{j=1}^p \big(f_{\#}(\tilde{\Omega}_j)-\sum_{\tilde{\Sigma}\subset \partial \tilde{\Omega}_j} R_{\tilde{\Sigma}}\big)=f_{\#}(\sum_{j=1}^p \tilde{\Omega}_j),
$$
and 
$$
\sup_{t\in [0,1]} {\bf M}(\Phi(t))\leq W^{(k)}(N,h)-\eta.
$$
This uses that $\sum_{j=1}^p\sum_{\tilde{\Sigma}\subset \partial \tilde{\Omega}_j} R_{\tilde{\Sigma}}=0$,  since $R_{\tilde{\Sigma}}=0$ if $\tilde{\Sigma}$ is a component of $\partial M$, and since the orientations induced on $\tilde{\Sigma}\in \tilde{\Lambda}$ by the adjacent domains are opposite. 

It follows that $\Phi$ is a sweepout of $f_{\#}(M)=k\cdot N$, since $\sum_{j=1}^p \tilde{\Omega}_j=M$. Hence $\Phi\in \Gamma^{(k)}(N)$, which is a contradiction since $\sup_{t\in [0,1]} {\bf M}(\Phi(t))\leq W^{(k)}(N,h)-\eta$.   

Hence  there is $1\leq j\leq p$ such that  ${\rm vol}_g(\Sigma_j')\geq W^{(k)}(N,h)$.  We can define $\Sigma=\Sigma_j'$ and the theorem is proved using case (i) by  Proposition \ref{index.width}.

\end{proof}

Theorem \ref{homology} can be used in the case of  $n$-volume nonincreasing Lipschitz maps  of Riemannian spheres.

 \begin{thm}\label{homology}
 Suppose  $M$ is a closed manifold,  $H_n(M,\mathbb{Z})=0$, and   $f:(M,g)\rightarrow (N,h)$ is an   $n$-volume nonincreasing Lipschitz map. If ${\rm deg}(f)\neq 0$, 
there is a smooth, connected, embedded, closed minimal hypersurface $\Sigma$ in $M$,  with Morse index   one, two-sided, and there is a smooth, connected, embedded, closed minimal hypersurface $\Sigma'$ in $N$, with Morse index   one, two-sided, such that
$$
|\Sigma'|_h\leq  |\Sigma|_g.
$$
\end{thm}

\begin{proof}

Let $\Lambda$, $\tilde{\Lambda}$, $k$ and $\Omega_1,\dots,\Omega_p$ be as in Theorem \ref{connected.minimal}. For any $1\leq j\leq p$,  there
is no  stable, two-sided,  closed minimal hypersurface  $\tilde{\Sigma}$ in ${\rm int}(\Omega_j)$, such that ${\rm vol}_g(\tilde{\Sigma})\leq W^{(k)}(N,h).$ As before, 
$W^{(k)}(N,h)\leq W^{(1)}(M,g).$ Since  $H_n(M,\mathbb{Z})=0,$
a closed, smooth, embedded, two-sided hypersurface  $\Sigma\subset M$ separates $M$.

If  $\tilde{\Omega}_j$ is  the metric closure of $\Omega_j$, then $\tilde{\Omega}_j$ is an $(n+1)$-dimensional compact manifold, possibly with boundary, any  connected component of $\partial \tilde{\Omega}_j$ can be identified with  a hypersurface in $\tilde{\Lambda}$, 
  and ${\rm int}(\tilde{\Omega}_j)=\Omega_j$. As in the proof of Theorem \ref{connected.minimal},  there is a smooth, connected, embedded, closed minimal hypersurface $\Sigma_j \subset {\rm int}(\tilde{\Omega}_j)$, two-sided, with Morse index one.
 
Suppose that there is $1\leq j\leq p$ such that ${\rm vol}_g(\Sigma_j)\geq W^{(k)}(N,h)$. By Proposition \ref{index.width}, there is a smooth, closed, two-sided, embedded minimal surface $\Sigma'\subset (N,h)$, of Morse index one,  such that ${\rm vol}_h(\Sigma')\leq W^{(k)}(N,h)$. Hence ${\rm vol}_h(\Sigma')\leq {\rm vol}_g(\Sigma_j)$, which proves the theorem in this case.
  
Hence we can  suppose that there is $\rho>0$ such that, for any $1\leq j\leq p$,  ${\rm vol}_g(\Sigma_j)\leq W^{(k)}(N,h)-2\rho$. Suppose that if $0<\xi<\rho$, 
there is a  closed, two-sided, embedded minimal hypersurface $\Sigma_\xi'\subset (N,h)$, of Morse index one, such that
${\rm vol}_h(\Sigma'_\xi)\leq  {\rm vol}_g(\Sigma_j) +2\xi$ for $j=j(\xi)$. Then there is a sequence $\xi_i\rightarrow 0$ and $1\leq j\leq p$ such that
$\Sigma'_{\xi_i}$ converges smoothly to a  closed, two-sided, embedded minimal hypersurface $\Sigma'\subset (N,h)$, of Morse index one,
with ${\rm vol}_h(\Sigma')\leq  {\rm vol}_g(\Sigma_j)$. This proves the theorem. 

Therefore there is $0<\xi<\rho$ such that there is no  closed, two-sided, embedded minimal hypersurface $\Sigma'\subset (N,h)$, of Morse index one, and $1\leq j\leq p$, such that
${\rm vol}_h(\Sigma')\leq  {\rm vol}_g(\Sigma_j) +2\xi$.

By Proposition \ref{no.stable.two-sided.case}, there is 
a map $\tilde{\Psi}_j:[-1,1]\rightarrow \tilde{\mathcal{Z}}_n(\tilde{\Omega}_j,\mathbb{Z})$, continuous in the mass topology, such that
$\tilde{\Psi}_j(-1)=-\partial' \tilde{\Omega}_j$,  $\tilde{\Psi}_j(1)=\partial'' \tilde{\Omega}_j$, the map $\tilde{\Psi}_j$ is a sweepout of $\tilde{\Omega}_j$, and 
$$
{\bf M}_g(\tilde{\Psi}_j(t))\leq {\rm vol}_g(\Sigma_j) +\xi \leq W^{(k)}(N,h)-\rho
$$ 
for any $t\in [-1,1]$.  The hypersurfaces $\partial'\tilde{\Omega}_j$, $\partial''\tilde{\Omega}_j$, are unions of components of $\partial \tilde{\Omega}_j$ such that $\partial'\tilde{\Omega}_j$ and  $\partial''\tilde{\Omega}_j$ are disjoint, with 
$\partial \tilde{\Omega}_j=\partial'\tilde{\Omega}_j \cup \partial''\tilde{\Omega}_j$. 

Let $1\leq r\leq p$ be such that $\Omega=\overline{\Omega}_1\cup \cdots \cup \overline{\Omega}_r$, perhaps changing the indices, satisfies that 
$\Omega$ is a connected manifold possibly with boundary and if $\Sigma$ is a component of $\partial\Omega$, and $R_{\Sigma,\Omega}$ is the domain in $M$ that is disjoint from ${\rm int}(\Omega)$ and $\partial R_{\Sigma,\Omega}=\Sigma$, then if $\delta>0$  there is a map $\Phi_\Sigma:[0,1]\rightarrow \mathcal{Z}_n(N,\mathbb{Z})$, continuous in the ${\bf F}$-metric, such that $\Phi_\Sigma(0)=-f_{\#}(\Sigma)$, $\Phi_\Sigma(1)=0$,  $\Phi_\Sigma$ is a sweepout of 
$f_{\#}(R_{\Sigma,\Omega})$, and 
$$
\sup_{t\in [0,1]}{\bf M}_h(\Phi_\Sigma(t))\leq {\bf M}_h(f_{\#}(\Sigma)) + \delta.
$$
Let $\Gamma$ be the set of such $\Omega$.  Notice that $\overline{\Omega}_1\cup \cdots \cup \overline{\Omega}_p\in \Gamma.$ The surface $\Sigma$ has the opposite orientation to that induced by $\Omega$.

Suppose $r\geq 2$.  Since closed, two-sided, embedded hypersurfaces separate $M$, there is  $1\leq r'\leq r$ such that the components
$\tilde{\Sigma}_1,\dots, \tilde{\Sigma}_\eta$ of $\partial\Omega_{r'}$ can be defined so that $\tilde{\Sigma}_1\subset {\rm int}(\Omega)$ and $\tilde{\Sigma}_i\subset \partial\Omega$ for $2\leq i\leq \eta$.  We can suppose that $r'=r$ by changing the indices.

Let $\Omega'=\Omega_1\cup \cdots \cup \Omega_{r-1}$. Then $\Omega'$ is a connected manifold with boundary. If $\Sigma$ is a component
of $\partial\Omega'$, then either $\Sigma$ is a component of $\partial \Omega$ or $\Sigma=\tilde{\Sigma}_1$.
In the case $\Sigma$ is a component of $\partial \Omega$, then $R_{\Sigma,\Omega}=R_{\Sigma,\Omega'}$.

Suppose $\Sigma=\tilde{\Sigma}_1$. Then $R_{\Sigma,\Omega'}=\Omega_r \cup R_{\tilde{\Sigma}_2,\Omega} \cup \cdots \cup R_{\tilde{\Sigma}_\eta,\Omega}$. Let $\tilde{\Psi}_r$ be as before, and suppose $\tilde{\Sigma}_1$ is a component of $\partial'\tilde{\Omega}_r$. Define $\tilde{\Phi}_r=f_{\#}\circ \tilde{\Psi}_r$. Since the map  $f$ is Lipschitz and $n$-volume nonincreasing,  $\tilde{\Phi}_r$ is continuous in the mass metric, $\tilde{\Phi}_r$ is a sweepout of $f_{\#}(\Omega_r)$, and 
$$
\sup_{t\in [-1,1]} {\bf M}_h(\tilde{\Phi}_r(t)) \leq \sup_{t\in [-1,1]} {\bf M}_g(\tilde{\Psi}_r(t)) \leq {\rm vol}_g(\Sigma_r) +\xi  \leq W^{(k)}(N,h)-\rho.
$$ 

Let $\tilde{\Phi}''_r(t)=\sum_{\Sigma''\subset \partial''\tilde{\Omega}_r}\Phi_{\Sigma''}(t).$ Then $\tilde{\Phi}''_r$ is a sweepout of 
$$\sum_{\Sigma''\subset \partial''\tilde{\Omega}_r} f_{\#}(R_{\Sigma'',\Omega}),
$$
and 
$$
{\bf M}_h(\tilde{\Phi}''_r(t))\leq  \sum_{\Sigma''\subset \partial''\tilde{\Omega}_r}  {\bf M}_h(\Phi_{\Sigma''}(t))\leq \sum_{\Sigma''\subset \partial''\tilde{\Omega}_r} {\bf M}_h(f_{\#}(\Sigma''))  +\eta \delta.
$$
for any $t\in [0,1]$. Hence, since $f$ is $n$-volume nonincreasing,
\begin{eqnarray*}
{\bf M}_h(\tilde{\Phi}''_r(t))&\leq&\sum_{\Sigma''\subset \partial''\tilde{\Omega}_r} {\bf M}_g(\Sigma'')  +\eta \delta={\bf M}_g(\partial''\tilde{\Omega}_r)+\eta\delta\\
&=& {\bf M}_g(\tilde{\Psi}_r(1))+\eta\delta \\
&\leq&  {\rm vol}_g(\Sigma_r) +\xi +\eta\delta\leq W^{(k)}(N,h)-\rho+\eta\delta
\end{eqnarray*}
for any $t\in [0,1]$.

Define $\tilde{\Phi}'_r(t)=-f_{\#}(\tilde{\Sigma}_1)-\sum_{\Sigma'\subset \partial'\tilde{\Omega}_r, \Sigma'\neq \tilde{\Sigma}_1}\Phi_{\Sigma'}(t).$ 
Then $\tilde{\Phi}'_r$ is a sweepout of 
$$-\sum_{\Sigma'\subset \partial'\tilde{\Omega}_r, \Sigma'\neq \tilde{\Sigma}_1} f_{\#}(R_{\Sigma',\Omega}),
$$
and 
\begin{eqnarray*}
{\bf M}_h(\tilde{\Phi}'_r(t))&\leq& {\bf M}_h(f_{\#}(\tilde{\Sigma}_1))+  \sum_{\Sigma'\subset \partial'\tilde{\Omega}_r, \Sigma'\neq \tilde{\Sigma}_1}  {\bf M}_h(\Phi_{\Sigma'}(t))\\
&\leq& {\bf M}_h(f_{\#}(\tilde{\Sigma}_1))+\sum_{\Sigma'\subset \partial'\tilde{\Omega}_r, \Sigma'\neq \tilde{\Sigma}_1} {\bf M}_h(f_{\#}(\Sigma'))  +\eta \delta
\end{eqnarray*}
for any $t\in [0,1]$. Since $f$ is $n$-volume nonincreasing,
\begin{eqnarray*}
{\bf M}_h(\tilde{\Phi}'_r(t))&\leq& {\bf M}_g(\tilde{\Sigma}_1)+\sum_{\Sigma'\subset \partial'\tilde{\Omega}_r, \Sigma'\neq \tilde{\Sigma}_1}{\bf M}_g(\Sigma')  +\eta \delta={\bf M}_g(\partial'\tilde{\Omega}_r)+\eta\delta\\
&=& {\bf M}_g(\tilde{\Psi}_r(-1))+\eta\delta \\
&\leq& {\rm vol}_g(\Sigma_r) +\xi +\eta\delta\leq W^{(k)}(N,h)-\rho+\eta\delta
\end{eqnarray*}
for any $t\in [0,1]$.

Notice that $\tilde{\Phi}_r'(0)=-f_{\#}(\partial'\tilde{\Omega}_r)$, and $\tilde{\Phi}_r'(1)=-f_{\#}(\tilde{\Sigma}_1).$ And $\tilde{\Phi}_r''(0)=f_{\#}(\partial''\tilde{\Omega}_r)$, and $\tilde{\Phi}_r''(1)=0.$ By concatenating the maps $\tilde{\Phi}_r'$, $\tilde{\Phi}_r$, and $\tilde{\Phi}_r''$, after reparametrizing, we obtain a map  $\tilde{\Phi}_{\tilde{\Sigma}_1}:[0,1]\rightarrow \mathcal{Z}_n(N,\mathbb{Z})$, continuous in the ${\bf F}$-metric, such that $\tilde{\Phi}_{\tilde{\Sigma}_1}(0)=-f_{\#}(\tilde{\Sigma}_1)$, $\tilde{\Phi}_{\tilde{\Sigma}_1}(1)=0$,  $\tilde{\Phi}_{\tilde{\Sigma}_1}$ is a sweepout of 
\begin{eqnarray*}
&&\sum_{\Sigma'\subset \partial'\tilde{\Omega}_r, \Sigma'\neq \tilde{\Sigma}_1} f_{\#}(R_{\Sigma',\Omega})+f_{\#}(\Omega_r)+\sum_{\Sigma''\subset \partial''\tilde{\Omega}_r} f_{\#}(R_{\Sigma'',\Omega})\\
&=&f_{\#}(\Omega_r+\sum_{i=2}^\eta R_{\tilde{\Sigma}_i,\Omega}))=f_{\#}(R_{\tilde{\Sigma}_1,\Omega'}),
\end{eqnarray*}
and 
$$
\sup_{t\in [0,1]}{\bf M}_h(\tilde{\Phi}_{\tilde{\Sigma}_1}(t))\leq  {\rm vol}_g(\Sigma_r) +\xi +\eta\delta \leq W^{(k)}(N,h)-\rho+\eta\delta.
$$
Let $\delta>0$ be such that $\eta\delta\leq \xi$. Then 
$$
\sup_{t\in [0,1]}{\bf M}_h(\tilde{\Phi}_{\tilde{\Sigma}_1}(t))\leq  {\rm vol}_g(\Sigma_r) +2\xi.
$$

Let $\Pi$ be the set of maps $\Phi$ that are continuous in the ${\bf F}$-metric and that are homotopic in the flat topology to $\tilde{\Phi}_{\tilde{\Sigma}_1}$ relative to $\partial([0,1])$. Let 
$$
W(\Pi)=\inf_{\Phi\in\Pi}\sup_{t\in [0,1]}{\bf M}_h(\Phi(t)).
$$  By min-max, and using the proof of Proposition \ref{index.width}, it follows that either $W(\Pi)={\bf M}_h(f_{\#}(\tilde{\Sigma}_1))$, or 
$W(\Pi)>{\bf M}_h(f_{\#}(\tilde{\Sigma}_1))$ and
there is a  closed, two-sided, embedded minimal hypersurface $\Sigma'\subset (N,h)$, of Morse index one, such that
${\rm vol}_h(\Sigma')\leq W(\Pi)$. Notice that $\tilde{\Phi}_{\tilde{\Sigma}_1}\in \Pi$, hence $W(\Pi)\leq {\rm vol}_g(\Sigma_r) +2\xi$.

Suppose that $W(\Pi)>{\bf M}_h(f_{\#}(\tilde{\Sigma}_1))$. Then 
$$
{\rm vol}_h(\Sigma')\leq W(\Pi)\leq {\rm vol}_g(\Sigma_r) +2\xi,
$$
which is a contradiction.

Hence $W(\Pi)={\bf M}_h(f_{\#}(\tilde{\Sigma}_1))$. Then for $\delta>0$ there is a map $\Phi_{\tilde{\Sigma}_1}:[0,1]\rightarrow \mathcal{Z}_n(N,\mathbb{Z})$, continuous in the ${\bf F}$-metric, such that $\Phi_{\tilde{\Sigma}_1}(0)=-f_{\#}(\tilde{\Sigma}_1)$, $\Phi_{\tilde{\Sigma}_1}(1)=0$,  $\Phi_{\tilde{\Sigma}_1}$ is a sweepout of 
$f_{\#}(R_{\tilde{\Sigma}_1,\Omega'})$, and 
$$
\sup_{t\in [0,1]}{\bf M}_h(\Phi_{\tilde{\Sigma}_1}(t))\leq {\bf M}_h(f_{\#}(\tilde{\Sigma}_1)) + \delta.
$$
This proves that $\Omega_1\cup \cdots \cup \Omega_{r-1}\in \Gamma$.

Then, by induction, we can suppose that $r=1$. Hence $\Omega=\Omega_1$, perhaps changing the indices, satisfies that 
$\Omega$ is a connected manifold possibly with boundary and if $\Sigma$ is a component of $\partial\Omega$, and $R_{\Sigma,\Omega}$ is the domain in $M$ that is disjoint from ${\rm int}(\Omega)$ and $\partial R_{\Sigma,\Omega}=\Sigma$, then if $\delta>0$  there is a map $\Phi_\Sigma:[0,1]\rightarrow \mathcal{Z}_n(N,\mathbb{Z})$, continuous in the ${\bf F}$-metric, such that $\Phi_\Sigma(0)=-f_{\#}(\Sigma)$, $\Phi_\Sigma(1)=0$,  $\Phi_\Sigma$ is a sweepout of 
$f_{\#}(R_{\Sigma,\Omega})$, and 
$$
\sup_{t\in [0,1]}{\bf M}_h(\Phi_\Sigma(t))\leq {\bf M}_h(f_{\#}(\Sigma)) + \delta.
$$
Let $\eta$ be the number of components of $\partial\Omega_1$, and $\delta>0$ be such that $\eta\delta<\rho/2$. By doing as before, using $\tilde{\Phi}_1=f_{\#}\circ \tilde{\Psi}_1$  and the maps $\Phi_{\Sigma}$, $\Sigma\subset \partial\Omega_1$, we obtain a sweepout of $k \cdot N=f_{\#}(M)$ such that the  masses are bounded  by $W^{(k)}(N,h)-\rho/2$. This is a contradiction, which finishes the proof of the theorem.

\end{proof}

Let $\pi:M'\rightarrow M$ be a finite cover. If $\psi:\Sigma\rightarrow M'$ is an immersed closed minimal hypersurface for the metric $\pi^*(g)$, the second variation quadratic form of $\psi$ coincides with that of $\pi\circ \psi:\Sigma \rightarrow M$. It follows that $\psi$ is stable if and only if $\pi\circ \psi$ is stable, and $\psi$ is degenerate if and only if $\pi\circ \psi$ is degenerate.

Suppose that the metric $g$ on a compact manifold $M$, possibly with boundary, has no stable, degenerate, closed, two-sided, immersed   minimal hypersurfaces. 
 It follows by the previous paragraph that if $\pi: M'\rightarrow M$ is a finite cover, then $(M',\pi^*(g))$ has no stable, degenerate, closed, two-sided, embedded   minimal hypersurfaces or one-sided, closed, embedded minimal hypersurfaces with stable, degenerate, oriented  two-cover.

By definition, a compact orientable manifold $M^{n+1}$, possibly with  boundary $\partial M$,   is compactly $n$-volume enlargeable if  for any $r>0$ and $g$ a Riemannian metric  on  $M$, there is a finite cover $M'$ of $M$ that has an $n$-volume nonincreasing
Lipschitz map  $f:(M',g) \rightarrow (S_r^{n+1},\overline{g})$, $f_{\#}(\Sigma')=0$ for any component $\Sigma'$ of $\partial M'$, with nontrivial degree. We denote by $g$ the pull-back metric on $M'$.


\begin{thm}\label{index.area}
Let $(M^{n+1},g)$ be a   compactly $n$-volume enlargeable manifold possibly with boundary.  Suppose that the metric $g$ has no stable, degenerate, closed, two-sided, immersed   minimal hypersurfaces, and that $\partial M$ is a strictly stable minimal hypersurface. Then, for any $\lambda>0$, there is a connected,  smooth, two-sided, virtually embedded, closed minimal hypersurface $\psi:\Sigma\rightarrow M$, of Morse index one, $\psi(\Sigma)\subset {\rm int}(M)$,
with 
$$
{\rm vol}_g(\Sigma)\geq \lambda.
$$
\end{thm}

\begin{proof}
Let $r>0$ such that $r^n {\rm vol}_{\overline{g}}(S_1^n)=\lambda.$ Since $M$ is  compactly $n$-volume enlargeable, there is a finite cover $\pi: M'\rightarrow M$ that admits a 
Lipschitz map  $f:(M',g) \rightarrow (S_r^{n+1},\overline{g})$,  $f_{\#}(\Sigma')=0$ for any component $\Sigma'$ of $\partial M'$,  that has nontrivial degree $k$ and that is $n$-volume nonincreasing. Since $\overline{g}$ has positive Ricci curvature, the manifold $(S_r^{n+1},\overline{g})$ has no stable, closed, two-sided, embedded   minimal hypersurfaces or one-sided, closed, embedded minimal hypersurfaces with stable,  oriented  two-cover.
 Therefore we can apply  Theorem \ref{connected.minimal} to the map $f:(M',g) \rightarrow (S_r^{n+1},\overline{g})$, and hence there is a smooth, connected, embedded, closed minimal hypersurface 
$\psi':\Sigma\rightarrow M'$, with Morse index   one, two-sided, $\psi'(\Sigma)\subset {\rm int}(M')$, such that
$$
{\rm vol}_g(\Sigma)\geq W^{(k)}(S_r^{n+1},\overline{g})=r^n {\rm vol}_{\overline{g}}(S_1^n)=\lambda.
$$

 The theorem is proved by defining $\psi=\pi \circ \psi':\Sigma\rightarrow M$.

\end{proof}

In $(n+1)$-dimensional spheres, $3\leq (n+1)\leq 6$, there are bumpy metrics for which the volume  of Morse index one closed immersed minimal hypersurfaces is uniformly bounded.

\begin{thm}\label{sphere}
Let $3\leq (n+1)\leq 6$. Then there is a neighborhood $U$ of $\overline{g}$ on $S^{n+1}$, and $C>0$ such that for any  connected closed minimal immersion $\psi:\Sigma^n\rightarrow (S^{n+1},g)$, $g\in U$, of Morse index one, ${\rm vol}_{\psi^*(g)}(\Sigma)\leq C$. 
\end{thm}

\begin{proof}
The function $f(x)=\langle x,v\rangle$,  $v\in S^{n+1}$, is such that ${\rm Hess}_{\overline{g}} f=-f \overline{g}$.
Let $h= {\rm Hess}_gf+ f g $, and $0<\eta<1$.

Suppose  the map $\psi$ is transversal to $\{f=-\eta\}$. Let $\Omega$ be a component of $\{y\in \Sigma:f(\psi(y))\geq -\eta\}$, and $\phi:\Sigma\rightarrow \mathbb{R}$ be the function $\phi(y)=f(\psi(y))+\eta$. Then $\phi(y)=0$ for $y\in \partial\Omega$. Since $\psi(\Sigma)$ is a minimal hypersurface for the metric $g$,  it follows that
\begin{eqnarray*}
\Delta_{\psi^*(g)}\phi&=&- n f\circ \psi+ tr_{\psi(\Sigma),g}h\\
&=& -n \phi+ n\eta + tr_{\psi(\Sigma),g}h.
\end{eqnarray*}

Hence
\begin{eqnarray*}
Q_\Omega(\phi,\phi)&=&\int_\Omega ( |\nabla_{\psi^*(g)}\phi|^2 - (|A_{\psi(\Sigma)}|^2+Ric_g(N,N))\phi^2)d\Sigma\\
&=&- \int_\Omega \phi (\Delta_{\psi^*(g)}\phi + (|A_{\psi(\Sigma)}|^2+Ric_g(N,N))\phi)d\Sigma\\
&=&- \int_\Omega \phi ( -n \phi+ n\eta + tr_{\psi(\Sigma),g}h + (|A_{\psi(\Sigma)}|^2+Ric_g(N,N))\phi)d\Sigma\\
&=&-\int_\Omega |A_{\psi(\Sigma)}|^2\phi^2 d\Sigma\\
&& -\int_\Omega \phi \{ n \eta + tr_{\psi(\Sigma),g}h + (Ric_g(N,N)-n)\phi\}d\Sigma,
\end{eqnarray*}
where $Q$ is the second variation of area quadratic form.

There is a neighborhood $U$ of $\overline{g}$ such that if $g\in U$, then 
$|tr_{\psi(\Sigma),g}h + (Ric_g(N,N)-n)\phi|\leq n\eta/2.$ Hence
$$
Q_\Omega(\phi,\phi)\leq -\frac{n\eta}{2} \int_\Omega \phi d\Sigma.
$$
We can suppose that $\{x\in S^{n+1}: f(x) \leq -c\}$ has strictly mean-convex boundary for any $\eta \leq c <1$, $|v|=1$ and $g\in U$.

Let $y\in \Sigma$, and $v=-\psi(y)$.  It follows by the maximum principle that $\psi(\Sigma)$ intersects $\{x\in S^{n+1}: f(x) > -\eta\}$.  Let $\eta<\eta'<1$. Since $f(\psi(y))=-1$, $\psi(\Sigma)$ intersects $\{x\in S^{n+1}:f(x)\leq -\eta'\}$.
  The hypersurface $\psi(\Omega)$ is  unstable   since $\phi>0$  on $\Omega$. If $\Omega'$ is the component of 
  $\{y\in \Sigma: f(\psi(y))\leq -\eta'\}$ that contains $y$, then $\psi(\Omega')$ is stable since $\psi(\Sigma)$ has Morse index one and the domains $\Omega$ and $\Omega'$ are disjoint. By the curvature estimates of \cite{mazet},
  $|A_\Sigma|(y)\leq c$ where $c>0$ is a constant that does not depend on $y$, $\psi(\Sigma)$ and $g\in U$.
  
  Let $x,y\in \Sigma$ be such that $r=d_{\psi^*(g)}(x,y)={\rm diam}(\Sigma, \psi^*(g))$.  The balls $B_{r/2}(x), B_{r/2}(y)$ are disjoint, hence  either $\psi(B_{r/2}(x))$ or $\psi(B_{r/2}(y))$ is stable. This uses that the Morse index of $\psi(\Sigma)$ is one. By \cite{shen-ye}, if $(n+1)\leq 5$, since there is $\lambda>0$ such that $biRic(g)\geq \lambda$ for $g\in U$, it follows that 
  $r\leq c'/\sqrt{\lambda}$ for some uniform constant $c'>0$. If $(n+1)=6$, this follows by the proof of Theorem 1.2 in \cite{catino-mastrolia-roncoroni} since $g\in U$ has nonnegative sectional curvature.
  
  Since $|A_\Sigma|(y)\leq c$, the Ricci curvature of $\psi^*(g)$ is uniformly bounded. It follows by the diameter estimate 
  ${\rm diam}(\Sigma,\psi^*(g))\leq c'/\sqrt{\lambda}$ that the volume of $(\Sigma,\psi^*(g))$ is uniformly bounded. This finishes the proof of the theorem.
  
  \end{proof}

 \section{Fiberings and  three-manifolds}

  If $(M^{n+1},g)$ is a closed Riemannian manifold, we denote by $\lambda_k(M,g)$ the $k$-th eigenvalue of the
 Laplacian $\Delta_g$.  Then
 $$ 
 \lambda_k(M,g)= \inf_{E \in \mathcal{E}_k} \sup_{f\in E\setminus \{0\}} \frac{\int_M |\nabla_gf|^2 dv_g}{\int_M f^2 dv_g},
 $$
 where $\mathcal{E}_k$ is the set of vector spaces $E\subset W^{1,2}(M)$, ${\rm dim}(E)=k$. We denote by $W^{1,2}(M)$ the Sobolev
 space of $L^2$ functions $f$ with $\nabla_g f\in L^2$.

  \begin{prop}\label{index.coverings}
  Let $(M^{n+1},g)$ be a closed Riemannian manifold, and $\phi:M\rightarrow \mathbb{R}$ be a smooth function.
  Suppose that ${\rm index}(\Delta_g+\phi)\geq 1$. Let $k\in \mathbb{N}$ be such that  there is a sequence of finite coverings $\pi_i:M_i\rightarrow M$ with $\lambda_k(M_i,\pi_i^*(g))\rightarrow 0$. Then, for sufficiently large $i$,
  $$
  {\rm index}(\Delta_{\pi_i^*(g)}+\phi\circ \pi_i)\geq k.
  $$
  \end{prop}
  
  \begin{proof}
  Let $\psi:M\rightarrow \mathbb{R}$, $\psi>0$, be the first eigenfunction of the linear operator $L=\Delta_g+\phi$. Hence
  $\Delta_g\psi+\phi \psi+\lambda \psi=0$, where $\lambda=\lambda_1(L)$ is the first eigenvalue of $L$. Since ${\rm index}(\Delta_g+\phi)\geq 1$, it follows that $\lambda<0$.
  
  We can define $\psi_i:M_i\rightarrow \mathbb{R}$ by $\psi_i=\psi\circ \pi_i$. Then
  $$
  \Delta_{\pi_i^*(g)}\psi_i+(\phi\circ \pi_i)\psi_i+\lambda\psi_i=0.
  $$
  
  Let $f:M_i\rightarrow \mathbb{R}$. Since $\psi_i>0$, we can define $h=f/\psi_i$. Then, if $g_i=\pi_i^*(g)$,
  \begin{eqnarray*}
 && \int_{M_i} \{|\nabla _{g_i}f|^2 - (\phi\circ \pi_i)f^2\}dv_{g_i}= \int_{M_i} \{|\nabla _{g_i}(h\psi_i)|^2 - (\phi\circ \pi_i)(h\psi_i)^2\}dv_{g_i}\\
  &=&\int_{M_i} \{\psi_i^2|\nabla_{g_i}h|^2 + 2\psi_ih \langle \nabla_{g_i}\psi_i,\nabla_{g_i}h\rangle + h^2|\nabla_{g_i}\psi_i|^2 -(\phi\circ \pi_i)(h\psi_i)^2\}dv_{g_i}.\\
\end{eqnarray*}
We have  that
\begin{eqnarray*}
&&2\int_{M_i} \psi_ih \langle \nabla_{g_i}\psi_i,\nabla_{g_i}h\rangle dv_{g_i}\\
&&=-\int_{M_i} h^2|\nabla_{g_i}\psi_i|^2 dv_{g_i}
+\int_{M_i}(\phi\circ\pi_i)h^2\psi_i^2dv_{g_i}+\lambda\int_{M_i}h^2 \psi_i^2dv_{g_i}.
\end{eqnarray*}

Therefore
$$
 \int_{M_i} \{|\nabla _{g_i}f|^2 - (\phi\circ \pi_i)f^2\}dv_{g_i}     =\int_{M_i} \{\psi_i^2|\nabla_{g_i}h|^2+\lambda \psi_i^2h^2 \}dv_{g_i}.
 $$


Notice that $c_1=(\inf_{M}\psi)^2\leq \psi_i^2 \leq c_2=(\sup_{M}\psi)^2$, $0<c_1\leq c_2$, since $\psi_i=\psi\circ \pi_i$.
Hence, because $\lambda<0$,
$$
 \int_{M_i} \{|\nabla _{g_i}f|^2 - (\phi\circ \pi_i)f^2\}dv_{g_i}
 \leq \int_{M_i} \{c_2|\nabla _{g_i}h|^2+c_1\lambda h^2\}dv_{g_i}.
$$


Since $\lambda_k(M_i,\pi_i^*(g))\rightarrow 0$, for sufficiently large $i$ there is a vector space $E_i\subset W^{1,2}(M_i)$, ${\rm dim}(E_i)=k$, such that
$$
\sup_{h\in E_i\setminus \{0\}} \frac{\int_M |\nabla_{g_i}h|^2 dv_{g_i}}{\int_M h^2 dv_{g_i}} \leq (-\lambda)c_1/(2c_2).
$$

Let $\tilde{E}_i=\{f=h\psi_i: h\in E_i\}\subset W^{1,2}(M_i)$. Then, if $f=h\psi_i\in \tilde{E}_i$, 
$$
\int_{M_i} \{|\nabla _{g_i}f|^2 - (\phi\circ \pi_i)f^2\}dv_{g_i}\leq \frac{c_1\lambda}{2} \int_M h^2 dv_{g_i}<0.
$$
Since ${\rm dim}(\tilde{E}_i)=k$, this proves that for sufficiently large $i$
$$
{\rm index}(\Delta_{\pi_i^*(g)}+\phi\circ \pi_i)\geq k.
$$
  \end{proof}

If $L=\Delta_g+\phi$ is nonnegative (${\rm index}(\Delta_g+\phi)=0$), then there is a positive eigenfunction $u:M\rightarrow \mathbb{R}$, $\Delta_gu+\phi u +\lambda u =0$, with $\lambda=\lambda_1(L)\geq 0$. If $\pi:M'\rightarrow M$ is a finite covering, then $v=u\circ \pi$  is a positive eigenfunction of $\Delta_{\pi^*(g)} v+ (\phi\circ \pi)$ with eigenvalue $\lambda$ and hence $\Delta_{\pi^*(g)} v+ (\phi\circ \pi)$ is nonnegative.

This can be applied to the stability operator $L_\Sigma$ of a two-sided, closed  immersed minimal hypersurface  $\psi:\Sigma^n\rightarrow (M^{n+1},g)$. Hence if $\psi$ is stable the hypersurface $\psi\circ \pi$ is stable for any finite covering $\pi:\Sigma'\rightarrow \Sigma$. 
  As defined in \cite{fraser-schoen} the map $\psi$ is covering stable.

 \begin{thm} \label{boundary.case}
Let $(\Omega^{n+1},g)$ be a compact  manifold with boundary such that $3\leq (n+1)\leq 7$. Suppose that  $\partial\Omega$ is a strictly stable minimal hypersurface and that the metric   $g$ is bumpy. Suppose that there is a surjective homomorphism $\varphi:\pi_1(\Omega)\rightarrow \mathbb{Z}$ such that  for a  component $\Sigma$ of $\partial \Omega$ the restriction  $\varphi:\pi_1(\Sigma)\rightarrow \mathbb{Z}$ is nontrivial.   Then $\Omega$ contains infinitely many geometrically distinct connected, smooth, two-sided, virtually embedded, closed minimal hypersurfaces  of Morse index one with unbounded $n$-volumes.
\end{thm}

\begin{proof}
Let $\pi_i:\Omega_i\rightarrow \Omega$ be the finite covering such that $(p_i)_{\#}(\pi_1(\Omega_i))= \varphi^{-1}(i\cdot \mathbb{Z})$. We denote by $g$ the pull-back metrics.

Let $\varphi(\pi_1(\Sigma))=j \cdot \mathbb{Z} \subset \mathbb{Z}$, $j\in \mathbb{N}$. If $k\in \mathbb{N}$, let $p_k:\Sigma^{(k)}\rightarrow \Sigma$ be the finite covering such that $(p_k)_{\#}(\pi_1(\Sigma^{(k)}))= \{c\in \pi_1(\Sigma): \varphi(c)\in k j \cdot \mathbb{Z}\}$. The map
$p_k:\Sigma^{(k)}\rightarrow \Omega$ lifts to an embedding $p_k':\Sigma^{(k)}\rightarrow \Omega_{kj}$. Hence $p_k'(\Sigma^{(k)})$ is a stable, two-sided, embedded closed minimal hypersurface in $\Omega_{kj}$ with ${\rm vol}_g(p_k'(\Sigma^{(k)}))=k\cdot {\rm vol}(\Sigma).$ By Proposition \ref{stable.index}, there is a connected, smooth, embedded, two-sided, closed minimal hypersurface $\Sigma_k\subset {\rm int}(\Omega_{kj})$, of Morse index one, such that
$$
{\rm vol}_g(\Sigma_k)>{\rm vol}_g(p_k'(\Sigma^{(k)}))=k \cdot {\rm vol}(\Sigma).
$$
Then $(\pi_{kj})_{|\Sigma_k}:\Sigma_k\rightarrow \Omega$ is a virtually embedded, connected, smooth, two-sided, closed minimal hypersurface of Morse index one in $\Omega$.

Suppose that there are only finitely many images of virtually embedded, Morse index one, closed minimal hypersurfaces $\{\pi_{i_1}(\Gamma_1), \dots, \pi_{i_q}(\Gamma_q)\}$ in $\Omega$,  $\Gamma_j\subset \Omega_{i_j}$, $\{i_1,\dots, i_q\}\subset \mathbb{N}$, $1\leq j \leq q$. Notice that
$$
(\pi_{i_j})^{-1}(\pi_{i_j}(\Gamma_j)) =\cup_{\sigma\in {\rm Aut}(\pi_{i_j})} \sigma(\Gamma_j).
$$

Let $1\leq j\leq q$. If $\varphi \circ (\pi_{i_j})_{\#}:\pi_1(\Gamma_j)\rightarrow \mathbb{Z}$ is the trivial homomomorphism, then the map
$(\pi_{i_j})_{|\Gamma_j}:\Gamma_j\rightarrow \Omega$ lifts to any covering $\pi_i:\Omega_i\rightarrow \Omega$. Therefore, if  $\Sigma'$ is a connected component of $\pi_{k\cdot i_j}^{-1}(\pi_{i_j}(\Gamma_j))$ in $\Omega_{k\cdot i_j}$, where $k\in \mathbb{N}$, then 
${\rm vol}_g(\Sigma')\leq {\rm vol}_g(\Gamma_j)$.  Let $\Lambda_1$ be the set of such $1\leq j\leq q$.

If $\varphi \circ (\pi_{i_j})_{\#}:\pi_1(\Gamma_j)\rightarrow \mathbb{Z}$ is nontrivial, then $(\varphi \circ (\pi_{i_j})_{\#})(\pi_1(\Gamma_j))=i_j\cdot t_j\cdot \mathbb{Z}$ where $t_j\in \mathbb{N}$. Let $p_{j,k}:\Gamma_{j,k}\rightarrow \Gamma_j$ be the finite covering such that $(p_{j,k})_{\#}(\pi_1(\Gamma_{j,k}))=\{c\in \pi_1(\Gamma_j):\varphi((\pi_{i_j})_{\#}(c))\in k \cdot i_j \cdot t_j\cdot \mathbb{Z}\}.$ Then the map
$p_{j,k}:\Gamma_{j,k}\rightarrow \Omega_{i_j}$ lifts to an embedding $p_{j,k}':\Gamma_{j,k}\rightarrow \Omega_{k\cdot i_j\cdot t_j}$. Let $\Lambda_2$ be the set of such $1\leq j\leq q$.

Since $\{p_{j,k}:\Gamma_{j,k}\rightarrow \Gamma_j\}_k$ is a sequence of cyclic coverings, by Brooks \cite{brooks} $\lambda_2(\Gamma_{j,k},g)\rightarrow 0$ as $k\rightarrow \infty$. Hence, by Proposition \ref{index.coverings} there is $k_j\in \mathbb{N}$ such that   ${\rm index}(p_{j,k}'(\Gamma_{j,k}))\geq 2$ for  $k\geq k_j$. Therefore, if  $\Sigma'$ is a connected component of $\pi_{k\cdot i_j\cdot t_j}^{-1}(\pi_{i_j}(\Gamma_j))$ in $\Omega_{k\cdot i_j\cdot t_j}$, where $k\geq k_j$, then 
${\rm index}(\Sigma')\geq  2$.

Define $i=(\prod_{j\in \Lambda_1} i_j)(\prod_{j\in \Lambda_2} i_jt_j)$, and $k'=\max_{j\in \Lambda_2} k_j$.

Hence, if $k \geq k'$  then 
$$
\pi_{k\cdot i \cdot j}^{-1}(\cup_{r=1}^q \pi_{i_r}(\Gamma_{r}))
$$
is a finite union of connected, embedded, two-sided, closed minimal hypersurfaces $\Sigma''_1, \dots, \Sigma''_h$, such that  if $1\leq r\leq h$ then either
${\rm vol}_g(\Sigma''_r)\leq \max_{j\in \Lambda_1} {\rm vol}_g(\Gamma_j)$ or ${\rm index}(\Sigma''_r)\geq 2$. Therefore, if $k$ is sufficiently large then the Morse index one, smooth, embedded, closed minimal hypersurface $\Sigma_{k\cdot i}$ in $\Omega_{k\cdot i \cdot j}$ is such that $\Sigma_{k\cdot i}\neq \Sigma''_r$ for any $1\leq r\leq h$. By unique continuation of minimal hypersurfaces, $\Sigma_{k\cdot i}\setminus \Sigma''_r$ is open and dense in $\Sigma_{k\cdot i}$ for any $1\leq r \leq h$.
Hence $\Sigma_{k\cdot i}\setminus (\cup_{r=1}^h \Sigma''_r)$ is open and dense in $\Sigma_{k\cdot i}$. Therefore
there is $y\in \pi_{k\cdot i\cdot j}(\Sigma_{k\cdot i})$ such that $y\notin \cup_{r=1}^q \pi_{i_r}(\Gamma_{r})$. This is a contradiction since
$\pi_{k\cdot i\cdot j}(\Sigma_{k\cdot i})$ is the image of a  virtually embedded, Morse index one, closed minimal hypersurface in $\Omega$.

Since there is a sequence of virtually embedded, Morse index one, closed minimal hypersurfaces in $\Omega$ with geometrically distinct images, the volumes are unbounded by Proposition \ref{curvature.estimates}. This proves the theorem. 

\end{proof}

 \begin{prop}\label{curvature.estimates}
Let $M^{n+1}$ be a compact manifold possibly  with boundary, $3\leq (n+1)\leq 7$. Suppose that $g$ is a bumpy Riemannian metric on $M$ such that $\partial M$ is a strictly stable minimal hypersurface. If $\lambda>0$, define $\mathcal{V}_\lambda$ to be the set of connected,  smooth, two-sided, virtually embedded,  of Morse index one, closed minimal hypersurfaces $\psi:\Sigma\rightarrow M$ such that ${\rm vol}_g(\psi(\Sigma))\leq \lambda$. Then the set of images $\{\psi(\Sigma):\Sigma\in \mathcal{V}_\lambda\}$ is finite.
\end{prop}
 
 \begin{proof}
 Let $\psi\in \mathcal{V}_\lambda$. It follows that there is a finite cover $\pi:M_\Sigma \rightarrow M$ which depends on $\psi$ such that $\psi$ lifts to an embedding
$\psi':\Sigma\rightarrow M_\Sigma$.
We identify $\Sigma$ with $\psi'(\Sigma)$ and consider the  lift $\psi':\Sigma\rightarrow M_\Sigma$ as the inclusion map.
 Hence $\Sigma$ is an embedded, closed minimal hypersurface in $M_\Sigma$ and $\psi=\pi_{|\Sigma}$. We would like to prove that 
 the set $\{\pi(\Sigma):\Sigma\in \mathcal{V}_\lambda\}$ is finite.

Let $\Sigma\in \mathcal{V}_\lambda$. Hence there is a finite cover $\pi:M_\Sigma\rightarrow M$ such that $\Sigma\subset M_\Sigma$.
The map $\pi_{|\Sigma}$ has a double cover if there are open sets $U,\tilde{U}\subset \Sigma$, disjoint, such that
$\pi(U)=\pi(\tilde{U})$.  Let $\mathcal{V}_\lambda'$ be the set of $\Sigma\in \mathcal{V}_\lambda$ such that $\pi_{|\Sigma}$ does not have a double cover, and  $\mathcal{V}_\lambda''=\mathcal{V}_\lambda \setminus \mathcal{V}_\lambda'.$
We are going to prove that the sets $\{\pi(\Sigma):\Sigma\in \mathcal{V}_\lambda'\}$ and
$\{\pi(\Sigma):\Sigma\in \mathcal{V}_\lambda''\}$ are finite.  The map $\pi$ depends on $\Sigma$.


Suppose, by contradiction, that there is a sequence $\Sigma_i\in \mathcal{V}_\lambda'$, $\Sigma_i\subset M_i$, $\pi_i:M_i\rightarrow M$,  
such that $\pi_i(\Sigma_i)\neq \pi_j(\Sigma_j)$ for any $i\neq j$. 
Then $\Sigma_i\subset M_i$ is a smooth, connected, two-sided, embedded, closed minimal hypersurface of Morse index one,
and $(\pi_i)_{|\Sigma_i}$ does not have a double cover. This implies that ${\rm vol}_{\pi_i^*(g)}(\Sigma_i)={\rm vol}_g(\pi_i(\Sigma_i))$, hence ${\rm vol}_{\pi_i^*(g)}(\Sigma_i)\leq \lambda$.

If  $r$ is the injectivity radius ${\rm inj}(M,g)$ of $(M,g)$, then ${\rm inj}(M_i,\pi_i^*(g))\geq r$. It follows from the  monotonicity formula for minimal surfaces in $(M,g)$ that there is a constant $c>0$ such that if $p\in \Sigma_i$, then ${\rm vol}_{\pi_i^*(g)}(\Sigma_i\cap B_r(p))\geq c$ where $B_r(p)$ is the geodesic ball in $M_i$. Hence  $d_i={\rm diam}(\Sigma_i)\leq d=2r\lambda/c$, where ${\rm diam}$ is the extrinsic diameter of $\Sigma_i\subset M_i$. By \cite{sharp}, there is  $\{j\}\subset \{i\}$ such that $\Sigma_j$ converges locally graphically away from finitely many points in a covering space. Since $g$ is bumpy this is a contradiction. This proves that $\{\pi(\Sigma):\Sigma\in \mathcal{V}_\lambda'\}$  is finite.

Suppose, by contradiction, that there is a sequence $\Sigma_i\in \mathcal{V}_\lambda''$, $\Sigma_i\subset M_i$, $\pi_i:M_i\rightarrow M$,  
such that $\pi_i(\Sigma_i)\neq \pi_j(\Sigma_j)$ for any $i\neq j$. 
Then $\Sigma_i\subset M_i$ is a smooth, connected, two-sided, embedded, closed minimal hypersurface of Morse index one,
and $(\pi_i)_{|\Sigma_i}$ has a double cover. 
Hence there is a nontrivial cover  $\pi_i':\Sigma_i\rightarrow \Sigma_i'$ such that there is an immersion 
$\psi_i':\Sigma_i'\rightarrow M$ that has no double cover satisfying $\psi_i'\circ \pi_i'=\pi_i.$ It follows that
${\rm vol}_{(\psi_i')^*(g)}(\Sigma_i')={\rm vol}_g(\psi_i'(\Sigma_i'))={\rm vol}_g(\pi_i(\Sigma_i))$, hence
${\rm vol}_{(\psi_i')^*(g)}(\Sigma_i')\leq \lambda$.

Since $\psi_i'\circ \pi_i'=\pi_i,$ we have that
$\pi_i^*(g)=(\psi_i'\circ \pi_i')^*(g)=(\pi_i')^*(\psi_i')^*(g).$ Therefore $\pi_i':(\Sigma_i,\pi_i^*(g))\rightarrow (\Sigma_i',(\psi_i')^*(g))$ is a local isometry.

Let $p\in \Sigma_i$. Since $\pi_i':\Sigma_i\rightarrow \Sigma_i'$ is not a diffeomorphism, there is $q\in \Sigma_i$ such that $\pi_i'(p)=\pi_i'(q)=y'$.  Therefore $\pi_i(p)=\pi_i(q)=y$, and hence $d(p,q)\geq r$ since ${\rm inj}(M_i,\pi_i^*(g))\geq r$. It follows that the balls $B_{r/2}(p)$ and $B_{r/2}(q)$ are disjoint.  

Let $\Sigma_i(p)$ be the connected component of $\Sigma_i\cap B_{r/2}(p)$ such that $p\in \Sigma_i(p)$, and $\Sigma_i(q)$ be the  component of $\Sigma_i\cap B_{r/2}(q)$ such that $q\in \Sigma_i(q)$. Then $\pi_i(\Sigma_i(p))$ and $\pi_i(\Sigma_i(q))$ are connected, compact minimal hypersurfaces properly embedded  in $B_{r/2}(y)$. Since $\pi_i(\Sigma_i(p))$ and $\pi_i(\Sigma_i(q))$ contain
$\psi_i'(B_{r/2}^{\Sigma_i'}(y'))$, where $B_{s}^{\Sigma_i'}(\tilde{p})$ is the intrinsic geodesic ball of radius $s$ centered at $\tilde{p}$ in the Riemannian surface $(\Sigma_i',(\psi_i')^*(g))$, it follows  that $\pi_i(\Sigma_i(p))=\pi_i(\Sigma_i(q))$.

Hence the   index of the minimal hypersurface $\Sigma_i(p)$ and   the  index of $\Sigma_i(q)$ coincide, where the index is defined  for the compactly supported variations.  If  ${\rm index}(\Sigma_i(p))\geq 1$ then
it would follow that ${\rm index}(\Sigma_i(q))\geq 1$ and hence ${\rm index}(\Sigma_i)\geq 2$. This contradicts that ${\rm index}(\Sigma_i)=1$.

Therefore $\Sigma_i(p)$ is a stable minimal hypersurface   for any $p\in \Sigma_i$. Since $(\pi_i')_{|\Sigma_i(p)}$ is
 injective, ${\rm vol}_{\pi_i^*(g)}(\Sigma_i(p))\leq {\rm vol}_{(\psi_i')^*(g)}(\Sigma_i')\leq  \lambda$.  It follows that there is a constant $C>0$ such that $|A_{\Sigma_i}|(p)\leq C$ for any $i\geq 1$ and $p\in \Sigma_i$, 
where $A_\Sigma$ denotes  the second fundamental form of  $\Sigma$. This uses the curvature estimates of \cite{schoen-simon}.

Hence  
there is $\{j\}\subset \{i\}$ such that the sequence $\psi_j':\Sigma_j'\rightarrow M$ converges  to $\psi:\Sigma'\rightarrow M$. The map $\psi$   is a degenerate minimal immersion since $\psi_i'(\Sigma_i')\neq  \psi_j'(\Sigma_j')$ for $i\neq j$. Since   $g$ is  bumpy this is a contradiction. Hence 
$\{\pi(\Sigma):\Sigma\in \mathcal{V}_\lambda''\}$ is  finite. This finishes the proof of the proposition.

\end{proof}

\begin{cor}\label{three-dimensional.boundary.case}
Let $\Omega$ be a compact three-manifold with boundary with a bumpy metric $g$ and  stable minimal boundary $\partial \Omega$. Suppose that there is a component of $\partial \Omega$ with positive genus.  Then $\Omega$ contains infinitely many geometrically distinct connected,  two-sided, virtually embedded, closed minimal hypersurfaces  of Morse index one with unbounded areas.
\end{cor}

\begin{proof}
The rank of the map $H_1(\partial\Omega,\mathbb{Q})\rightarrow H_1(\Omega,\mathbb{Q})$ induced by the inclusion  $i:\partial\Omega\rightarrow \Omega$ is $\frac12 {\rm dim}(H_1(\partial\Omega,\mathbb{Q}))$  (\cite{hatcher_three-manifold}, Lemma 3.5), and hence it is positive since the boundary $\partial \Omega$ has a component with positive genus.  By naturality, using that $H_1(X,\mathbb{Q})=H_1(X,\mathbb{Z})\otimes \mathbb{Q}$, the image of the homomorphism $i_{\#}:H_1(\partial\Omega,\mathbb{Z})\rightarrow H_1(\Omega,\mathbb{Z})$ cannot be contained in the torsion part of $H_1(\Omega,\mathbb{Z})$.

Hence there is a homomorphism $\psi:H_1(\Omega,\mathbb{Z})\rightarrow \mathbb{Z}$ such that $\psi\circ i_{\#}:H_1(\partial\Omega,\mathbb{Z})\rightarrow \mathbb{Z}$ is nontrivial. It follows that there is a component $\Sigma$ of $\partial\Omega$ such that the restriction homomorphism $\psi\circ i_{\#}:H_1(\Sigma,\mathbb{Z})\rightarrow \mathbb{Z}$ is nontrivial.

Let $p:\pi_1(\Omega)\rightarrow H_1(\Omega,\mathbb{Z})$ be the natural homomorphism. Theorem \ref{boundary.case} can be used with the homomorphism  $\varphi=\psi\circ p:\pi_1(\Omega)\rightarrow \mathbb{Z}$. This finishes the proof.






\end{proof}

\begin{thm}\label{fiber.bundle}
Let $M_\phi$ be an $(n+1)$-dimensional mapping torus  with fiber $M$ such that  $b_1(M)>0$, and     $3\leq (n+1)\leq 7$.  Suppose   $M'$ is a closed manifold such that there is a finite cover $\pi:M''\rightarrow M'$  that has a continuous map $h:M''\rightarrow M_\phi$ of nontrivial degree. If $g$ is a bumpy metric on $M'$, then $(M',g)$
contains infinitely many geometrically distinct connected,  two-sided, immersed, closed minimal hypersurfaces  of Morse index one with unbounded volumes.
 \end{thm}

\begin{proof}
Since $\pi:M''\rightarrow M'$ is a finite cover, and 
$${\rm vol}_g(\pi(\Sigma))\geq {\rm vol}_g(\Sigma)/([M'':M'])
$$
for any smooth immersed hypersurface $\Sigma$ in $M''$,   we can suppose that $M'$ has a  map $h:M'\rightarrow M_\phi$ of nontrivial degree $k$.

Let $PD$ denote the Poincar\'{e} duality map.  Define $\alpha=PD([M])\in H^1(M_\phi,\mathbb{Z})$, where $M$ is the fiber of $M_\phi$. If $\sigma=PD(h^*(\alpha))\in H_n(M',\mathbb{Z})$,  $h_*(\sigma)= k \cdot [M] \in H_n(M_\phi,\mathbb{Z})$. By minimizing the volume in the homology class $\sigma$, there is a disjoint collection $\{\Sigma_1,\dots, \Sigma_q\}$ of stable, two-sided, embedded closed minimal hypersurfaces of $M'$ and $\{m_1,\dots, m_q\}\subset \mathbb{N}$ such that
$
\sum_{i=1}^q m_i\Sigma_i \in \sigma.
$

Let $\Omega$ be a connected component of the complement of $\cup_{i=1}^q\Sigma_i$ in $M'$, and $\tilde{\Omega}$ be its metric closure.
Let $i:\tilde{\Omega}\rightarrow M'$ be the map induced by the inclusion $\Omega\subset M'$. If $c:S^1 \rightarrow \tilde{\Omega}$ is a loop
and $\cdot$ denotes  the algebraic intersection number, then $(i\circ c)_*([S^1]) \cdot \sigma =0$.  Hence $PD(\sigma)((i\circ c)_*([S^1]))=0$,
or $h^*(\alpha)((i\circ c)_*([S^1]))=0$. Therefore $\alpha((h\circ i\circ c)_*([S^1]))=0$.

Let $p:M\times \mathbb{R}\rightarrow M_\phi$ be the covering map such that $p_{\#}(\pi_1(M\times \mathbb{R}))=\{c'\in \pi_1(M_\phi):\alpha(c')=0\}$, and $M$ be the fiber $M\times \{0\}$.
Then $(h\circ i)_{\#}(\pi_1(\tilde{\Omega}))\subset p_{\#}(\pi_1(M\times \mathbb{R}))$, and hence there is a map $\lambda:\tilde{\Omega}\rightarrow M\times \mathbb{R}$ such that $p\circ \lambda=h\circ i$. Let $\Sigma$ be a component of $\partial \tilde{\Omega}$, with the orientation as  in $\{\Sigma_1,\dots,\Sigma_q\}$.  If $\pi':M\times \mathbb{R}\rightarrow M$ denotes the projection map $\pi'(x,t)=(x,0)$, then $\lambda_*(\Sigma)$ is homologous to $\pi'_*(\lambda_*(\Sigma))$. There is $k_{\Sigma,\Omega}\in \mathbb{Z}$ such that $\pi'_*(\lambda_*(\Sigma))=k_{\Sigma,\Omega} \cdot [M]$. Then 
$$
h_*(i_*([\Sigma]))=p_*(\lambda_*([\Sigma]))=p_*(k_{\Sigma,\Omega} \cdot [M])=k_{\Sigma,\Omega} \cdot [M]\in H_n(M_\phi,\mathbb{Z}).
$$

Since
$\sum_{i=1}^q m_ih_*([\Sigma_i])\neq 0$, there is a connected component $\Omega$  of  $M'\setminus (\cup_{i=1}^q\Sigma_i)$ and a component $\Sigma\subset \partial\Omega$ such that $k_{\Sigma,\Omega}\neq 0$. The map $f=\pi'\circ \lambda:\tilde{\Omega}\rightarrow M$ is such that
$f_*(\Sigma)=k_{\Sigma,\Omega} \cdot [M]$. Hence the map  $f_{|\Sigma}:\Sigma\rightarrow M$ has nontrivial degree and therefore
$(f_{|\Sigma})^*:H^1(M,\mathbb{Z})\rightarrow H^1(\Sigma,\mathbb{Z})$ is injective. Since $b_1(M)>0$, there is $\psi\in H^1(M,\mathbb{Z})$, $\psi\neq 0$, hence $(f_{|\Sigma})^*(\psi)\neq 0 \in H^1(\Sigma,\mathbb{Z})$. Therefore the cohomology class $f^*(\psi)\in H^1(\tilde{\Omega},\mathbb{Z})$ is such that $f^*(\psi)_{|\Sigma}\neq 0$. The theorem is proved by using Theorem \ref{boundary.case} and the homomorphism $\varphi:\pi_1(\tilde{\Omega})\rightarrow \mathbb{Z}$ defined by $\varphi(c)=f^*(\psi)([c])$.

\end{proof}

More precisely, there is a nontrivial, finite, disjoint collection of  stable, two-sided, closed embedded minimal hypersurfaces $\{\Sigma_i\}$ in $M''$ such that there is a sequence of  two-sided, closed minimal hypersurfaces  of Morse index one with unbounded volumes,  virtually embedded in a connected component $\Omega$ of the complement of $\cup\Sigma_i$.

  \begin{thm}\label{index.surfaces}
Let $(M,g)$ be a closed orientable three-dimensional manifold. Suppose that $M$ is not a connected sum of spherical quotients and copies of  $S^2\times S^1$. If $g$ is bumpy, then there is a sequence of  connected,  smooth, two-sided, virtually embedded, closed minimal surfaces $\psi_i:\Sigma_i\rightarrow M$, of Morse index one, such that 
 ${\rm area}_g(\psi_i(\Sigma_i))\rightarrow \infty.$
\end{thm}

\begin{proof}

By Propositions \ref{stable.index} and \ref{stable.one-sided.index}, we can suppose that there is a constant $c>0$ such that any smooth, connected,   embedded closed minimal hypersurface $\Sigma'\subset M$ that is either two-sided stable or one-sided with stable oriented two-cover, is such that ${\rm area}_g(\tilde{\Sigma'})\leq c$. Here 
$\tilde{\Sigma}'=\Sigma'$ if  $\Sigma'$ is a two-sided surface, and $\tilde{\Sigma}'$ is the oriented two-cover of $\Sigma'$ in the one-sided case. It follows from the curvature estimates of \cite{schoen} for stable surfaces, since $g$ is a bumpy metric, that the set $\Lambda$ of stable embedded minimal two-spheres and embedded minimal projective planes with stable oriented two-cover in $M$     is finite. Therefore there is a  set 
$\tilde{\Lambda}\subset \Lambda$ of mutually disjoint closed minimal surfaces, that is not contained in any other set with this property.  

Let $\Omega_1,\dots, \Omega_p$ be the connected components of the complement of $\cup_{\Sigma\in \tilde{\Lambda}}\Sigma$.
 It follows that, for any $1\leq j\leq p$, there is no stable embedded minimal two-sphere or embedded minimal projective plane with stable oriented two-cover in   ${\rm int}(\Omega_j)$. 

Let $1\leq j\leq p$, and define $\tilde{\Omega}_j$ to be the metric closure of $\Omega_j$. Suppose $\Sigma\subset {\rm int}(\tilde{\Omega}_j)$ is an embedded two-sphere. Let $\mathcal{I}(\Sigma)$ be the isotopy class of $\Sigma$  in $\tilde{\Omega}_j$ and $\{\Gamma_i\}$ in $\mathcal{I}(\Sigma)$ be a minimizing sequence. Then it follows from \cite{meeks-simon-yau}, after passing to a subsequence, that there is a sequence  $\{\tilde{\Gamma}_i\}_i$ such that $\tilde{\Gamma}_i$ is obtained from $\Gamma_i$ by performing  $\gamma$-reductions and so that $\tilde{\Gamma}_i$ is isotopic to 
$$
\Gamma_i' \cup \left(\cup_{k=1}^r \Gamma_{i,k}\right),
$$ 
where $\Gamma_{i,k}$ is a disjoint union of  surfaces parallel to a component of $\partial \tilde{\Omega}_j$  for any $1\leq k \leq r$ and $\lim_{i\rightarrow \infty} {\rm area}_g(\Gamma_i')=0.$ 
A $\gamma$-reduction modulo an isotopy corresponds to a neckpinch topologically, in which a cylinder contained in the surface is replaced by two disks and the union of the cylinder and the disks is the boundary of an open three-ball  disjoint from the surface. 


 \medskip
 
   {\bf Claim:} Let $\Omega'$ be a closed three-manifold. Let $\Sigma\subset \Omega'$ be an embedded two-sphere.  Suppose that $\tilde{\Sigma}$ is obtained from $\Sigma$ by a $\gamma$-reduction in $\Omega'$, and let $\tilde{\Sigma}_1$ and $\tilde{\Sigma}_2$ be the connected components of $\tilde\Sigma$. If $\tilde{\Sigma}_i$ is the boundary of a three-ball $\overline{B}_i\subset \Omega'$  for $i=1,2$, then $\Sigma$ is a separating hypersurface that is the boundary of a three-ball in $\Omega'$. 

  \medskip
  
  {\it Proof of the Claim:} Let $C\subset \Omega'$ be the open three-ball bounded by the union of the cylinder $S$ and the two disks $\tilde{D}_1$, $\tilde{D}_2$ of the $\gamma$-reduction of $\Sigma$. Denote by  $A,B\subset \Omega'$  the connected  domains that are disjoint from $C$ such that 
  $\partial A=\tilde{\Sigma}_1$ and $\partial B=\tilde{\Sigma}_2$, where $\tilde{D}_i\subset \tilde{\Sigma}_i$ for $i=1,2$. 
  
  We prove that the open sets $A$ and $B$ are disjoint. The disk  $\tilde{D}_2\subset \tilde{\Sigma}_2$  is disjoint from  $A$.  Since $\tilde{\Sigma}_2$ is connected, and $\tilde{\Sigma}_1\cap \tilde{\Sigma}_2=\emptyset$, the surface $\tilde{\Sigma}_2$ is disjoint from $A$.  By analogy, the surface $\tilde{\Sigma}_1$ is disjoint from $B$. This implies that
  $A\cap B$ is open and closed in $A$, and hence by connectedness either $A\cap B=\emptyset$ or $A \subset B$.  Since $A\cap B$ is open and closed in $B$, then either $A\cap B=\emptyset$ or $B \subset A$. Therefore either $A$ and $B$ are disjoint, or $A=B$. But $\partial A\neq \partial B$, which proves that $A$ and $B$ are disjoint. 
  
  Then $\overline{(A\cup C\cup B)}$ is  a connected domain in $\Omega'$ that is  bounded by $\Sigma$, hence $\Sigma$ is a separating hypersurface. Let $W={\rm int}(\overline{(A\cup C\cup B)})$ and $W'= \Omega'\setminus 
  \overline{(A \cup C \cup B)}$ so $\partial W'=\Sigma$. Suppose $\overline{W'}$ is not a three-ball. 
  By hypothesis, it follows that $A$ and $B$ are three-balls ($A=B_1$, $B=B_2$).  Hence $\overline{W}=\overline{A}\cup \overline{C}\cup \overline{B}$ is a connected sum along the boundary of the three-balls $\overline{A}$, $\overline{B}$, and so  is diffeomorphic to a three-ball.  This finishes the proof of the claim.

 
 Let $\tilde{\Omega}_{j,1}$ be the topological closed manifold obtained by gluing  three-balls with identification of the boundary to each connected component
 of $\partial \tilde{\Omega}_j$. Since any component of $\Gamma_i' \cup \left(\cup_{k=1}^r \Gamma_{i,k}\right)$ bounds a three-ball in $\tilde{\Omega}_{j,1}$, 
  by the topological claim the surface $\Sigma$ bounds a three-ball in $\tilde{\Omega}_{j,1}$. Since any embedded two-sphere in $\tilde{\Omega}_{j,1}$ is isotopic to a two-sphere in ${\rm int}(\tilde{\Omega}_j)$, this proves that 
$\tilde{\Omega}_{j,1}$ is irreducible.

Define $\tilde{\Lambda}'$ to be the set of embedded two-spheres in $\tilde{\Lambda}$. Then the connected components of  
 the complement of $\cup_{\Sigma\in \tilde{\Lambda}'}\Sigma$ are $\Omega_1',\dots,\Omega_p'$, where $\Omega_j'$ is the union of $\Omega_j$ and the projective spaces in $\tilde{\Lambda}$ with oriented two-cover in $\tilde{\Omega}_j$.
 Let $\tilde{\Omega}_j'$ be the metric closure of $\Omega_j'$, and $\tilde{\Omega}'_{j,1}$ be the closed manifold obtained by gluing balls with identification of the boundary to the two-spheres in $\partial \tilde{\Omega}'_j$.
 Then $\tilde{\Omega}'_{j,1}$ is a connected sum of $\tilde{\Omega}_{j,1}$ and copies of $\mathbb{RP}^3$.
 
 Let $\sqcup_{j=1}^p \Omega_j''$ be a disjoint union of compact manifolds with boundary such that $\Omega_j''$ is diffeomorphic to $\tilde{\Omega}'_{j}$. Notice that $\tilde{\Omega}'_{j}$ is diffeomorphic to $\tilde{\Omega}'_{j,1}\setminus (\cup_{i=1}^{t_j} B_{j,i})$, where $t_j$ is the number of components of $\partial\tilde{\Omega}'_{j}$ and $\{B_{j,i}\}$ is   a set of disjoint balls in $\tilde{\Omega}'_{j,1}$. Hence $M$ is the closed manifold obtained by adding copies of $S^2\times [0,1]$ with identification of the boundary to $\sqcup_{j=1}^p \Omega_j''$.

Therefore, since $M$ is not a connected sum of spherical quotients and copies of  $S^2\times S^1$, by geometrization there is $1\leq i\leq p$ such that $\tilde{\Omega}_{i,1}$ has infinite fundamental group.


Let $\{\Sigma_s\}_{s=1}^t$ be the embedded tori  of the canonical JSJ decomposition of $\tilde{\Omega}_{i,1}$ (\cite{aschenbrenner-friedl-wilton}). These surfaces are  disjoint and incompressible. We can suppose that $\Sigma_s\subset {\rm int}(\tilde{\Omega}_i)$ for any $1\leq s\leq t$ by applying to $\{\Sigma_s\}_{s=1}^t$ an isotopy. This uses that $\tilde{\Omega}_{i,1}\setminus \tilde{\Omega}_i$ is a union of disjoint balls.

By minimizing the area in isotopy classes as in \cite{meeks-simon-yau}, and because of the incompressibility of $\Sigma_s$, the surface  $\cup_{s=1}^t \Sigma_s$ is isotopic in $\tilde{\Omega}_i$ to $\cup_{\eta=1}^q \Gamma_{\eta}'\subset {\rm int}(\tilde{\Omega}_i)$, where $\Gamma_\eta'$ is a union of parallel surfaces to a stable minimal embedded torus $T_\eta$ or to a minimal Klein bottle $T_\eta$ with stable oriented two-cover depending on $\eta$.
Hence there is a connected component $R$ of ${\rm int}(\tilde{\Omega}_i)\setminus (\cup_{\eta=1}^q T_\eta)$,  the boundary  $\partial\tilde{R}$ of the metric closure of $R$ a union of minimal tori and minimal spheres, such that the manifold obtained by gluing balls with boundary identification to $\tilde{R}$ is diffeomorphic to a JSJ component of $\tilde{\Omega}_{i,1}$. 

We proved that there is an open set $R$ in $M$ such that the boundary $\partial\tilde{R}$ of the metric closure of $R$ is a union of strictly stable minimal tori and strictly stable minimal spheres and $\tilde{R}$ after gluing three-balls with identification of the boundary  is isotopic to a JSJ component. 

Suppose that  $\partial \tilde{R}$ has a toroidal component. Then, by Corollary \ref{three-dimensional.boundary.case}, ${\rm int}(\tilde{R})$
contains a sequence of geometrically distinct connected,  two-sided, virtually embedded, closed minimal hypersurfaces   of Morse index one with unbounded areas.

 If not then the manifold $\tilde{R}_{1}$ obtained by gluing balls to $\tilde{R}$ with boundary identification is a closed manifold. Hence $\tilde{R}_{1}=\tilde{\Omega}_{i,1}$ is the single piece of the JSJ decomposition of $\tilde{\Omega}_{i,1}$. 
 
 If the manifold $\tilde{\Omega}_{i,1}$ is atoroidal, then there is a hyperbolic metric by geometrization. By \cite{agol}, there is a finite cover 
 $\tilde{\Omega}_{i,1}'$ of 
$\tilde{\Omega}_{i,1}$ that is a fiber bundle with base a circle and such that the fiber has positive genus. 

There is a domain $\Omega\subset \tilde{\Omega}_{i,1}'$ that covers $\tilde{\Omega}_i=\tilde{R}$ such that the complement of $\Omega$ in $\tilde{\Omega}_{i,1}'$ is a disjoint union of three-balls. The pull-back metric $g$ on $\Omega$ is such that $\partial\Omega$ is a union of stable minimal spheres. The fiber can be isotoped to a closed embedded hypersurface $\Sigma$ that is contained in ${\rm int}(\Omega)$. We use the proof of Theorem \ref{fiber.bundle} by minimizing the area in the isotopy class, instead of the homology class, of $\Sigma$ in the domain $\Omega$. Hence the surfaces of Theorem \ref{fiber.bundle} are virtually embedded in a domain that is isotopic to $\Omega \setminus \Sigma$. Since the homomorphism $\pi_1(\Omega\setminus \Sigma)\rightarrow \pi_1(\Omega)$ induced by the inclusion is injective, and $\pi_1(\Omega)$ is LERF (\cite{agol}), a surface that is virtually embedded in $\Omega\setminus \Sigma$ is virtually embedded in $\Omega$. Hence the domain ${\rm int}(\tilde{R})$
has a sequence of geometrically distinct connected,  two-sided, virtually embedded, closed minimal hypersurfaces   of Morse index one with unbounded areas.

If $\tilde{\Omega}_{i,1}$ is not atoroidal then it is Seifert fibered. 
 Then there is a finite covering $\pi:R'\rightarrow \tilde{R}_1=\tilde{\Omega}_{i,1}$ that is an oriented circle bundle  with base an orientable compact surface $X$. Since $\tilde{\Omega}_{i,1}$ is irreducible and has infinite fundamental group, ${\rm genus}(X)\geq 1$.
 Let $\sigma$ be an embedded loop in $X$ that does not separate, and $\Sigma$ be the torus in $R'$ that projects to $\sigma$. Then  
 $\Sigma$ is a nonseparating incompressible embedded torus in $R'$ that can be isotoped to a surface $\Sigma'$ in $\pi^{-1}({\rm int}(\tilde{\Omega}_i))$.
 
 By area-minimization, since $\partial (\pi^{-1}(\tilde{R}))$ is a minimal surface for the metric $\pi^*(g)$, $\Sigma'$ is isotopic in 
  $\pi^{-1}(\tilde{R})$ to a minimal embedded  torus $\Sigma''$.
 Then, by Corollary \ref{three-dimensional.boundary.case}, there is  a sequence of geometrically distinct connected,  two-sided, immersed, closed minimal hypersurfaces   of Morse index one with unbounded areas that are virtually embedded in the complement of $\Sigma''$ in $\pi^{-1}(\tilde{R})$. Since $\pi_1(\pi^{-1}(\tilde{R}))$ is LERF (\cite{scott}), as before these surfaces are virtually embedded in $\pi^{-1}(\tilde{R})$.  
 
 Therefore in any case there is a sequence of geometrically distinct connected,  two-sided, closed minimal hypersurfaces   of Morse index one $\psi_i(\Sigma_i)$ with unbounded areas  that are virtually embedded in ${\rm int}(\tilde{R})$.

  By \cite{wilton-zalesskii} (Theorem A), the immersed surfaces $\psi_i(\Sigma_i)$
  are virtually embedded in $\tilde{\Omega}_i$. This is because ${\rm int}(\tilde{R}_1)$ is a JSJ component of $\tilde{\Omega}_{i,1}$. Since the boundary components 
 of $\partial\tilde{\Omega}_i$ are topological spheres,  it follows that  the surfaces  $\psi_i(\Sigma_i)$ are virtually embedded in $M$. 
 This finishes the proof.


\end{proof}

An example  of a  compact three-dimensional manifold with boundary that satisfies the topological conditions of Theorem 
\ref{boundary.case} and of Corollary \ref{three-dimensional.boundary.case} is $\Sigma_g\times [0,1]$, where $\Sigma_g$ is a closed surface with genus $g\geq 1$. 
There are bumpy metrics on $M=\Sigma\times [0,1]$ with $\partial M$ strictly stable, where $\Sigma$ is a closed surface with genus $g\geq 1$, such that there are no sequences of closed, embedded, two-sided minimal hypersurfaces of Morse index at most one and unbounded area.  Hence the surfaces of Theorem \ref{boundary.case} do not need to be embedded.

Let $M=\Sigma\times [0,1]$, where $\Sigma$ is a closed surface with genus $g$. We define the metric 
$g=g_\Sigma+dt^2$ on $M$, where $g_\Sigma$ is a Riemannian metric on $\Sigma$. Let $\Sigma_t=\Sigma \times \{t\}$.

\begin{prop}\label{product}
Then there is a neighborhood $U$ of $g$ in the smooth topology, and $c>0$, such that if $g'\in U$  and $\Sigma'\subset M$ is a connected, smooth, embedded, closed minimal surface for the metric $g'$, ${\rm index}(\Sigma')\leq 1$, then ${\rm area}_{g'}(\Sigma')\leq c$. 
\end{prop}

\begin{proof}
Suppose, by contradiction, that there is a sequence of metrics $g_i$ on $M$ converging to $g$ in the smooth topology, and $\Sigma_i\subset M$ a  connected, smooth, embedded, closed minimal surface for the metric  $g_i$ such that 
${\rm index}(\Sigma_i)\leq 1$ and  ${\rm area}_{g_i}(\Sigma_i)\rightarrow \infty$. 

Since ${\rm index}(\Sigma_i)\leq 1$, it follows that there is $p=(x,t)\in M$ such that  the sequence of  surfaces $\Sigma_i$ converges locally smoothly  to a minimal lamination $\Lambda$  in $M\setminus \{p\}$ for the metric $g=g_\Sigma+dt^2$. This lamination can be extended to $M$. This follows for instance from Theorem 1.5 of \cite{chodosh-ketover-maximo}.

We denote by $\Lambda'$ the extended lamination, so that $\Lambda'$ is  closed in $M$. Let $\eta\leq \eta'$ be such that $\Lambda'$ intersects $\Sigma_\eta$, $\Sigma_{\eta'}$, and is disjoint from $\cup_{t<\eta} \Sigma_t$ and $\cup_{t>\eta'} \Sigma_t$.

Suppose that $\eta<\eta'$. Since ${\rm index}(\Sigma_i)\leq 1$, it follows that there is a neighborhood $\Omega$ of either 
$\Sigma_\eta$ or $\Sigma_{\eta'}$ such that $\Sigma_i\cap \Omega$ is stable.  We can suppose that $\Omega$ is a neighborhood of 
$\Sigma_\eta$. Since $\Sigma_\eta$ is a minimal surface, it follows by the maximum principle that $\Sigma_\eta\subset \Lambda'$.
By the curvature estimates for stable minimal surfaces, there is a neighborhood $\Omega'\subset \Omega$ of $\Sigma_\eta$ such that for sufficiently large $i$, any $p\in \Sigma_i\cap \Omega'$ has a neighborhood in $\Sigma_i$ that is a graph on a disk of uniform radius in $\Sigma_\eta$. Since the surface $\Sigma_i$ is connected, this implies that $\Sigma_i$ is a graph on $\Sigma_\eta$ for sufficiently large $i$ and hence converges smoothly to $\Sigma_\eta$. It follows that $\{\Sigma_i\}$ has bounded area, which is a contradiction.

Suppose that $\eta=\eta'$. Then $\Lambda'=\Sigma_\eta$. Since ${\rm index}(\Sigma_i)\leq 1$, it follows that there is $p\in \Sigma_\eta$
such that for any $r>0$, the surfaces $\Sigma_i$ satisfy uniform curvature estimates in the complement of $B_{r/2}(p)$. Hence the sequence $\Sigma_i$ converges smoothly locally graphically to $\Sigma_\eta$ in $M \setminus B_{r/2}(p)$. Denote by $k(i)$  the number of local graphs. Since $\Sigma_i$ is an  embedded  surface these local graphs do not intersect.
 
 It follows by Theorem 1.17 of \cite{chodosh-ketover-maximo} that there is $r>0$ such that at most one component $\Sigma_i'$ of $\Sigma_i\cap B_r(p)$ has curvature blowing-up. The intersection of $\Sigma_i'$ with $\partial B_{r/2}(p)$ has at most $m$ components, where $m$ does not depend on $i$. The components
$\{\Sigma_i'\}$ have bounded area.    In the case that 
$k(i)>m$, there is a   component of $\Sigma_i\cap \partial B_r(p)$ that bounds  a disk component $\Sigma_i''$ of $\Sigma_i\cap B_r(p)$ with uniformly bounded curvature. 
Hence, if $k(i)>m$, $\Sigma_i$ contains a graph on $\Sigma_\eta$. Since the surface $\Sigma_i$ is connected, $\Sigma_i$ is a graph converging  smoothly to $\Sigma_\eta$. It follows that $\{\Sigma_i\}$ has bounded area, and this is a contradiction. Therefore $k(i)\leq m$. Since  $\Sigma_i'$ has bounded area and there are no  disk components of $\Sigma_i\cap B_r(p)$, it follows that $\{\Sigma_i\}$ has bounded area. This is a contradiction, which finishes the proof of the proposition.

 \end{proof}

 Let $f:[0,1]\rightarrow \mathbb{R}$ be a smooth function such that $f(t)=f(1-t)$ for $t\in [0,1]$, $f(0)=f(1)=0$, $f'(0)=f'(1)=0$, $f''(0)=f''(1)>0$, and $f'(t)>0$ for $0< t<1/2$. We define the metric 
$g_\eta= [1+\eta f(t)]^2(g_\Sigma+dt^2)$ on $M=\Sigma\times [0,1]$, $\eta>0$, where $g_\Sigma$ is a Riemannian metric on $\Sigma$. Notice that $\partial M$ is a strictly stable minimal surface in the metric $g_\eta$, and there is  $\eta>0$  such that $g_\eta\in U$. Therefore there is a bumpy metric  $g'\in U$  such that $\partial M$ is strictly stable for $g'$. 

By applying the Proposition \ref{product} to $g'$, it follows that the virtual embeddedness property cannot be replaced by embeddedness in the statement of Theorem \ref{boundary.case} .

\section{Rigidity for the scalar curvature}


The classification of closed oriented  three-manifolds that admit metrics of positive scalar curvature   follows from the proof of  the geometrization  conjecture by Perelman \cite{perelman}.  
The manifold is diffeomorphic to a connected sum of spherical space forms and copies of $S^2\times S^1$:
$$
  S^3 \# (S^3/\Gamma_1) \#\cdots (S^3/\Gamma_j)\# (S^2 \times S^1) \# \cdots \# (S^2\times S^1).
$$


 In the case of compact manifolds with boundary, a classification was proven by Carlotto and Li \cite{carlotto-li}. We are interested in positive scalar curvature oriented compact three-dimensional manifolds with strictly stable minimal boundary. The components of the boundary are topological spheres, and by extending the manifold one can perturb the metric so the boundary becomes strictly mean-convex. Then by Theorem 1.1 of \cite{carlotto-li}, the manifold is diffeomorphic to a connected sum
 $$
 B_1 \# \cdots \# B_{j'} \# (S^3/\Gamma_1) \# \cdots \# (S^3/\Gamma_j)\# (S^2 \times S^1) \# \cdots \# (S^2\times S^1),
 $$
 where $B_i$ is a three-dimensional disk for any $1\leq i\leq j'$.

\begin{prop}\label{area.surface}
Let $(\Omega^3,g)$ be a compact oriented Riemannian manifold, possibly with boundary, with scalar curvature $R_g\geq 6$ and such that $\partial M$ is a strictly stable minimal surface. If $g$ is a bumpy metric, then there is a smooth, closed, two-sided, embedded minimal surface $\Sigma'\subset {\rm int}(\Omega)$, with  ${\rm area}_g(\Sigma')\leq 4\pi$, ${\rm index}(\Sigma')\leq 1$, or a smooth, embedded minimal projective plane 
$\Sigma\subset {\rm int}(\Omega)$ with ${\rm area}_g(\Sigma)\leq 2\pi/3$ and stable oriented two-cover $\tilde{\Sigma}$.
\end{prop}

\begin{proof}
The compact manifold $\Omega$ is diffeomorphic to a connected sum $$
 B_1 \# \cdots \# B_{j'} \# (S^3/\Gamma_1) \# \cdots \# (S^3/\Gamma_j)\# (S^2 \times S^1) \# \cdots \# (S^2\times S^1),
 $$
 where $B_i$ is a three-dimensional disk for any $1\leq i\leq j'$.

 If there is a nonseparating two-sided surface in ${\rm int}(\Omega)$, 
 by minimizing the area in its homology class  we would find a two-sided, stable minimal surface $\Sigma'\subset {\rm int}(\Omega)$. The proposition is proved in this case since ${\rm area}(\Sigma')\leq 4\pi/3$ (for instance, Proposition A.1 of \cite{marques-neves-rigidity-spheres}).
 
 Hence we can suppose that  $\Omega$ is diffeomorphic to a connected sum
  $$
 B_1 \# \cdots \# B_{j'} \# (S^3/\Gamma_1) \# \cdots \# (S^3/\Gamma_{j}).
 $$
 Denote by  $\Omega_1$  the topological manifold obtained by gluing a three-ball with identification of the boundary to each connected component
 of $\partial \Omega$.

 Suppose that it is possible to find two disjoint embedded minimal projective planes $\Sigma^{'}, \Sigma^{''}$ in ${\rm int}(\Omega)$ with unstable oriented two-covers 
 $\tilde{\Sigma}^{'}, \tilde{\Sigma}^{''}$. Then there is a neighborhood $V'$ of $\Sigma^{'}$ such that $\partial V'$ has mean curvature vector pointing strictly away from $\Sigma^{'}$, and similarly there is a neighborhood $V''$ of $\Sigma^{''}$ such that $\partial V''$ has mean curvature vector pointing strictly away from $\Sigma^{''}$. By minimizing the area in the homology class of $\partial V'$ inside $\Omega \setminus (V'\cup V'')$, and 
 using that $\partial V''$ is barrier, it follows that there is a stable, embedded, two-sided minimal surface in ${\rm int}(\Omega)$. The proposition is proved in this case.
 
 Suppose then that there is an embedded minimal projective plane $\Sigma'$ in ${\rm int}(\Omega)$ with unstable oriented two-cover, such that for a neighborhood $V'$ of $\Sigma'$ as in the previous paragraph there is no embedded minimal projective plane $\Sigma''$ in $\Omega\setminus V'$ with unstable oriented two-cover. The surface $\partial V'$ is a topological sphere.  Notice that any two-sided, embedded, closed surface in $\Omega\setminus V'$ separates. 
 
Let $\mathcal{I}(\Sigma)$ denote the isotopy class of a surface $\Sigma$ in $\Omega\setminus V'$.   We minimize the area in the isotopy class 
 $\mathcal{I}(\partial V')$ of $\partial V'$ by using  \cite{meeks-simon-yau}. We can suppose that in  ${\rm int}(\Omega) \setminus V'$ there are no two-sided stable minimal surfaces or embedded minimal projective planes since otherwise the proposition is proved. This is because an embedded minimal projective plane $\Sigma$ would have stable oriented two-cover $\tilde{\Sigma}$, hence ${\rm area}(\tilde\Sigma)\leq 4\pi/3$.  By Theorem 1 of \cite{meeks-simon-yau} it follows that a minimizing sequence $\{\Gamma_i\}_i\subset \mathcal{I}(\partial V')$ converges as Radon measures
 to $n_1\Sigma^{(1)}+\cdots +n_r\Sigma^{(r)}$, where $n_k$ is a nonnegative integer and $\Sigma^{(k)}$ is a component of $\partial \Omega$ for $1\leq k \leq r$.

Let   $\{\Gamma_i\}_i$ in  $\mathcal{I}(\partial V')$ be a minimizing sequence. Then it follows from \cite{meeks-simon-yau}, after passing to a subsequence, that there is a sequence of surfaces $\{\tilde{\Gamma}_i\}_i$ such that $\tilde{\Gamma}_i$ is obtained from $\Gamma_i$ by performing a number of $\gamma$-reductions and so that $\tilde{\Gamma}_i$ is isotopic to 
$$
\Gamma_i' \cup \left(\cup_{k=1}^r \Gamma_{i,k}\right),
$$ 
where $\Gamma_{i,k}$ is a disjoint union of $n_k$ surfaces parallel to $\Sigma^{(k)}$ for any $1\leq k \leq r$ and $\lim_{i\rightarrow \infty} {\rm area}_g(\Gamma_i')=0.$

   

   It follows from Section 2 of \cite{meeks-simon-yau} that for sufficiently large $i$ any component of $\Gamma_i'$ bounds a three-ball. Since $\Sigma'$ is a projective plane disjoint from $\Gamma_i'$, it must be contained in the complement of any of these balls. A parallel surface to $\Sigma^{(k)}$, which is of the form $\{x\in \Omega: d_g(x, \Sigma^{(k)})=r\}$, in some tubular neighborhood of $\Sigma^{(k)}$, bounds a three-ball in the extended manifold $\Omega_1$ whose complement contains $\Sigma'$.  We proved that any component of $\tilde{\Gamma}_i$ bounds a three-ball in $\Omega_1$ whose complement contains $\Sigma'$.  If $\tilde{\Gamma}_i$ is obtained by a $\gamma$-reduction of a closed surface $\tilde{\Gamma}_{i,1}$ with spherical components, then by the topological  claim of Theorem   \ref{index.surfaces}
   applied to $\Omega'=\Omega_1$   the same is true for any component of $\tilde{\Gamma}_{i,1}$.
 Since $\tilde{\Gamma}_i$ is obtained from $\Gamma_i$ by a number of $\gamma$-reductions,  it follows by induction that $\Gamma_i$ bounds a three-ball whose complement contains $\Sigma'$.  
Since $\Gamma_i \in \mathcal{I}(\partial V')$,  we proved that $\Omega_1\setminus V'$ is diffeomorphic to a three-ball.
 This implies that $\Omega_1$ is diffeomorphic to $\mathbb{RP}^3$.

In this case $\Omega$ is diffeomorphic to $\mathbb{RP}^3\setminus (\cup_{i=1}^{j'} B_i')$, where  $\{B_i'\}_i$ is a disjoint set of balls in 
$\mathbb{RP}^3$. The surface $\partial B_i'$ is a strictly stable minimal sphere in the metric $g$. 

Let $\Omega'$ be the universal cover of $\Omega$, and $\pi:\Omega'\rightarrow \Omega$ be the covering map. 
This map is the restriction of the covering map $\pi:S^3 \rightarrow \mathbb{RP}^3$ to $S^3 \setminus \{(\cup_{i=1}^{j'} B_{i,1}')\cup (\cup_{i=1}^{j'} B_{i,2}')\}$, with $\pi^{-1}(B_{i}')= B_{i,1}' \cup B_{i,2}'$ and so that $\{B_{i,j}'\}$ is a disjoint set of balls in $S^3$.  Since $g$ is a bumpy metric, as in the proof of Theorem \ref{connected.minimal} it follows that $\Omega'=\Omega_1'\cup \cdots \cup \Omega'_q$, such that  $\{\Omega_i'\}$ are compact domains, possibly with boundary, with disjoint interiors, $\partial \Omega_i'$ is a union of strictly stable minimal spheres, and so that there is no strictly stable minimal surface  in ${\rm int}(\Omega_i')$ for $1\leq i\leq q$.
The map $\pi: \Omega'\rightarrow \Omega$ is a local isometry with the pull-back metric, hence it is area-nonincreasing. We cannot apply Theorem \ref{connected.minimal} to $\pi$ as $\Omega$ could have boundary. We will adapt the proof to this case by using the cylindrical extension construction of Song \cite{song-existence}.

Let $\tilde{\Sigma}$ be a component of $\partial \Omega_i'$. If $\tilde{\Sigma}$ is a component of $\partial \Omega'$, then $\pi_{\#}(\tilde{\Sigma})$ is a component of $\partial \Omega$ with multiplicity one since the two-sphere is simply-connected. If $\tilde{\Sigma}\subset {\rm int}(\Omega')$,  $\pi_{\#}(\tilde{\Sigma})\in \tilde{\mathcal{Z}}_2(\Omega,\mathbb{Z})$ and we can do a constrained minimization in its homology class as in the proof of Theorem
\ref{connected.minimal}. Since there are no stable two-sided embedded minimal hypersurfaces in ${\rm int}(\Omega)$, it follows that for any $\delta>0$
there is a map continuous in the ${\bf F}$-metric $\Phi:[0,1]\rightarrow \tilde{\mathcal{Z}}_2(\Omega,\mathbb{Z})$ such that $\Phi(0)=-\pi_{\#}(\tilde{\Sigma})$, ${\rm support}(\Phi(1))\subset \partial \Omega$, and ${\bf M}_g(\Phi(t))\leq {\bf M}_g(\pi_{\#}(\tilde{\Sigma}))+\delta$ for any $t\in [0,1]$.

Let $\mathcal{C}(\Omega)$ be the cylindrical extension
$$
\Omega \cup (\partial \Omega \times [0,\infty))
$$
of $\Omega$. This manifold is defined in \cite{song-existence}. The metric $h$ on $\mathcal{C}(\Omega)$ is the Lipschitz metric such that
$h_{|\Omega}=g$ and $h_{|\partial \Omega \times [0,\infty)}=g_{|\partial \Omega} + dt^2$.

Define $\mathcal{C}(\Omega)_\eta$ as $\Omega\cup (\partial \Omega\times [0,\eta])$. If $T\in \tilde{\mathcal{Z}}_2(\Omega,\mathbb{Z})$ is such that ${\rm support}(T)\subset \partial \Omega$, then  by translation in the $t$ direction there is a map  $\Phi:[0,1]\rightarrow \tilde{\mathcal{Z}}_2(\mathcal{C}(\Omega)_\eta,\mathbb{Z})$ such that $\Phi(0)=T$, ${\rm support}(\Phi(1))\subset \partial \mathcal{C}(\Omega)_\eta$ and
${\bf M}_h(\Phi(s))={\bf M}_h(T)$ for any $s\in [0,1]$. Let $\mathcal{Z}_2(\mathcal{C}(\Omega)_\eta, \partial \mathcal{C}(\Omega)_\eta,\mathbb{Z})$ be the space of relative flat boundaries with integer coefficients of $\mathcal{C}(\Omega)_\eta$. The map $\Phi$
induces a map $\Phi':[0,1]\rightarrow \mathcal{Z}_2(\mathcal{C}(\Omega)_\eta, \partial \mathcal{C}(\Omega)_\eta,\mathbb{Z})$ such that
$\Phi'(0)=T$, $\Phi'(1)=0$, and ${\bf M}_h(\Phi'(s))={\bf M}_h(T)$ for any $s\in [0,1)$. The map $\Phi'$ is continuous in the flat topology and has no concentration of mass. 

We proved that for any component $\tilde{\Sigma}$ of $\partial \Omega_i'$ and $\delta>0$, using interpolation there is a map
$\Phi_{\tilde{\Sigma}}:[0,1]\rightarrow \mathcal{Z}_2(\mathcal{C}(\Omega)_\eta, \partial \mathcal{C}(\Omega)_\eta,\mathbb{Z})$ continuous in the ${\bf F}$-metric such that $\Phi_{\tilde{\Sigma}}(0)=-\pi_{\#}(\tilde{\Sigma})$, $\Phi_{\tilde{\Sigma}}(1)=0$, and 
${\bf M}_h(\Phi_{\tilde{\Sigma}}(s))\leq {\bf M}_g(\tilde{\Sigma})+\delta$ for any $s\in [0,1]$. If $\tilde{\Sigma}$ is a component of $\partial \Omega'$, the map $\Phi_{\tilde{\Sigma}}$ is a sweepout of $\pi_{\#}(\tilde{\Sigma})\times [0,\eta]$. Notice that there are two components $\tilde{\Sigma}_1$ and $\tilde{\Sigma}_2$ with $\pi(\tilde{\Sigma}_i)=\Sigma$ for any component $\Sigma$ of $\partial\Omega$. We will use that $\pi_{\#}(\Omega')=2 \cdot \Omega$.

Let $\mathcal{C}(\Omega_i')$ be the cylindrical extension of $\Omega_i'$, and  $h_i$  be the Lipschitz metric on $\mathcal{C}(\Omega_i')$ such that
$(h_i)_{|\Omega_i'}=g$ and $(h_i)_{|\partial \Omega_i' \times [0,\infty)}=g_{|\partial \Omega} + dt^2$. Song \cite{song} constructs a sequence of smooth Riemannian metrics $h_{i,k}$
on  $\mathcal{C}(\Omega_i')_{\eta_k}$, $\eta_k\rightarrow\infty$, such that $(\mathcal{C}(\Omega_i')_{\eta_k},h_{i,k})$ converges to $(\mathcal{C}(\Omega_i'),h_i)$.  By definition of the metrics $h_{i,k}$  there is a compact set $\Lambda\subset {\rm int}(\Omega_i')$  such that $\mathcal{C}(\Omega_i')_{\eta_k}\setminus \Lambda$ is foliated by smooth spheres with mean curvature vector pointing strictly away from $\Lambda$. The metric $h_{i,k}$  can be extended smoothly to a metric $h_{i,k}'$ on the manifold $\Omega_{i,k}$ obtained by gluing three-balls to $\mathcal{C}(\Omega_i')_{\eta_k}$ with boundary identification such that the foliation extends  to a strict mean-convex foliation that glues smoothly to a family of concentric geodesic spheres in each ball.

The manifold $\Omega_{i,k}$ is diffeomorphic to a three-sphere.  Hence we can do min-max with smooth one-parameter sweepouts by two-spheres. Therefore there is an embedded minimal two-sphere $\Sigma_{i,k}\subset (\Omega_{i,k},h_{i,k}')$ such that  ${\rm index}(\Sigma_{i,k})\leq 1$.
By doing as in \cite{song}, the area of $\Sigma_{i,k}$ is uniformly bounded and there is $\{k'\}\subset \{k\}$ such that $\Sigma_{i,k'}$ converges to an embedded minimal sphere $\Sigma_i'\subset {\rm int}(\Omega_i')$. There is a similar construction in \cite{wang-zhou} (Section 8) which uses free boundary minimal surfaces as in \cite{song}.

The extension of $h_{i,k}$ to $h_{i,k}'$ can be proved by extending $h_{i,k}$ to any metric and then doing a conformal deformation as in the following claim.

\medskip

{\bf Claim:} Let $g$ be a smooth Riemannian metric on $\Sigma^n \times [0,1]$ and $0<\eta<1$ such that the mean curvature vector of $\Sigma\times \{t\}$ for $t<\eta$ or $t>1-\eta$ is positive. There is a smooth function $f=f(t)$ such that the mean curvature vector of $\Sigma\times \{t\}$ in the $\tilde{g}=e^{2f}g$ metric is positive for any $t\in [0,1]$,  $\tilde{g}=\lambda g$  for $t<\eta/2$, $\lambda\in \mathbb{R}_+$, and $\tilde{g}=g$ for $t>1-\eta/2$.

\medskip

Let $\Sigma_t=\Sigma\times \{t\}$. Then
$$
H_{\Sigma_t,\tilde{g}}=e^{-f}(H_{\Sigma_t,g}+ n\, \partial_\nu f),
$$
where $\nu$ is the unit normal of $\Sigma_t$ in the $g$-metric such that $g(\nu,\partial_t)>0$ and $H_{\Sigma_t,g}=-g(\vec{H}_{\Sigma_t,g},\nu)$. Hence
$$
H_{\Sigma_t,\tilde{g}}=e^{-f}(H_{\Sigma_t,g}+ n\, f'(t)g( \nabla_gt,\nu)).
$$
There is a constant $c>0$ such that $g(\nabla_gt,\nu)\geq c$. 

If $f'(t)\geq 0$ for any $t\in [0,1]$, then $H_{\Sigma_t,\tilde{g}}>0$ for $t<\eta$ or $t>1-\eta$. If 
$$
h=\sup_{t\in [\eta,1-\eta],x\in \Sigma} |H_{\Sigma_t,g}|(x,t),
$$
and if $f'(t)$ is a constant $\kappa>0$ for $t\in [\eta,1-\eta]$ where  $\kappa \geq 2h/(n\, c)$ then
$H_{\Sigma_t,\tilde{g}}>0$ for $t\in [\eta,1-\eta]$. Therefore the claim is proved with any nondecreasing smooth function $f=f(t)$ such that $f=0$ for $t>1-\eta/2$, $f$ is constant for $t<\eta/2$ and $f'(t)=\kappa \geq 2h/(n\, c)$ for $t\in [\eta,1-\eta]$.

\medskip

Therefore there is an embedded minimal sphere $\Sigma_i'$ of index at most one in ${\rm int}(\Omega_i')$. This can also be proved with the more general Theorem 6.2 of \cite{ketover-liokumovich-song}.

Since $R_{\pi^{*}(g)}\geq 6$, by Hersch's trick (for instance, Proposition A.1 of \cite{marques-neves-rigidity-spheres}),  it follows that ${\rm area}_{\pi^{*}(g)}(\Sigma_i')\leq 4\pi$.

By using the concatenating method of Theorem \ref{connected.minimal}, since the map $\pi$ is area-nonincreasing, it follows that for any $\delta>0$ there is a map 
$$
\Phi:[0,1]\rightarrow \mathcal{Z}_2(\mathcal{C}(\Omega)_\eta, \partial \mathcal{C}(\Omega)_\eta,\mathbb{Z}),
$$
continuous in the ${\bf F}$-metric, such that $\Phi(0)=0$, $\Phi(1)=0$, ${\bf M}_h(\Phi(t))\leq 4\pi+\delta$ for any $t\in [0,1]$, and so that
$\Phi$ is a relative sweepout of $2\cdot \mathcal{C}(\Omega)_\eta$.

Let  $\{\mathcal{C}(\Omega)_{\eta_i}\}$ be an exhaustion of $\mathcal{C}(\Omega)$,  $\eta_i \rightarrow \infty$.  Then
 $$
 W^{(2)}(\mathcal{C}(\Omega),h)=\lim_{i\rightarrow \infty} W^{(2)}(\mathcal{C}(\Omega)_{\eta_i},\partial \mathcal{C}(\Omega)_{\eta_i},h)\leq 4\pi.
 $$
 By Song \cite{song-existence},
 $$
 W^{(2)}(\mathcal{C}(\Omega),h) =\sum_{i=1}^k n_i{\rm area}_g(\Sigma_i),
 $$
 where $n_i\in \mathbb{N}$ and $\{\Sigma_1,\dots, \Sigma_k\}$ are the disjoint components of a smooth, embedded, closed minimal hypersurface contained in ${\rm int}(\Omega)$. 
 
 If $\Sigma_i$ is a two-sided surface, then ${\rm area}_g(\Sigma_i)\leq 4\pi$ and ${\rm index}(\Sigma_i) \leq 1$ by the index estimates (\cite{marques-neves-index}).
  If $\Sigma_i$ is a one-sided surface the multiplicity $n_i$ is even, hence ${\rm area}_g(\Sigma_i)\leq 2\pi$. Suppose that the oriented two-cover $\tilde{\Sigma}_i$ of $\Sigma_i$ is unstable. By Proposition \ref{index.one.boundary} and Proposition \ref{no.stable.one-sided.case}, there is a two-sided, closed, embedded minimal hypersurface $\Sigma' \subset {\rm int}(\Omega)$, ${\rm index}(\Sigma')\leq 1$, with ${\rm area}_g(\Sigma')<2\, {\rm area}_g(\Sigma_i) \leq 4\pi$. If $\tilde{\Sigma}_i$ is stable, then $\tilde{\Sigma}_i$ is a topological sphere with ${\rm area}_g(\tilde{\Sigma}_i)\leq 4\pi/3$. Hence $\Sigma_i$ is a projective plane such that  ${\rm area}_g(\Sigma_i)\leq 2\pi/3$.
  This proves the proposition in this case. 
 
 Suppose that there is no embedded, minimal projective plane in ${\rm int}(\Omega)$ with unstable oriented two-cover. If there is an embedded minimal projective plane with stable oriented two-cover the proposition is proved. Hence we can suppose that there are no embedded, minimal projective planes, and that there are no two-sided stable embedded minimal surfaces in ${\rm int}(\Omega)$.  Let $\Omega_1$ be as before. If $j\geq 2$, there is an embedded two-sphere $\Sigma\subset {\rm int}(\Omega)$ that separates 
 $\Omega_1$ into two components $W,W'$, such that $W$ and $W'$ are not simply connected.
 
 We minimize the area in the isotopy class 
 $\mathcal{I}(\Sigma)$ of $\Sigma$ in $\Omega$  by using  \cite{meeks-simon-yau}. Since in ${\rm int}(\Omega) $ there are no two-sided stable minimal surfaces or embedded minimal projective planes, by Theorem 1 of \cite{meeks-simon-yau} it follows that a minimizing sequence $\{\Gamma_i\}_i\subset \mathcal{I}(\Sigma)$ converges as Radon measures
 to $n_1\Sigma^{(1)}+\cdots +n_r\Sigma^{(r)}$, where $n_k$ is a nonnegative integer and $\Sigma^{(k)}$ is a component of $\partial \Omega$ for $1\leq k \leq r$.
 
Let   $\{\Gamma_i\}_i$ in  $\mathcal{I}(\Sigma)$ be a minimizing sequence. Then, as before, by  passing to a subsequence there is a sequence of surfaces $\{\tilde{\Gamma}_i\}_i$ such that $\tilde{\Gamma}_i$ is obtained from $\Gamma_i$ by performing a number of $\gamma$-reductions and so that $\tilde{\Gamma}_i$ is isotopic to 
$\Gamma_i' \cup \left(\cup_{k=1}^r \Gamma_{i,k}\right)$, 
where $\Gamma_{i,k}$ is a disjoint union of $n_k$ surfaces parallel to $\Sigma^{(k)}$ for any $1\leq k \leq r$ and $\lim_{i\rightarrow \infty} {\rm area}_g(\Gamma_i')=0.$  Then any component of $\tilde{\Gamma}_i$ bounds a three-ball in $\Omega_1$.
The claim implies by induction that $\Gamma_i$, and hence $\Sigma$, bounds a three-ball in $\Omega_1$. This is a contradiction which implies $j\leq 1$.

Therefore $\Omega$ is diffeomorphic either to $(S^3/\Gamma)\setminus (\cup_{i=1}^{j'} B_i')$ or to $S^3\setminus (\cup_{i=1}^{j'} B_i')$, where  $\{B_i'\}_i$ is a disjoint set of balls. By passing to the universal cover  the proposition is proved   as in the case of $\mathbb{RP}^3\setminus (\cup_{i=1}^{j'} B_i')$. This finishes the proof of the proposition.

\end{proof}

\begin{thm}\label{scalar-rigidity}
Let $(M^3,g)$ be a compact oriented Riemannian manifold, possibly with boundary,  with scalar curvature $R_g\geq 6$ and such that $\partial M$ is a strictly stable minimal surface. Suppose that $(N^3,h)$ is a closed  Riemannian manifold such that there is no closed, two-sided, embedded, stable minimal surface $\Sigma\subset N$ with ${\rm area}_h(\Sigma)\leq 4\pi/3+\rho$, and 
  $W^{(k)}(N,h)\geq 4\pi+\rho$, where $k\in \mathbb{N}$,  $\rho>0$.  Then there is no area nonincreasing Lipschitz map 
 $f:(M,g)\rightarrow (N,h)$, with $f_{\#}(\Sigma')=0$ for any connected component $\Sigma'$ of $\partial M$, and ${\rm deg}(f)=k$. 
\end{thm}

\begin{proof}

Suppose, by contradiction, that there is a map  $f:(M,g)\rightarrow (N,h)$ that is area-nonincreasing, with $f_{\#}(\Sigma')=0$ for any connected component $\Sigma'$ of $\partial M$, of ${\rm deg}(f)=k\neq0$.
By perturbing the metric $g$ and modifying  $h$ to $c\cdot h$, $c\in\mathbb{R}$, we can suppose that $g$ is a bumpy metric.

If $\Sigma\subset (M,g)$ is a stable and orientable connected closed minimal surface, then $\Sigma$ is a two-sphere with ${\rm area}_g(\Sigma)\leq 4\pi/3$ (Proposition A.1 of \cite{marques-neves-rigidity-spheres}).   If $\Sigma\subset (M,g)$ is a one-sided connected closed minimal surface with stable oriented two-cover $\tilde\Sigma$, then $\tilde\Sigma$ is a two-sphere with ${\rm area}_g(\tilde\Sigma)\leq 4\pi/3$ and $\Sigma$ is a projective plane.   By the curvature estimates of \cite{schoen} for stable surfaces, and since $g$ is a bumpy metric, the set $\Lambda$ of connected closed minimal surfaces that are either stable and two-sided or one-sided with stable oriented two-cover in ${\rm int}(M)$  is finite. Therefore there is a  set 
$\tilde{\Lambda}\subset \Lambda$ of mutually disjoint closed minimal surfaces, that is not contained in any other set with this property.  

Let $\Omega_1,\dots, \Omega_p$ be the connected components of the complement of $\cup_{\Sigma\in \tilde{\Lambda}}\Sigma$.
 It follows that, for any $1\leq j\leq p$, there is no  closed minimal surface that is either stable and two-sided or one-sided with stable oriented two-cover in   ${\rm int}(\Omega_j)$. If $\Sigma\in \tilde{\Lambda}$ is a one-sided surface we denote by $\tilde{\Sigma}$ its oriented two-cover, otherwise $\tilde{\Sigma}=\Sigma$.

In the case that $M$ has a boundary, $\partial M=\sum_{i=1}^tS_i \in \mathcal{Z}_2(M,\mathbb{Z})$  where $S_i$ is the integral current of
a stable minimal sphere with ${\rm area}(S_i)\leq 4\pi/3$ and with the induced orientation.

Let $1\leq j\leq p$. If  $\tilde{\Omega}_j$ is  the metric closure of $\Omega_j$, then $\tilde{\Omega}_j$ is a three-dimensional compact manifold possibly with boundary such that  any  connected component of $\partial \tilde{\Omega}_j$ can be identified with one of the components of $\partial M$ or with a surface $\tilde{\Sigma}$, $\Sigma \in \tilde{\Lambda}$,
 and ${\rm int}(\tilde{\Omega}_j)=\Omega_j$.  By Proposition \ref{area.surface},  there is a  closed, two-sided, embedded minimal surface 
 $\Sigma_j'\subset {\rm int}(\tilde{\Omega}_j)$,  ${\rm area}_g(\Sigma_j')\leq 4\pi$, ${\rm index}(\Sigma_j')= 1$.

The map $f:\Omega_j\rightarrow N$ induces a Lipschitz map $f:\tilde{\Omega}_j\rightarrow N$ that is area-nonincreasing. 
 Since $R_g\geq 6$, a component $\Sigma$ of $\partial \tilde{\Omega}_j$ is a stable minimal two-sphere  with ${\rm area}_g(\Sigma)\leq 4\pi/3$ and hence ${\rm area}_h(f_{\#}(\Sigma))\leq 4\pi/3$.  If $\tilde{\Sigma}$ is a component of $\partial\tilde{\Omega}_j$ that is an oriented two-cover of a one-sided surface in $M$, then $f_{\#}(\tilde{\Sigma})=0$.

Since  there is no two-sided, embedded, closed, stable minimal surface $\Sigma$ for the metric $h$ with ${\rm area}_h(\Sigma)\leq 4\pi/3$, we can do the constrained minimization and  the   concatenation as in the proof of Theorem \ref{connected.minimal}. It follows that 
$W^{(k)}(N,h)\leq 4\pi$ for $k={\rm deg}(f)$.  This is a contradiction, which finishes the proof of the theorem.



\end{proof}


Since closed oriented three-manifolds are spin, the following theorem is a consequence of Llarull's theorem \cite{llarull} if $(N,h)=(S^3,\overline{g})$,  and the area-nonincreasing map $f$  is smooth.

\begin{thm}\label{scalar-rigidity.2}
Let $(M^3,g)$ be a closed, oriented, Riemannian manifold  with  $R_g\geq 6$. Suppose that $(N^3,h)$ is a closed,  oriented Riemannian manifold such that  there is no closed, two-sided, embedded, stable minimal surface $\Sigma\subset N$ with ${\rm area}_h(\Sigma)\leq 4\pi/3+\rho$, $\rho>0$,
and such that  $W^{(k)}(N,h)\geq 4\pi$. Let $f:(M,g)\rightarrow (N,h)$ be a  Lipschitz map that is  area-nonincreasing of nontrivial degree $k$. Then $(M,g)$ is isometric to the unit sphere and $f$ is a smooth Riemannian  isometry.
  \end{thm}
  
  \begin{proof}
  
 
 Let $g_i$ be a sequence of bumpy Riemannian metrics on $M$ converging smoothly to $g$ as in \cite{white-bumpy}.  Let $\Lambda$, $\tilde{\Lambda}$ and 
 $\{\Omega_1,\dots, \Omega_p\}$ be as in the proof of Theorem \ref{scalar-rigidity} for $g=g_i$. Since the map
$f:(M,g)\rightarrow (N,h)$ is area-nonincreasing, there is a sequence $\delta_i\rightarrow 0$, $\delta_i>0$, such that
$$
|df_x(v_1) \wedge df_x(v_2)|_{h}\leq (1+\delta_i) |v_1 \wedge v_2|_{g_i}
$$
for almost any $x\in M$ where $df_x$ is defined,  and $\{v_j\}_{j=1}^2\subset T_xM$.
By following the proof of Theorem \ref{scalar-rigidity},  there is $1\leq j\leq p$ such that
there is no closed, two-sided, unstable $g_i$-minimal surface in ${\rm int}(\tilde{\Omega}_j)$ with area bounded by 
$4\pi(1+\delta_i)^{-1}-\delta_i$.
This is because if not then one can use the map
$f$
 to construct a sweepout of $(N,h)$ with area bounded strictly  by $4\pi$. 
By Proposition \ref{index.one.boundary}
and Proposition \ref{no.stable.one-sided.case}, there is no closed, one-sided, embedded $g_i$-minimal surface in ${\rm int}(\tilde{\Omega}_j)$ with unstable oriented two-cover such that the area of the two-cover is bounded by $4\pi(1+\delta_i)^{-1}-\delta_i$. Hence, by min-max, $W^{(1)}(\tilde{\Omega}_j,g_i)>4\pi(1+\delta_i)^{-1}-\delta_i$.

We denote $(\Omega_i',g_i)=(\Omega_j,g_i)$.
Then 
any component of $\partial \tilde{\Omega'_i}$ can be identified with  a surface in $\tilde{\Lambda}$,
where $\tilde{\Omega'_i}$ denotes the metric closure of $\Omega_i'$. 

Let $0<\eta<2\pi/3$. By Proposition \ref{area.surface}, since $R_{g}\geq 6$, and $g_i\rightarrow g$, there is a closed, embedded, two-sided, $g_i$-minimal surface in ${\rm int}(\tilde{\Omega'_i})$, of Morse index one, and with area bounded by $4\pi+\eta$. It follows from the  Proposition \ref{no.stable.two-sided.case} that $W^{(1)}(\tilde{\Omega_i'},g_i)\leq 4\pi+\eta$. Hence ${\bf M}_{g_i}(\partial\tilde{\Omega_i'})\leq 8\pi+2\eta$.  By the monotonicity formula for minimal surfaces, the number of components of $\partial \tilde{\Omega_i'}$ is bounded by a constant $c$. Hence we can suppose that $\partial \tilde{\Omega_i'}$ converges locally smoothly by curvature estimates of stable minimal surfaces.


It follows by the maximum principle that if $\Sigma,\Sigma'$ are limits of components of $\partial \tilde{\Omega'_i}$, then either $\Sigma=\Sigma'$ or $\Sigma$ and $\Sigma'$ are disjoint.  Suppose that components $\Sigma_{i,1}$ and $\Sigma_{i,2}$ of $\partial \tilde{\Omega'_i}$ converge smoothly to a two-sided, closed minimal surface $\Sigma \subset {\rm int}(M)$.  The surface $\Sigma$ has to be stable for the metric $g$, and hence ${\rm area}_g(\Sigma)\leq 4\pi/3$. The domain $\Omega_i'$ cannot be the region bounded by $\Sigma_{i,1}$ and $\Sigma_{i,2}$ in the neighborhood of $\Sigma$. If that is the case, there would be a sweepout of $\tilde{\Omega_i'}$ connecting $\Sigma_{i,1}$ to  $\Sigma_{i,2}$ with $g_i$-areas bounded by $4\pi/3+\eta< 2\pi$.  This is a contradiction for sufficiently large $i$, since $W^{(1)}(\tilde{\Omega_i'},g_i)>4\pi(1+\delta_i)^{-1}-\delta_i$.
It  also follows that there can be 
at most two components of $\partial \tilde{\Omega'_i}$ converging to $\Sigma$. In the case of two components, the limit of $\Omega_i'$ contains a neighborhood of $\Sigma$.

Suppose that  components $\Sigma_{i,1}$ and $\Sigma_{i,2}$ of $\partial \tilde{\Omega'_i}$ converge  to a one-sided, closed minimal surface $\Sigma \subset {\rm int}(M)$.  The surfaces cannot be two-covers of one-sided surfaces in $M$, as these surfaces would have to intersect. The surfaces cannot be one-covers of two-sided surfaces in $M$ since in that case min-max  could be used as in the previous paragraph. This is because  $\Sigma_{i,2}$ is $g_i$-stable, and hence  the two-cover of $\Sigma$ is $g$-stable and has area bounded by $4\pi/3$. Suppose $\Sigma_{i,1}$ is a two-cover of a one-sided surface and 
$\Sigma_{i,2}$ is a one-cover of a two-sided surface. Then $\Sigma_{i,2}$ is the boundary of a region $\Omega$ in a neighborhood of $\Sigma$ such that $\Sigma_{i,1}\subset \Omega$.  In this case, $\Omega_i'=\Omega \setminus \Sigma_{i,1}$.   Hence there is a sweepout of $\tilde{\Omega_i'}$ connecting 
 $\Sigma_{i,2}$ to $\Sigma_{i,1}$ with areas bounded by $4\pi/3+\eta$. This is a contradiction  as  $4\pi/3+\eta<4\pi(1+\delta_i)^{-1}-\delta_i$  for sufficiently large $i$.
 
 This implies that at most one component of $\partial \tilde{\Omega'_i}$ can converge  to a  one-sided, closed minimal surface $\Sigma \subset {\rm int}(M)$. This convergence is smooth with multiplicity one if the component is a two-cover of a one-sided surface,  or locally smooth with multiplicity two if the component is a one-cover of a two-sided surface. In any case the limit of $\Omega_i'$ contains a neighborhood of $\Sigma$.

 This proves that the limit $\Omega$ of $\Omega_i'$ is a nontrivial region of $M$. Let $\tilde{\Omega}$ be its metric closure. 
A component of $\partial \tilde{\Omega}$ can be identified with 
a two-sided $g$-stable minimal surface or a $g$-stable two-cover of a one-sided minimal surface in $M$. The Riemannian manifolds $(\tilde{\Omega_i'},g_i)$ converge smoothly to $(\tilde{\Omega},g)$.
Since $W^{(1)}(\tilde{\Omega_i'},g_i)>4\pi(1+\delta_i)^{-1}-\delta_i$, it follows that 
 $W^{(1)}(\tilde{\Omega},g)\geq 4\pi$.  A two-sided  surface $\Sigma\subset \Omega$ separates $\Omega$, otherwise there would be a two-sided, $g_i$-stable minimal surface in ${\rm int}(\tilde{\Omega_i'})$.


Notice that if  $\Sigma$ is one-sided with stable two-cover, then $W^{(1)}(\tilde{\Omega},g)\geq 4\pi$ implies $W^{(1)}(\tilde{\Omega}_\Sigma,g)\geq 4\pi$ for $\Omega_\Sigma=\Omega\setminus \Sigma$.
If $\{\Sigma_j\}_j$ is a sequence of connected, one-sided, with stable two-cover, closed, embedded, minimal surfaces, then by curvature estimates there is a sequence $\{\tilde{j}\}\subset \{j\}$ such that $\Sigma_{\tilde{j}}$ converges smoothly to a one-sided, surface $\Sigma$.  This implies that for some $j'$ the surfaces $\Sigma_{j_1}$ and $\Sigma_{j_2}$ intersect if $j_1,j_2\geq j'$.  It follows that there is a finite set $\Lambda'$
of connected, one-sided, with stable two-cover, closed, embedded, minimal surfaces, mutually disjoint, such that any such closed minimal surface intersects $\cup_{\Sigma\in \Lambda'}\Sigma$.
Let $\Omega_1=\Omega \setminus (\cup_{\Sigma\in \Lambda'}\Sigma)$,
so that there is no one-sided, with stable two-cover, closed, embedded, minimal surface in ${\rm int}(\tilde{\Omega}_1)$, and  
$W^{(1)}(\tilde{\Omega}_1,g)\geq 4\pi$. 

Let $\mathcal{I}_{\Omega_1}$ be the set of connected, open domains $\Omega'\subset {\rm int}(\tilde\Omega_1)$ such that a component of $\partial \tilde{\Omega}'$ is a component of $\partial \tilde{\Omega}_1$ or a two-sided stable minimal surface in $\Omega_1$,  and $W^{(1)}(\tilde{\Omega}',g)\geq 4\pi$. 
Notice that $\Omega_1 \in \mathcal{I}_{\Omega_1}$. Let $\{\Omega^{(j)}\}_j \subset \mathcal{I}_{\Omega_1}$ be a sequence such that
$$
{\rm vol}_g(\Omega^{(j)}) \rightarrow \inf_{\Omega'\in \mathcal{I}_{\Omega_1}} {\rm vol}_g(\Omega').
$$

We want to prove that the number of components of $\partial \Omega^{(j)}$ is bounded. The space   of smooth, two-sided, connected, stable minimal surfaces in $\tilde{\Omega}_1$ is compact in the Hausdorff metric. This follows from the curvature estimates for stable surfaces, since any such surface has area bounded by $4\pi/3$ and there is no one-sided, with stable two-cover surface in $\tilde{\Omega}_1$.  Hence, if the number of components of $\partial \Omega^{(j)}$ is unbounded, there is a sequence $\{i\}\subset \{j\}$ and components $\Sigma^{(i)}_1, \Sigma^{(i)}_2$ of $\partial \Omega^{(i)}$ such that $\Sigma^{(i)}_1, \Sigma^{(i)}_2\rightarrow \Sigma$ in the Hausdorff metric. Recall that any two-sided surface in $\Omega$ separates $\Omega$. This implies that $\Omega^{(i)}$ is the region bounded by $\Sigma^{(i)}_1$ and  $\Sigma^{(i)}_2$ in the neighborhood of $\Sigma$. Therefore 
$W^{(1)}(\tilde{\Omega}^{(i)},g)\leq 4\pi/3+\eta<4\pi$ for sufficiently large $i$. This is a contradiction.

Hence the number of components of $\partial \Omega^{(j)}$ is bounded.
 Therefore there is a sequence $\{i\}\subset \{j\}$ such that $\partial \tilde{\Omega}^{(i)}$, and hence $\tilde{\Omega}^{(i)}$,  converges smoothly.  Let $\Omega''$ be the limit of $\tilde{\Omega}^{(i)}$.
 Then $\Omega'' \in \mathcal{I}_{\Omega_1}$, and ${\rm vol}_g(\Omega'')=\inf_{\Omega'\in \mathcal{I}_{\Omega_1}} {\rm vol}_g(\Omega')$. By doing as in the previous paragraph it follows that $\Omega''$ is a nontrivial domain.
 
 We are going to analyze the closed minimal surfaces in $\Omega''$. Let $\Sigma \subset \Omega''$ be an embedded, closed minimal surface for the metric $g$. Then $\Sigma \subset \Omega_i'$ for sufficiently large $i$. Suppose that $\Sigma$ is two-sided and unstable. Then there is a neighborhood $V$ of $\Sigma$ such that the mean curvature vector of $\partial V$ points strictly away from $\Sigma$. If $i$ is sufficiently large, the mean curvature vector of $\partial V$ in the metric $g_i$ points strictly away from $\Sigma$. The surface $\Sigma$ has to separate $\Omega_i'$, otherwise there is a two-sided stable $g_i$-minimal surface in $\Omega_i'$ by area minimization.
It follows from the  proof of Proposition \ref{no.stable.two-sided.case} that for any $\xi>0$, $W^{(1)}(\tilde{\Omega_i'},g_i)\leq {\rm area}_g(\Sigma)+\xi$ for sufficiently large $i$. Hence ${\rm area}_g(\Sigma)\geq 4\pi$.

We proved that a two-sided, unstable, closed embedded minimal surface $\Sigma\subset \Omega''$ 
is such that ${\rm area}_g(\Sigma)\geq 4\pi$. If $\Sigma$ is one-sided with unstable double cover $\tilde{\Sigma}$, then
${\rm area}_g(\tilde{\Sigma})\geq 4\pi$ by using the proof of Proposition \ref{no.stable.one-sided.case}.
 
 Let $\Sigma \subset \Omega''$ be a two-sided, stable, closed embedded minimal surface. Then ${\rm area}_g(\Sigma)\leq 4\pi/3$.
 There is a foliation $\{\Sigma_t\}_{t\in [-\delta,\delta]}$ of a neighborhood of $\Sigma$ such that the mean curvature of $\Sigma_t$ has a sign for any $t\in [-\delta,\delta]$. If there is an interval $[a,b]\subset [-\delta,\delta]$ such that the mean curvature vector of $\partial(\bigcup_{t\in [a,b]}\Sigma_t)$ points strictly outward, this implies as before that 
there is a sweepout of $\tilde{\Omega_i'}$ with $g_i$-area bounded by $4\pi/3+\eta$ which is a contradiction. 
If there is an interval $[a,b]\subset [-\delta,\delta]$ such that the mean curvature vector of $\partial(\bigcup_{t\in [a,b]}\Sigma_t)$ points strictly inward, then by minimizing area in the $g_i$-metric there is a two-sided, $g_i$-stable  minimal surface in 
$(\bigcup_{t\in [a,b]}\Sigma_t)\subset {\rm int}(\tilde{\Omega_i'})$ for sufficiently large $i$. This is a contradiction. This proves that the mean curvature of $\Sigma_t$ has the same sign for any $t\in [-\delta,\delta]$.

 Let $\Sigma$ be a component of $\partial\Omega''$. Then there is a foliation $\{\Sigma_t\}_{t\in [0,\delta]}$ of a neighborhood of 
 $\Sigma$ in $\Omega''$ such that $\Sigma_0=\Sigma$ and so that  the mean curvature of $\Sigma_t$ has a sign.  As in the previous paragraph, the mean curvature of $\Sigma_t$ has the same sign for any $t\in [0,\delta]$. If there is $\tilde{t}\in (0,\delta]$
 such that $\Sigma_t$ is minimal for any $t \in [0,\tilde{t}]$, then 
 $W^{(1)}(\tilde{\Omega}''\setminus (\cup_{t\in [0,\tilde{t})}\Sigma_t),g)\geq 4\pi$. This is because the surface $\Sigma_{\tilde{t}}$ can be connected to $\Sigma$ with constant area by the foliation, and $W^{(1)}(\tilde{\Omega}'',g)\geq 4\pi$. Since in this case $\Sigma_{\tilde{t}}$ is stable degenerate, 
 $$
 \Omega_{\tilde{t}}=\tilde{\Omega}''\setminus (\cup_{t\in [0,\tilde{t})}\Sigma_t)\in \mathcal{I}_{\Omega_1}.
 $$
 This is a contradiction, since ${\rm vol}_g( \Omega_{\tilde{t}})<{\rm vol}_g(\tilde{\Omega}'')$. Hence there is a sequence $t_j>0$, $t_j\rightarrow 0$, such that either $\Sigma_{t_j}$ has mean curvature vector pointing strictly away from $\Sigma$ for any $j$ or 
 the surface $\Sigma_{t_j}$ has mean curvature vector pointing strictly towards $\Sigma$ for any $j$. Hence either
 ${\rm area}_g(\Sigma_t)<{\rm area}_g(\Sigma)$ for any $t\in [0,\delta]$, or ${\rm area}_g(\Sigma_t)>{\rm area}_g(\Sigma)$ for any $t\in [0,\delta]$.
 
 If $\Gamma$ is the set of two-sided, connected, closed stable minimal surfaces in ${\rm int}(\tilde{\Omega}'')$, let $\{\Sigma^{(i)}\in \Gamma\}_i$ be a sequence such that
 $$
 {\rm area}_g(\Sigma^{(i)})\rightarrow \inf_{\Sigma \in \Gamma} {\rm area}_g(\Sigma).
 $$
 By compactness, we can suppose that $\Sigma^{(i)}$ converges smoothly to a two-sided, connected,  closed stable minimal surface $\Sigma \subset \tilde{\Omega}''$. It follows from the maximum principle that either $\Sigma \subset \Omega''$ or $\Sigma$ is a component of $\partial \tilde{\Omega}''$.
 
 Suppose $\Sigma$ is a component of $\partial \tilde{\Omega}''$, and $\{\Sigma_t\}_{t\in [0,\delta]}$  is the foliation of the neighborhood of $\Sigma$. It follows by the maximum principle, since the mean curvature vector of $\Sigma_t$ has the same sign, 
 that $\Sigma^{(i)}$ is one of the leaves of $\{\Sigma_t\}_{t\in [0,\delta]}$ for sufficiently large $i$. We can suppose that 
 ${\rm area}_g(\Sigma^{(i+1)})\leq {\rm area}_g(\Sigma^{(i)})$ for any $i$. Therefore $\{\Sigma_t\}_{t\in [0,\delta]}$ is such that
 ${\rm area}_g(\Sigma_t)>{\rm area}_g(\Sigma)$ for any $t\in [0,\delta]$. Since the mean curvature vector of $\Sigma_t$ has the same sign, the function $t \mapsto {\rm area}_g(\Sigma_t)$ is nondecreasing. But then if $\Sigma^{(i)}=\Sigma_{t_i}$, 
 $W^{(1)}(\tilde{\Omega}''\setminus (\cup_{t\in [0,t_i)}\Sigma_t),g)\geq 4\pi$. This is because the surface $\Sigma_{t_i}$ can be connected to $\Sigma$ with area  bounded by ${\rm area}_g(\Sigma_{t_i})$ using the foliation, and $W^{(1)}(\tilde{\Omega}'',g)\geq 4\pi$. Since  $\Sigma^{(i)}$ is a stable minimal surface, 
 $$
 \Omega_{t_i}=\tilde{\Omega}''\setminus (\cup_{t\in [0,t_i)}\Sigma_t)\in \mathcal{I}_{\Omega_1}.
 $$
 This is a contradiction, since ${\rm vol}_g( \Omega_{t_i})<{\rm vol}_g(\tilde{\Omega}'')$. 
 
 Therefore $\Sigma \subset \Omega''$, and hence $\Sigma \in \Gamma$. Let $\{\Sigma_t\}_{t\in [-\delta,\delta]}$ be the foliation of a neighborhood of $\Sigma$ as before. 
Since $\Sigma$ is stable, ${\rm area}_g(\Sigma)\leq 4\pi/3$.  Suppose that  $\{\Sigma_t\}_{t\in [-\delta,\delta]}$ can be extended to $\{\Sigma_t\}_{t\in [-\delta,\tilde{\delta}]}$, $0<\delta\leq \tilde{\delta}$, such that $\Sigma_t$ is a minimal surface for any $t\in [0,\tilde{\delta}]$ 
and $\Sigma_{\tilde{\delta}}$ is a component of $\partial \tilde{\Omega}''$. Then 
$W^{(1)}(\tilde{\Omega}''\setminus (\cup_{t\in [0,\tilde{\delta}]}\Sigma_t),g)\geq 4\pi,$ which is a contradiction as in the previous paragraphs.

Hence we can suppose that there is a foliation $\{\Sigma_t\}_{t\in [-\delta,\delta]}$, and $0\leq \tilde{\delta}_1,\tilde{\delta}_2<\delta$, such that $\Sigma_t$ is minimal for any $t\in [-\tilde{\delta}_2,\tilde{\delta}_1]$, the mean curvature vector of $\Sigma_t$
has the same sign, and there are sequences $t_{j,1},t_{j,2}\rightarrow 0$, $t_{j,1},t_{j,2}>0$,  so  $\Sigma_{\tilde{\delta}_1+t_{j,1}}$ has mean curvature vector pointing strictly away from $\Sigma$, and   $\Sigma_{-\tilde{\delta}_2-t_{j,2}}$ has mean curvature vector pointing strictly towards $\Sigma$.

Let $\Omega_{\tilde{\delta}_1+t_{j,1}}$ be the connected component of $\Omega''\setminus \Sigma_{\tilde{\delta}_1+t_{j,1}}$ that does not contain $\Sigma$. Suppose $\xi>0$ is such that 
$$
{\bf M}_g(\Sigma_{\tilde{\delta}_1+t_{j,1}})+2\xi\leq {\bf M}_g(\Sigma_{\tilde{\delta}_1})={\bf M}_g(\Sigma).
$$
Define $\Gamma_{\tilde{\delta}_1+t_{j,1}}$ to be the set of maps
$\Phi:[0,1]\rightarrow \mathcal{Z}_n(\Omega_{\tilde{\delta}_1+t_{j,1}},\mathbb{Z})$, continuous in the ${\bf F}$-metric, such that
$\Phi(0)=\Sigma_{\tilde{\delta}_1+t_{j,1}}$, and
$$
\sup_{t\in [0,1]}{\bf M}_g(\Phi(t))\leq {\bf M}_g(\Sigma_{\tilde{\delta}_1+t_{j,1}})+ \xi.
$$
If $\lambda=\inf_{\Phi\in \Gamma_{\tilde{\delta}_1+t_{j,1}}}{\bf M}_g(\Phi(1)),$ 
by doing as in the proof of Theorem \ref{connected.minimal} there is a map $\Phi\in \Gamma_{\tilde{\delta}_1+t_{j,1}}$ such
that ${\bf M}_g(\Phi(1))=\lambda$.  

If $\lambda>0$, since $\partial \Omega_{\tilde{\delta}_1+t_{j,1}}$ is mean-convex, the support $\Sigma'$ of $\Phi(1)$ is a two-sided, stable closed embedded minimal surface.  Since
$$
{\bf M}_g(\Sigma')\leq {\bf M}_g(\Sigma_{\tilde{\delta}_1+t_{j,1}})+\xi < {\bf M}_g(\Sigma),
$$
it follows that 
$$
\Sigma'\subset \partial \Omega_{\tilde{\delta}_1+t_{j,1}} \setminus \Sigma_{\tilde{\delta}_1+t_{j,1}}\subset \partial\Omega''.
$$
This is because ${\rm area}_g(\Sigma) =\inf_{\tilde{\Sigma}\in \Gamma}{\rm area}_g(\tilde{\Sigma})$. Since $\Phi(1)$ is 
homologous to $\Sigma_{\tilde{\delta}_1+t_{j,1}}$ in $\Omega_{\tilde{\delta}_1+t_{j,1}}$, and ${\rm support}(\Phi(1))\subset\partial \Omega_{\tilde{\delta}_1+t_{j,1}} \setminus \Sigma_{\tilde{\delta}_1+t_{j,1}}$, by the constancy theorem 
$$
\Phi(1)=\partial \Omega_{\tilde{\delta}_1+t_{j,1}} \setminus \Sigma_{\tilde{\delta}_1+t_{j,1}}.
$$

If $\Omega_\Sigma$ is the connected component of $\Omega''\setminus \Sigma$ that does not contain $\Sigma_{-\delta}$, then 
by concatenating $\Phi$ with the foliation there is a sweepout of $\tilde{\Omega}_\Sigma$ connecting $\Sigma$ to 
$\partial\Omega_\Sigma \setminus \Sigma$ with areas bounded by ${\rm area}_g(\Sigma)$.  It follows that
$W^{(1)}(\tilde{\Omega}''\setminus \Omega_\Sigma,g)\geq 4\pi$, which is a contradiction as before.

This proves that there is no two-sided, closed, stable embedded minimal surface in ${\rm int}(\tilde{\Omega}'')$.

Let $\tilde{g}_i$ be a sequence of bumpy Riemannian metrics converging smoothly to $g$ on $\tilde{\Omega}''$ such that 
$\partial \tilde{\Omega}''$ is a strictly stable closed minimal surface for the metric  $\tilde{g}_i$. Suppose that $\{\Sigma_i\}$ is
a sequence of $\tilde{g}_i$-stable, two-sided, minimal surfaces. Then the sequence converges smoothly, after passing to a subsequence, to a component of $\partial \tilde{\Omega}''$. This is because there is no two-sided, stable, or one-sided, with stable two-cover, closed minimal surface in ${\rm int}(\tilde{\Omega}'')$ for the metric $g$. If $\{\Sigma_{i,1}\}_i$ and $\{\Sigma_{i,2}\}_i$ are sequences of  $\tilde{g}_i$-stable minimal surfaces converging smoothly to a component $\Sigma$, $\Sigma_{i,1}\neq 
\Sigma_{i,2}$, then the surfaces are strictly on one side of each other. If not, there would be a changing sign Jacobi field on $
\Sigma$ which is a contradiction since $\Sigma$ is stable for the metric $g$.

Hence there is $\Omega_i''\subset \Omega''$, $\Omega_i''\rightarrow \Omega''$, $\partial \Omega_i''$ a $\tilde{g}_i$-stable closed minimal surface converging smoothly to $\partial \Omega''$, such that there is no two-sided, $\tilde{g}_i$-stable closed minimal surface in ${\rm int}(\Omega_i'')$. Hence $\lim_i W^{(1)}(\Omega_i'', \tilde{g}_i)=W^{(1)}(\tilde{\Omega}'',g)\geq 4\pi$.

Suppose $\Sigma$ is a component of $\partial \tilde{\Omega}''$, and $\{\Sigma_t\}_{t\in [0,\delta]}$  is the foliation of the neighborhood of $\Sigma$ such that the mean curvature  vector of $\Sigma_t$ has the same sign.
We can suppose the  surfaces $\{\Sigma_t\}_{t\in (0,\delta]}$ have the same sign strictly, as otherwise there would be a two-sided, stable minimal surface in ${\rm int}(\tilde\Omega'')$.

If there is $\tilde{t}\in (0,\delta]$ such that the mean curvature vector of $\Sigma_{\tilde{t}}$ points strictly away from $\Sigma$, then we could use min-max to construct a sweepout of $\tilde{\Omega}_i''\setminus \cup_{t\in [0,\tilde{t})}\Sigma_t$ connecting $
\Sigma_{\tilde{t}}$ to $\partial \tilde{\Omega}_i''\setminus \cup_{t\in [0,\tilde{t}]}\Sigma_t$ with $\tilde{g}_i$-areas bounded by $4\pi/3+2\eta$. This is a contradiction since $\lim_i W^{(1)}(\Omega_i'', \tilde{g}_i)\geq 4\pi$. This proves that for any component $
\Sigma$ of $\partial \tilde{\Omega}''$, there is a foliation $\{\Sigma_t\}_{t\in [0,\delta]}$ of a neighborhood of $\Sigma$ in $
\tilde{\Omega}''$ such that the mean curvature of $\Sigma_t$ points strictly towards $\Sigma$ for any $t\in (0,\delta]$.

Since $\lim R_{\tilde{g}_i}=R_g$ and the metric $\tilde{g}_i$ is bumpy, by using the Proposition \ref{area.surface},   there is a closed, connected, two-sided, embedded minimal surface $\Sigma_i$ for the metric $\tilde{g}_i$, of Morse index one,  such that 
$\limsup_i {\rm area}_{\tilde{g}_i}(\Sigma_i)\leq 4\pi$, in ${\rm int}(\tilde{\Omega}_i'')$. Then $\Sigma_i$ converges as varifolds to 
$k \Sigma$, where $\Sigma$ is a closed minimal surface for the metric $g$ in $\tilde{\Omega}''$. If $\Sigma$ is a component of 
$\partial \tilde{\Omega}''$, then for sufficiently large $i$ the surface $\Sigma_i$ is contained in $\cup_{t\in [0,\delta/2]}\Sigma_t$.
Since, for sufficiently large $i$, the surface $\Sigma_\delta$ has mean curvature vector in the metric $\tilde{g}_i$ pointing strictly towards $\Sigma$, by minimizing the area and using $\Sigma_i$ as barrier there would be a $\tilde{g}_i$-stable minimal surface in ${\rm int}(\Omega_i'')$ which is a contradiction.  

Hence $\Sigma \subset {\rm int}(\Omega'')$.  The surface $\Sigma$ cannot be one-sided, since in that case $k\geq 2$ and  the two-cover $\tilde{\Sigma}$ of $\Sigma$ would be stable.
Therefore the surface $\Sigma$ is two-sided and unstable, and the convergence of $\Sigma_i$ to $\Sigma$ is smooth with multiplicity one. Hence ${\rm area}_g(\Sigma)\leq 4\pi$, and ${\rm index}(\Sigma)=1$. It follows that $\Sigma$ has a neighborhood with strictly mean-concave boundary $V$ in the metric $g$. Hence $\partial V$ is strictly mean-concave   in the metric $\tilde{g}_i$ for sufficiently large $i$.

Let $\xi>0$ be such that ${\bf M}_g(\Sigma')+2\xi\leq {\bf M}_g(\Sigma)$ if $\Sigma'$ is a component of $\partial V$. Since $\Omega_i''\rightarrow 
\Omega''$, and using that $V$ has a foliation by surfaces with area in the metric $g$ bounded by ${\rm area}_g(\Sigma)$, we can construct a  sweepout for $\tilde{\Omega}''$ that is optimal for $W^{(1)}(\tilde{\Omega}'',g)$. This is because by the proof of  Proposition \ref{no.stable.two-sided.case}, since $\partial\Omega_i''\rightarrow \partial\Omega''$, there is a map $\tilde{\Phi}\in \Gamma^{(1)}(\tilde{\Omega}'')$ such that for some $0<\lambda<1/5$, ${\rm area}_g(\tilde{\Phi}(t))\leq {\rm area}_g(\Sigma)-\xi$ for $t\in [0,1]\setminus (1/2-\lambda,1/2+\lambda)$, and $\{\Sigma_t=\tilde{\Phi}(t)\}_{t\in [1/2-\lambda,1/2+\lambda]}$ is the smooth foliation of $V$ with ${\rm area}_g(\Sigma_t)<{\rm area}_g(\Sigma)$ for $t\in [1/2-\lambda,1/2+\lambda]$, $t\neq 0$. The function $f(t)={\rm area}_g(\Sigma_t)$ is such that $f''(0)<0$. The map $\tilde{\Phi}$ is optimal for $W^{(1)}(\tilde{\Omega}'',g)$ since
$$
\sup_{t\in [0,1]}{\bf M}_g(\tilde{\Phi}(t))={\rm area}_g(\Sigma)\leq 4\pi\leq W^{(1)}(\tilde{\Omega}'',g)\leq \sup_{t\in [0,1]}{\bf M}_g(\tilde{\Phi}(t)).
$$
Hence $W^{(1)}(\tilde{\Omega}'',g)=4\pi={\rm area}_g(\Sigma)$.

Let $g(t)$ be the Ricci flow on $M$ such that $g(0)=g$. Let $\psi:\Omega''\rightarrow M$ be the inclusion map, which extends to
$\psi:\tilde\Omega''\rightarrow M$ and define  $\tilde{g}(t)=\psi^*(g(t))$. Then $\tilde{g}(t)$ converges smoothly as $t\rightarrow 0$ to $g$ on $\tilde{\Omega}''$. We denote by $g$ the metric $\psi^*(g)$. 

Let $\{\Sigma_i\}$ be a sequence of connected, closed minimal surfaces in ${\rm int}(\tilde{\Omega}'')$  for the metric $g$ such that
${\rm area}_g(\Sigma_i)\leq 4\pi+\eta$ and ${\rm index}(\Sigma_i)\leq 1$. By \cite{sharp}, there is $\{\Sigma_j\}\subset \{\Sigma_i\}$
such that $\Sigma_j$ converges in varifold sense to $k\Sigma$, where $k\in \mathbb{N}$ and $\Sigma$ is a connected,  closed minimal surface. 
Let $\Lambda_{\partial \tilde\Omega''}$ be the neighborhood of $\partial \tilde\Omega''$  such that $\Lambda_{\partial \tilde\Omega''}\setminus \partial \tilde{\Omega}''$  is foliated by surfaces with mean curvature vector pointing strictly towards $\partial \tilde\Omega''$. If $\Sigma \subset \Lambda_{\partial \tilde\Omega''}$, then by the maximum principle $\Sigma_j$ is a component of $\partial \tilde{\Omega}''$ 
for $j\geq \tilde{j}$. This is a contradiction because $\Sigma_j\subset {\rm int}(\tilde{\Omega}'')$.
Hence the surface $\Sigma$ intersects $\tilde{\Omega}''\setminus \Lambda_{\partial \tilde\Omega''}$.


Therefore there is a region $\Gamma$,
$\overline{\Gamma}\subset \Omega''$, such that  any closed minimal surface $\Sigma$ with
${\rm area}_g(\Sigma)\leq 4\pi+\eta$,  ${\rm index}(\Sigma)\leq 1$ and $\Sigma\subset \Omega''$  is contained in $\Gamma$.


By doing as in the proof of Proposition \ref{area.surface}, the domain $\tilde{\Omega}''$ is diffeomorphic to the complement of finitely many balls in a space form. Let $\pi:S^3 \setminus (B_1 \cup \dots \cup B_j)\rightarrow \tilde{\Omega}''$ be the universal cover, where $B_1,\dots, B_j$ are disjoint balls in  $S^3$. The map $\pi$ induces a diffeomorphism in a neighborhood of any $\partial B_{\tilde{j}}$, $1\leq \tilde{j}\leq j$, with image a neighborhood of $\pi(\partial B_{\tilde{j}})$. Hence $\partial B_{\tilde{j}}$ has a neighborhood that is foliated by surfaces with mean curvature vector pointing strictly towards $\partial B_{\tilde{j}}$ in the metric $\pi^*(g)$, for any $1\leq \tilde{j}\leq j$. Hence, as in the previous paragraph, we can suppose by changing  $\Gamma$ if needed  that a closed minimal surface $\Sigma$ in the metric $\pi^*(g)$ with
${\rm area}_{\pi^*(g)}(\Sigma)\leq 4\pi+\eta,$  ${\rm index}(\Sigma)\leq 1$, and $\Sigma\subset {\rm int}(S^3 \setminus (B_1 \cup \dots \cup B_j))$,  is contained in $\pi^{-1}(\Gamma)$.

We can suppose, by decreasing $\Lambda_{\partial \tilde{\Omega}''}$ if needed,  that there is $\delta>0$ such that $d_g(\Lambda_{\partial \tilde{\Omega}''}, \Gamma)\geq 2\delta$. 
Let $\chi:\tilde{\Omega}'' \rightarrow \mathbb{R}$ be a smooth   function such that  $0\leq \chi\leq 1$ on $\tilde{\Omega}''$,   $\chi=0$ on $\Lambda_{\partial \tilde{\Omega}''}$,  and $\chi=1$ on $\{x:d_g(x,\Gamma)\leq \delta\}$. Define the metric $g'(t)=(1-\chi)g+\chi \tilde{g}(t)$. Then $g'(t)=g$ on $\Lambda_{\partial \tilde{\Omega}''}$, $g'(t)=\tilde{g}(t)$ on $\{x:d_g(x,\Gamma)\leq \delta\}$ and $g'(t)$ converges smoothly to $g$ as $t\rightarrow 0$. 

Let $\{\Sigma_{t_i}\}$ be a sequence of two-sided, connected, closed  surfaces, $\Sigma_{t_i}$ a stable minimal surface in $\tilde{\Omega}''$ for the metric $g'(t_i)$, $t_i\rightarrow 0$. Since the sequence has uniformly bounded area, by using the positivity of the scalar curvature, it converges to a component of $\partial \tilde{\Omega}''$. Therefore $\Sigma_{t_i}\subset \Lambda_{\partial \tilde{\Omega}''}$ for $i\geq \tilde{i}$, and hence by the maximum principle $\Sigma_{t_i}$ is a component of $\partial \tilde{\Omega}''$. It follows  that there is $\rho>0$ such that for $t\in [0,\rho]$, there is no two-sided, stable closed minimal surface for the metric $g'(t)$ in $\Omega''$. By decreasing $\rho>0$ if needed, there is no one-sided, with stable two-cover, closed minimal surface   for $g'(t)$ if $t\in [0,\rho]$.


Suppose that $\{\Sigma_{t_i}\}$ is  a sequence of two-sided, connected, closed  surfaces,  $\Sigma_{t_i}$ a minimal surface for the metric $g'(t_i)$, ${\rm area}_{g'(t_i)}(\Sigma_{t_i}) \leq 4\pi+\eta$, ${\rm index}(\Sigma_{t_i})\leq 1$, $\Sigma_{t_i}\subset {\rm int}(\tilde{\Omega}'')$,  $t_i\rightarrow 0$. Then the surface $\Sigma_{t_i}$ intersects  $\tilde{\Omega}''\setminus \Lambda_{\partial \tilde{\Omega}''}$. The sequence  $\{\Sigma_{t_i}\}$ converges in varifold sense to $k \Sigma$, where $\Sigma$  is a  closed, minimal surface  for the metric $g$, $\Sigma\subset {\rm int}(\tilde{\Omega}'')$. It follows that the surface $\Sigma$ is two-sided because there are no one-sided minimal surfaces with  stable oriented two-cover. 
Hence ${\rm area}_g(\Sigma)\leq 4\pi+\eta$, ${\rm index}(\Sigma)\leq 1$, and  $\Sigma \subset \Gamma$.  Therefore, by decreasing $\rho>0$ if needed, 
it follows that for $t\in [0,\rho]$ if $\Sigma'$ is a closed minimal surface for the metric $g'(t)$, ${\rm area}_{g'(t)}(\Sigma')\leq 4\pi+\eta$, ${\rm index}(\Sigma')\leq 1$, $\Sigma'\subset {\rm int}(\tilde{\Omega}'')$,  then $\Sigma' \subset \{x: d_g(x, \overline\Gamma)\leq \delta/2\}$.  And similarly, for $t\in [0,\rho]$ if $S$ is a closed minimal surface for the metric $\pi^*(g'(t))$, ${\rm area}_{\pi^*(g'(t))}(S)\leq 4\pi+\eta$, ${\rm index}(S)\leq 1$, 
$S \subset {\rm int}(S^3 \setminus (B_1 \cup \dots \cup B_j))$, then $S \subset \{x: d_{\pi^*(g)}(x, \pi^{-1}(\overline\Gamma))\leq \delta/2\}$.

Let $t\in [0,\rho]$. Suppose $\tilde{g}_i$ is a sequence of bumpy Riemannian metrics converging smoothly to $g'(t)$ on $\tilde{\Omega}''$ such that 
$\partial \tilde{\Omega}''$ is a strictly stable closed minimal surface for the metric  $\tilde{g}_i$.  The sequence $\tilde{g}_i$ can be constructed since $g'(t)=g$ on $\Lambda_{\partial \tilde\Omega''}$.  As in the case of $(\tilde{\Omega}'',g)$, there is a sequence of regions $\Omega_i\subset \Omega''$, $\Omega_i\rightarrow \Omega''$, such that $\partial \Omega_i$ is a  strictly stable minimal surface for $\tilde{g}_i$, and there is no two-sided, stable, or one-sided, with stable two-cover, minimal surface for $\tilde{g}_i$ in 
$\Omega_i$.

By doing as in the proof of Proposition \ref{area.surface}, 
there is a connected, two-sided,  closed minimal surface $\Sigma_i\subset \Omega_i$ for $\tilde{g}_i$, with Morse index one, and $S_i$ in $\pi^{-1}(\tilde{\Omega}_i)$ that is a minimal two-sphere for  $\pi^*(\tilde{g}_i)$, of Morse index at most one, such that  ${\rm area}_{\tilde{g}_i}(\Sigma_i)\leq {\rm area}_{\pi^*(\tilde{g}_i)}(S_i)$. By min-max for $W^{(1)}(\tilde{\Omega}_i,\tilde{g}_i)$, using
Proposition \ref{no.stable.two-sided.case} and  Proposition \ref{no.stable.one-sided.case}, we can suppose that
${\rm area}_{\tilde{g}_i}(\Sigma_i)=W^{(1)}(\tilde{\Omega}_i,\tilde{g}_i)$.

If $\Sigma_i\subset \Lambda_{\partial \tilde{\Omega}''}$, by minimizing the area in $\Lambda_{\partial \tilde{\Omega}''}$ there would be a two-sided, stable minimal surface in $(\Omega_i,\tilde{g}_i)$ which is a contradiction.  Hence the surface $\Sigma_i$ intersects 
$\tilde{\Omega}''\setminus \Lambda_{\partial \tilde{\Omega}''}$ for sufficiently large $i$.  Therefore $\Sigma_i$ converges in varifold sense to $k \Sigma$, where $\Sigma\subset \Omega''$. Since there is no two-sided, stable, or one-sided, with stable two-cover,  closed minimal surface for the metric $g'(t)$ in $\Omega''$, the surface $\Sigma$ is two-sided and unstable. Therefore $\Sigma_i$ converges smoothly to $\Sigma$, and  it follows that ${\rm area}_{g'(t)}(\Sigma)=W^{(1)}(\tilde{\Omega}'',g'(t))$ and ${\rm index}(\Sigma)=1$.

The sequence $\{S_i\}$ converges in varifold sense to a closed minimal surface that is contained in $\{x: d_{\pi^*(g)}(x, \pi^{-1}(\overline\Gamma))\leq \delta/2\}$. The sequence $\pi^*(\tilde{g}_i)$ converges smoothly to $\pi^*(g'(t))$, and $\pi^*(g'(t))=\pi^*(\tilde{g}(t))$ on $\{x: d_{\pi^*(g)}(x, \pi^{-1}(\overline\Gamma))\leq \delta/2\}$.  Hence one can use the scalar curvature bounds of $\tilde{g}(t)$ to estimate the area of $S_i$.

The metric $g$ is such that $R_g \geq 6$. It follows by the  maximum principle for parabolic equations, if applied to the evolution equation of the scalar curvature of $\tilde{g}(t)$, that
$$
R_{\tilde{g}(t)}\geq \frac{6}{1-4t}
$$
for $t\in [0,\xi]$. 
Hence (e.g. Proposition A. 1, \cite{marques-neves-rigidity-spheres}),
$$
\limsup {\rm area}_{\pi^*(\tilde{g}_i)}(S_i)\leq 4\pi(1-4t).
$$
It follows that ${\rm area}_{g'(t)}(\Sigma)\leq 4\pi(1-4t)$.

Let $\tilde{t}\in [0,\rho]$. By doing as before, there is a map $\tilde{\Phi}_{\tilde{t}}\in \Gamma^{(1)}(\tilde{\Omega}'')$ such that for some $0<\lambda<1/5$, ${\rm area}_{g'(\tilde{t})}(\tilde{\Phi}_{\tilde{t}}(s))\leq {\rm area}_{g'(\tilde{t})}(\Sigma_{\tilde{t}})-\xi$ for $s\in [0,1]\setminus (1/2-\lambda,1/2+\lambda)$, and $\{\Sigma^{(\tilde{t})}_s=\tilde{\Phi}_{\tilde{t}}(s)\}_{s\in [1/2-\lambda,1/2+\lambda]}$ is a smooth foliation of a neighborhood of $\Sigma_{\tilde{t}}$ with
$\tilde{\Phi}_{\tilde{t}}(1/2)=\Sigma_{\tilde{t}}$,  ${\rm area}_{g'(\tilde{t})}(\Sigma^{(\tilde{t})}_s)<{\rm area}_{g'(\tilde{t})}(\Sigma_{\tilde{t}})$ for $s\in [1/2-\lambda,1/2+\lambda]$, $s\neq 0$. The function $f(s)={\rm area}_{g'(\tilde{t})}(\Sigma^{(\tilde{t})}_s)$ is such that $f''(0)<0$.
The surface $\Sigma_{\tilde{t}}$ is a two-sided, closed minimal surface for $g'(\tilde{t})$, ${\rm index}(\Sigma_{\tilde{t}})=1$, ${\rm area}_{g'(\tilde{t})}(\Sigma_{\tilde{t}})=W^{(1)}(\tilde{\Omega}'',g'(\tilde{t}))\leq 4\pi(1-4\tilde{t})$.
It follows that  the map $\tilde{\Phi}_{\tilde{t}}$ is optimal for $W^{(1)}(\tilde{\Omega}'',g'(\tilde{t}))$ since
$$
\sup_{s\in [0,1]}{\bf M}_{g'(\tilde{t})}(\tilde{\Phi}_{\tilde{t}}(s))={\rm area}_{g'(\tilde{t})}(\Sigma_{\tilde{t}})=W^{(1)}(\tilde{\Omega}'',g'(\tilde{t}))\leq \sup_{s\in [0,1]}{\bf M}_{g'(\tilde{t})}(\tilde{\Phi}_{\tilde{t}}(s)).
$$

  The surface $\Sigma_{\tilde{t}}$ is contained in $\{x:d_g(x,\overline{\Gamma})\}\leq \delta/2\}$. Recall that 
 $g'(t)=\tilde{g}(t)$ on $\{x:d_g(x,\overline{\Gamma})\}\leq \delta/2\}$ for $t\in [0,\rho]$.

 Hence as in Proposition 4.2 of \cite{marques-neves-rigidity-spheres}, the  function $t \mapsto W^{(1)}(\tilde{\Omega}'',g'(t))$ which is Lipschitz is such that
 $$
 W^{(1)}(\tilde{\Omega}'',g'(t))\geq W^{(1)}(\tilde{\Omega}'',g)-16\pi t
 $$
 for $t\in [0,\xi]$. This uses that $t \mapsto \tilde{g}(t)$ is a Ricci flow. Since $W^{(1)}(\tilde{\Omega}'',g)=4\pi$, and
 $W^{(1)}(\tilde{\Omega}'',g'(t))={\rm area}_{g'(t)}(\Sigma_t)\leq 4\pi(1-4t)$, it follows that
 $$
 W^{(1)}(\tilde{\Omega}'',g'(t))={\rm area}_{g'(t)}(\Sigma_t)=4\pi(1-4t).
 $$

For $t\in (0,\rho]$,  it follows that 
there is $x\in \tilde\Omega''$ such that 
$R_{\tilde{g}(t)}(x)=\frac{6}{1-4t}$. The equality case in the parabolic maximum principle  implies that $R_g=6$ and  the metric $g$ is Einstein. Hence $g$ has constant sectional curvature one, and  $(M^3,g)$  is a closed spherical space form.

If $M$ contains non-orientable  embedded surfaces, it follows by the proof of Theorem 4.8 of \cite{marques-neves-rigidity-spheres} that there is a closed, one-sided, embedded minimal surface $\Sigma$ in $M$ with ${\rm area}_g(\Sigma)\leq 2\pi$. By Proposition \ref{no.stable.one-sided.case}, $W^{(1)}(M,g)<4\pi$. This is because $g$ has positive Ricci curvature and hence there are no two-sided stable minimal immersions.
Hence, by Theorem 4.4 of \cite{marques-neves-rigidity-spheres}, either $(M,g)$ is isometric to $(S^3,\overline{g})$ or $W^{(1)}(M,g)<4\pi$.  
  Since $4\pi\leq W^{(k)}(N,h)\leq W^{(1)}(M,g)$, it follows that $(M,g)$ is isometric to the unit sphere and that 
 $$
 W^{(k)}(N,h)=4\pi.
 $$
 
 Let $f:(S^3,\overline{g})\rightarrow (N,h)$ be an area-nonincreasing Lipschitz map with  nontrivial degree $k$. Suppose that $N$ is simply-connected.   If $v\in \mathbb{R}^4$, $|v|=1$, let $\Sigma_v =\{x\in S^3: \langle x,v\rangle=0\}$ and $\Sigma_{v,t}=\{x\in S^3:\langle x,v\rangle=t\}$, $t\in [-1,1]$.
 The map $\Phi_v:[-1,1]\rightarrow \mathcal{Z}_2(S^3,\mathbb{Z})$ defined by $\Phi_v(t)=\partial(\{x\in S^3: \langle x,v \rangle \leq t\})$ is continuous in the ${\bf F}$-metric. It follows by the Appendix of \cite{marques-neves-index} that there is a sequence
 of maps $\Phi_v^{(i)}:[-1,1]\rightarrow \mathcal{Z}_2(S^3,\mathbb{Z})$, continuous in the mass metric, such that
 $\Phi_v^{(i)}(-1)=\Phi_v^{(i)}(1)=0$,
 $\sup_{t\in [-1,1]}{\bf F}(\Phi_v^{(i)}(t),\Phi_v(t))\rightarrow 0$ as $i\rightarrow \infty$, and $\Phi_v^{(i)}$ is homotopic to $\Phi_v$ in the flat topology relative to $\{-1,1\}$. Hence $\Phi_v^{(i)}\in \Gamma^{(1)}(S^3)$, after reparametrizing.
 
 Therefore $\Psi_v^{(i)}=f_{\#}\circ \Phi_v^{(i)}$ is continuous in the mass metric, and $\Psi_v^{(i)}\in \Gamma^{(k)}(N)$ after reparametrizing. Since ${\bf M}_h(\Psi_v^{(i)}(t))={\bf M}_h(f_{\#}\circ \Phi_v^{(i)}(t))\leq {\bf M}_{\overline{g}}(\Phi_v^{(i)}(t))$ for any $t\in [-1,1]$, it follows that
 $$
 \limsup_{i\rightarrow \infty} \sup_{t\in [-1,1]} {\bf M}_h(\Psi_v^{(i)}(t))\leq 4\pi.
 $$ 
 Hence  $\{\Psi_v^{(i)}\}_i$ is an optimal sequence of sweepouts for $W^{(k)}(N,h)$. It follows by min-max theory that there are sequences $\{j\}\subset \{i\}$ and $\{t_j\in [-1,1]\}_j$ such that 
 $$
 |\Psi_{v}^{(j)}(t_j)|\rightarrow V=\sum_{i=1}^{r} m_i |\Sigma_i|
 $$
 in varifold sense, where $m_i\in \mathbb{N}$,  $\{\Sigma_1,\dots,\Sigma_r\}$ is a disjoint collection of smooth, closed, embedded minimal hypersurfaces of $(N,h)$, and 
 $$
 \sum_{i=1}^r m_i{\rm area}_h(\Sigma_i)=4\pi.
 $$
 
Since $\Psi_{v}^{(j)}(t_j)=f_{\#}(\Phi_{v}^{(j)}(t_j))$, $f$ is area-nonincreasing, and ${\bf M}_h(\Psi_{v}^{(j)}(t_j))\rightarrow {\bf M}_h(V)=4\pi$, it follows that $t_j\rightarrow 0$ and  $\Phi_{v}^{(j)}(t_j)\rightarrow \Sigma_v$ in the ${\bf F}$-metric.  Suppose that there is $x\in \Sigma_v$ such that $f(x)\notin \cup_{i=1}^r\Sigma_i$. The map $f$ is continuous, hence  there is a neighborhood $\tilde{V}$ of $x$ in $S^3$ and $\delta>0$ such that
$d_h(f(\tilde{V}), \cup_{i=1}^r\Sigma_i)\geq \delta$. Since $f$ is area-nonincreasing, it follows that 
$$
{\bf M}_h(|\Psi_{v}^{(j)}(t_j)| \cap \overline{B}_{\delta/2}(\cup_{i=1}^r\Sigma_i))\leq {\bf M}_{\overline{g}}(\Phi_{v}^{(j)}(t_j)\setminus \tilde{V}).
$$
for sufficiently large $j$.
  
  This is a contradiction, since 
$$
\limsup_{j\rightarrow \infty}{\bf M}_{\overline{g}}(\Phi_{v}^{(j)}(t_j)\setminus \tilde{V})
<4\pi
$$
and $|\Psi_{v}^{(j)}(t_j)| \cap \overline{B}_{\delta/2}(\cup_{i=1}^r\Sigma_i)\rightarrow V$ as varifolds.
 
 Therefore $f(\Sigma_v)\subset  \cup_{i=1}^r\Sigma_i$. Suppose $r\geq 2$. Since $\Sigma_v$ is connected, and $f$ is continuous,  $f(\Sigma_v)\subset \Sigma_{\tilde{i}}$ for some 
 $\tilde{i}=1,\dots,r$. Let $\xi>0$ be such that $\Sigma_{\tilde{i}}$ is disjoint from  $B_\xi(\cup_{i\neq \tilde{i}} \Sigma_i)$. The sequence $|\Phi_{v}^{(j)}(t_j)|$ converges to $|\Sigma_v|$ as varifolds.  Hence, for any $\eta>0$,
 ${\bf M}_g(|\Phi_{v}^{(j)}(t_j)| \cap (S^3\setminus B_\eta(\Sigma_v))\leq \eta$ for sufficiently large $j$. Let $\eta>0$
 be such that $f(\overline{B}_\eta(\Sigma_v))$ is disjoint from $\overline{B}_{\xi/2}(\cup_{i\neq \tilde{i}} \Sigma_i)$ and 
 $2\eta \leq \sum_{i\neq \tilde{i}}m_i{\bf M}_h(\Sigma_i)$.
 Then, by the area formula, and since $f$ is area-nonincreasing,
 $$
 {\bf M}_h(|\Psi_{v}^{(j)}(t_j)| \cap \overline{B}_{\xi/2}(\cup_{i\neq \tilde{i}} \Sigma_i))\leq {\bf M}_g(|\Phi_{v}^{(j)}(t_j)| \cap (S^3\setminus B_\eta(\Sigma_v))\leq \eta
 $$
 for sufficiently large $j$. This is a contradiction because
 $$
 |\Psi_{v}^{(j)}(t_j)| \cap \overline{B}_{\xi/2}(\cup_{i\neq \tilde{i}} \Sigma_i)\rightarrow \sum_{i\neq \tilde{i}} m_i |\Sigma_i|
 $$
 as varifolds.
 
 Hence $r=1$, and $V=m\cdot |\Sigma|$, $f(\Sigma_v)\subset \Sigma$.  Let $K=f(\Sigma_v)$. The sequence $\Phi_v(t_{j})$ converges smoothly to $\Sigma_v$. If $\delta>0$, there is $\xi>0$ such that $f(B_\xi(\Sigma_v))\subset B_\delta(K)$. Then ${\bf M}(\Phi_v(t_{j}) \cap (S^3\setminus B_\xi(\Sigma_v))=0$ for sufficiently large $j$, and 
 ${\bf F}(\Phi_v^{(j)}(t_{j}), \Phi_v(t_{j}))\rightarrow 0$, hence
 ${\bf M}(\Phi_v^{(j)}(t_{j}) \cap (S^3\setminus B_\xi(\Sigma_v))\rightarrow 0$. Hence ${\bf M}(\Psi_v^{(j)}(t_{j}) \cap (N\setminus B_\delta(K))\rightarrow 0$, and therefore ${\bf M}(V \cap (N\setminus B_\delta(K))= 0$. Since $V=m\cdot |\Sigma|$, it follows that
 $K=\Sigma$. Hence $f(\Sigma_v)=\Sigma$.

  Let $\tilde{V}\in {\bf C}(\{\Psi_v^{(j)}\})$. By doing as in the previous paragraph, it follows that ${\rm support}(\tilde{V})\subset \Sigma$. Let $\tilde{\Psi}_v^{(j)}$ be the sequence of sweepouts obtained from $\Psi_v^{(j)}$ by pull-tight. Then 
 ${\bf C}(\{\tilde{\Psi}_v^{(j)}\})\subset {\bf C}(\{\Psi_v^{(j)}\})$ and $\tilde{V}$ is stationary if $\tilde{V}\in {\bf C}(\{\tilde{\Psi}_v^{(j)}\})$.
 Suppose $\tilde{V}\in  {\bf C}(\{\tilde{\Psi}_v^{(j)}\})$. Hence $\tilde{V}$ is a stationary varifold with support contained in $\Sigma$
 and ${\bf M}_h(\tilde{V})=4\pi$, for any $\tilde{V}\in  {\bf C}(\{\tilde{\Psi}_v^{(j)}\})$. Since $\Sigma$ is a smooth, closed, embedded, minimal surface, it follows by the constancy theorem for stationary varifolds that ${\bf C}(\{\tilde{\Psi}_v^{(j)}\})=\{m|\Sigma|\}$.

The surface $\Sigma$ is two-sided since $N$ is simply-connected. 
It follows from the index estimates of \cite{marques-neves-index} that ${\rm index}(\Sigma)\leq 1$.  The metric $h$ is not necessarily bumpy, but the deformation theorems of \cite{marques-neves-index} can be applied since the number of elements of ${\bf C}(\{\tilde{\Psi}_v^{(j)}\})$ is one.

This proves that for any $v\in S^3$, $|v|=1$, $f(\Sigma_v)$ is a smooth, closed, two-sided, embedded minimal surface of $(N,h)$ such that
${\rm index}(f(\Sigma_v))\leq 1$ and  $m \, {\bf M}_h(f(\Sigma_v))=4\pi$ for some $m\in \mathbb{N}$ which depends on $v$. Define the stationary varifold $\Gamma(v)= m |f(\Sigma_v)|$ so that ${\bf M}_h(\Gamma(v))=4\pi$. Let $v_i\in S^3$, $v_i\rightarrow v$. Suppose $\{j\}\subset \{i\}$ is a sequence such
that $\Gamma(v_j)\rightarrow V$ as varifolds. Since $f$ is a continuous map, ${\rm support}(V)\subset f(\Sigma_v)$.  It follows by the constancy theorem for stationary varifolds that $V=\Gamma(v)$, since ${\bf M}(V)=4\pi$. This proves that $\Gamma(v_i)\rightarrow \Gamma(v)$ as varifolds. Hence the map $v\mapsto \Gamma(v)$ is continuous in the varifold topology.
 The proof of the Proposition 6.2 of \cite{ambrozio-marques-neves-rigidity} can be used for the set of varifolds $\{\Gamma(v)\}$, since the map $f$ is surjective. 
 
 Hence, if  $v\in S^3$ then there is no closed, embedded, minimal surface disjoint from $f(\Sigma_v)={\rm support}(\Gamma(v))$, and $f(\Sigma_v)$ is part of a local foliation so that the leaves on each of the sides of $f(\Sigma_v)$ have mean curvature vector pointing strictly away from $f(\Sigma_v)$.  Let $\{\Sigma'(t)\}_{t\in [-\eta,\eta]}$ be the foliation, $\Sigma'(0)=f(\Sigma_v)$.  The surface $f(\Sigma_v)$ separates $N$. If $\Sigma'(\eta)=\partial R_\eta$, $R_\eta$ disjoint from $f(\Sigma_v)$,  by min-max there is a sweepout 
 $\Phi_\eta:[0,1]\rightarrow \mathcal{Z}_2(N,\mathbb{Z})$ of $R_\eta$, $\Phi_\eta(0)=\Sigma'(\eta)$, $\Phi_\eta(1)=0$, such that
 $\sup_{t\in [0,1]}{\bf M}_h(\Phi_\eta(t))\leq {\bf M}_h(f(\Sigma_v))$. This uses that $R_\eta$ has strictly mean-convex boundary, and that there is no closed, embedded, minimal surface disjoint from $f(\Sigma_v)$. Similarly, if $\Sigma'(-\eta)=\partial R_{-\eta}$, $R_{-\eta}$ disjoint from 
 $f(\Sigma_v)$, there is a sweepout 
 $\Phi_{-\eta}:[0,1]\rightarrow \mathcal{Z}_2(N,\mathbb{Z})$ of $R_{-\eta}$, $\Phi_{-\eta}(0)=\Sigma'(-\eta)$, $\Phi_{-\eta}(1)=0$, such that
 $\sup_{t\in [0,1]}{\bf M}_h(\Phi_{-\eta}(t))\leq {\bf M}_h(f(\Sigma_v))$.  
 
 By concatenating $\Phi_{-\eta}$, $\{\Sigma'(t)\}_{t\in [-\eta,\eta]}$, and $\Phi_\eta$, after reparametrizing, we obtain a map $\Psi'\in \Gamma^{(1)}(N,\mathbb{Z})$ such that
 $\sup_{t\in [0,1]}{\bf M}_h(\Psi'(t))\leq {\bf M}_h(f(\Sigma_v))$. Hence
 $$
 W^{(k)}(N,h)\leq W^{(1)}(N,h) \leq {\bf M}_h(f(\Sigma_v)).
 $$
Since $W^{(k)}(N,h)=4\pi$, it follows that $\Gamma(v)$ has multiplicity one and $\Gamma(v)=|f(\Sigma_v)|$ for any $v\in S^3$.

 The proofs of Proposition 4.10 and Theorem 4.11 of 
 \cite{marques-neves-lower-bound} imply that there is an orientation of $f(\Sigma_v)$ such that  $f_{\#}(\Sigma_v)=f(\Sigma_v)$. 

  By definition, if $\omega$ is a smooth differential two-form on $N$,
 $$
 f_{\#}(\Sigma_v)(\omega)=\int_{\Sigma_v} \langle \omega(f(x)),  (df_x)_{\#}\xi(x)\rangle d\Sigma_v
 $$
 where $\xi(x)$ is the orientation of $\Sigma_v$ at $x$. Since $f_{|\Sigma_v}$ is Lipschitz, the  map $d(f_{|\Sigma_v})_x$ is well-defined for a.e. $x\in \Sigma_v$.

 Let $\tilde{\omega}$ be such that $||\tilde{\omega}||\leq 1$ and $\tilde{\omega}_{|\Sigma}$ is the area form of $\Sigma$ for the metric $h$, where $\Sigma=f(\Sigma_v)$.
 Then $\Sigma(\tilde{\omega})=\int_\Sigma \tilde{\omega}= {\bf M}_h(\Sigma)=4\pi$. Since $f_{\#}(\Sigma_v)=\Sigma$, it follows that
 $$
 \int_{\Sigma_v} \langle \tilde{\omega}(f(x)),  (df_x)_{\#}\xi(x)\rangle d\Sigma_v=4\pi.
 $$
 
 Let $\mu_{\overline{g}}$ be the volume measure of $(S^3,\overline{g})$. Since $f:S^3 \rightarrow N$ is a  Lipschitz area-nonincreasing map, there is $\Omega \subset S^3$, $\mu_{\overline{g}}(S^3\setminus \Omega)=0$, such that  $f$ is differentiable at $x$ and 
 $df_x:T_xS^3\rightarrow T_{f(x)}N$ is area-nonincreasing  for any  $x\in \Omega$.
 Therefore, for  a.e. $v$, $|v|=1$, $x\in \Omega$  for a.e. $x\in \Sigma_v$.
Hence 
 $$
 \langle \tilde{\omega}(f(x)),  (df_x)_{\#}\xi(x)\rangle\leq |\langle \tilde{\omega}(f(x)),  (df_x)_{\#}\xi(x)\rangle| \leq 1
 $$
 for a.e.  $v$, $|v|=1$, and for a.e. $x\in \Sigma_v$. Since ${\bf M}_{\overline{g}}(\Sigma_v)=4\pi$, it follows that 
 $$
 \langle \tilde{\omega}(f(x)),  (df_x)_{\#}\xi(x)\rangle= 1
 $$
  for a.e.  $v$, $|v|=1$, and for a.e. $x\in \Sigma_v$.

This proves that $df_x:T_x\Sigma_v\rightarrow T_{f(x)}\Sigma$ is area-preserving and orientation-preserving for a.e.  $v$, $|v|=1$, and for a.e. $x\in \Sigma_v$.

By the coarea formula, for a.e. $(x,\sigma)$ in the two-dimensional Grassmann bundle of $S^3$, 
$\sigma\subset T_xS^3$, if $\Sigma_\sigma$ is the equatorial sphere with $T_x\Sigma_\sigma=\sigma$, the map $f_{|\Sigma_\sigma}$ is differentiable at $x$ and $d(f_{|\Sigma_\sigma}):T_x\Sigma_\sigma=\sigma \rightarrow T_{f(x)}S^3$ is area-preserving. Let $\Lambda$ be the set of such $(x,\sigma)$. If $x\in \Omega$, let $\Omega_x$ be the set of two-dimensional planes $\sigma\subset T_xS^3$ such that $(x,\sigma)\in \Lambda$.
Hence there is $\Omega'\subset \Omega$, $\mu_{\overline{g}}(\Omega\setminus \Omega')=0$,  such that if $x\in \Omega'$ then $(x,\sigma)\in \Lambda$ for a.e. $\sigma \subset T_xS^3$.

  Hence, if $x\in \Omega'$, $(df_x)_{|\sigma}$ is area-preserving for a.e. $\sigma\subset T_xS^3$. Since $df_x$ is a linear map, it follows that $df_x:T_xS^3\rightarrow T_{f(x)}N$ is area-preserving. If $x\in \Omega'$, there is a $\overline{g}$-orthonormal basis 
  $\{e_{\overline{g},1},e_{\overline{g},2},e_{\overline{g},3}\}\subset T_xS^3$  such that $df_x(e_{\overline{g},i})=\lambda_i e_{h,i}$, 
  $\lambda_i\geq 0$, $\{e_{h,1}, e_{h,2}, e_{h,3}\}$ an $h$-orthonormal basis of $T_{f(x)}N$. 
Then $\lambda_i\lambda_j=1$ for $i\neq j$, and hence $\lambda_i=1$ for any $1\leq i\leq 3$. Therefore
$df_x:T_xS^3\rightarrow T_{f(x)}N$ is length-preserving for any $x\in \Omega'$.

  We claim that the map $f:S^3\rightarrow N$ is a homeomorphism.
 Let $p,q\in S^3$, $p\neq q$.  Let $v,w\in S^3$, $\langle v,w\rangle =0$, be such that $v,w$ are orthogonal to $p,q$. Define $v(t)=(\cos t) v + (\sin t) w$, for $t\in [0,\pi]$. Then $p,q\in \Sigma_{v(t)}$ for any $t\in [0,\pi]$. The set $\Omega'$ is such that $\mu_{\overline{g}}(S^3\setminus \Omega')=0$. Hence, if $\mathcal{H}^2$ denotes the two-dimensional Hausdorff measure, it follows that   $\mathcal{H}^2(\Sigma_{v(t)}\setminus \Omega') =0$ for a.e. $t\in [0,\pi]$. Let $\tilde{t}\in [0,\pi]$ be such that $\mathcal{H}^2(\Sigma_{v(\tilde{t})}\setminus \Omega') =0$. It follows from the previous paragraphs that there is a two-sided surface $\Sigma\subset N$ with ${\bf M}_h(\Sigma)=4\pi$ so that $f:\Sigma_{v(\tilde{t})}\rightarrow \Sigma$ is such that $df_x:T_x\Sigma_{v(\tilde{t})}\rightarrow T_{f(x)}\Sigma$ is length-preserving and orientation-preserving for a.e. $x\in \Sigma_{v(\tilde{t})}$.

 Hence the map $f_{|\Sigma_{v(\tilde{t})}}$ satisfies the Cauchy-Riemann equations a.e. in local conformal charts for $\Sigma_{v(\tilde{t})}$ and $\Sigma$.  This implies $f:\Sigma_{v(\tilde{t})}\rightarrow \Sigma$ is smooth and hence a local Riemannian isometry. The surface $(\Sigma_{v(\tilde{t})}, \overline{g}_{|\Sigma_{v(\tilde{t})}})$
 is isometric to the unit sphere, and therefore $(\Sigma,h_{|\Sigma})$ has constant curvature one.  Since ${\rm area}(\Sigma,h_{|\Sigma})={\bf M}_h(\Sigma)=4\pi$, the map $f:\Sigma_{v(\tilde{t})}\rightarrow \Sigma$ is an isometry.  This proves that $f(p)\neq f(q)$.

 This implies that  $f:S^3\rightarrow N$ is injective and hence a homeomorphism.
 Since $df_x:T_xS^3\rightarrow T_{f(x)}N$ is length-preserving for a.e. $x\in S^3$, $f$ is a metric isometry by Proposition 2.14 of \cite{cecchini-hanke-schick}. This implies that $f$ is a smooth Riemannian isometry in the case $N$ is simply-connected.

 In the case $N$ is not simply-connected, 
let $\pi:\tilde{N}\rightarrow N$ be the universal cover of $N$. The map $f$ lifts to a map $\tilde{f}:(S^3,\overline{g})\rightarrow (\tilde{N},\pi^*(h))$, $\pi \circ \tilde{f}=f$.
 If $\tilde{N}$ is noncompact, $H_3(\tilde{N},\mathbb{Z})=0$  and hence $\tilde{f}_{\#}(S^3)=\partial R$. This implies $k\cdot [N]=f_{\#}(S^3)=\partial(\pi_{\#}(R))$, which is a contradiction since $k\neq 0$. Therefore $\tilde{N}$ is compact.

 The map $\tilde{f}$ has nontrivial degree because ${\rm deg}(f)={\rm deg}(\pi)\cdot {\rm deg}(\tilde{f})$.
The map $\pi$ is a local isometry, and $\pi_{\#}([\tilde{N}])={\rm deg}(\pi)\cdot [N]$. This implies that 
$W^{{\rm deg}(f)}(N,h)\leq W^{{\rm deg}(\tilde{f})}(\tilde{N},\pi^*(h))$. Hence the map $\tilde{f}:(S^3,\overline{g})\rightarrow (\tilde{N},\pi^*(h))$ is  area-nonincreasing  with nontrivial degree, and 
$$
W^{{\rm deg}(\tilde{f})}(\tilde{N},\pi^*(h))\geq 4\pi.
$$

 It follows from the  simply-connected case  that $\tilde{f}$ is a Riemannian isometry.  Hence $(N,h)$ is isometric to a spherical space form.
 As before, it follows that either $(N,h)$ is isometric to $(S^3,\overline{g})$ or $W^{(1)}(N,h)<4\pi$.  Since $N$ is not simply-connected,  $W^{(1)}(N,h)<4\pi$ which  is a contradiction.  This finishes the proof of the theorem.

 \end{proof}

   \begin{thm}\label{rigidity.surfaces}
Let $(M^3,g)$ be a closed, oriented, Riemannian manifold  such that  $R_g\geq 6$. Suppose     $(N^3,h)$ is a closed,  oriented, Riemannian manifold such that if $\Sigma$ is a  closed,  embedded, minimal surface in $(N,h)$,  then  ${\rm area}_h(\Sigma)\geq4\pi$ if $\Sigma$ is two-sided and ${\rm area}_h(\Sigma)\geq 2\pi$ if   $\Sigma$ is one-sided. Suppose $f:(M,g)\rightarrow (N,h)$ is a Lipschitz  area-nonincreasing  map such that  ${\rm deg}(f)$ is nontrivial. Then $(M,g)$ is isometric to the unit sphere and  $f$ is a smooth Riemannian  isometry.
  \end{thm}
  
  \begin{proof}
  Suppose $W^{(k)}(N,h)<4\pi$, where $k={\rm deg}(f)$. Let $h_i$ be a sequence of bumpy metrics on $N$ converging smoothly to $h$.
  By Proposition \ref{index.width}, there is a closed, embedded, connected, two-sided, with Morse index one, $h_i$-minimal hypersurface $\Sigma_i$
  such that ${\rm area}_{h_i}(\Sigma_i)\leq W^{(k)}(N,h_i)$. The varifolds $|\Sigma_i|$ by \cite{sharp} converge to $m |\Sigma|$ where $\Sigma$ is a closed, embedded, minimal surface in $(N,h)$ and $m\in \mathbb{N}$.  Since $W^{(k)}(N,h_i)\rightarrow W^{(k)}(N,h)<4\pi$,
  it follows that ${\rm area}_h(\Sigma)< 4\pi$. If $\Sigma$ is one-sided, then $m\geq 2$ and hence ${\rm area}(\Sigma)<2\pi$. This is a contradiction and therefore $W^{(k)}(N,h)\geq 4\pi$. The hypotheses of Theorem \ref{scalar-rigidity.2} are satisfied with $\rho <4\pi-4\pi/3$.
  \end{proof}
  
  The hypothesis of Theorem \ref{rigidity.surfaces} can be phrased as $\mathcal{A}_1(N)\geq 4\pi$ (\cite{mazet-rosenberg}).

  \section{Homology and incompressible surfaces}\label{homology.incompressible}

Let $[\Sigma]$ be the integer homology class of a closed hypersurface $\Sigma \subset M$.

\begin{thm}\label{stable}
Let $M^{n+1}$, $(n+1)\geq 3$, be an oriented  closed manifold satisfying that the cup product 
$$
\smile\,: H^1(M,\mathbb{Z})\times H^1(M,\mathbb{Z})\rightarrow H^2(M,\mathbb{Z})
$$
is nontrivial, in the sense that there are cohomology classes $\alpha,\beta\in H^1(M,\mathbb{Z})$ such that $\alpha \smile \beta \neq 0$.
Then, for any smooth Riemannian metric $g$ on $M$, there is a sequence of homologically area-minimizing, two-sided,   smoothly embedded outside a set of codimension seven, connected closed  minimal 
hypersurfaces  $\Sigma_i\subset (M,g)$, with  $[\Sigma_i]\neq [\Sigma_j]$ for $i\neq j$, and ${\rm vol}_g(\Sigma_i)\rightarrow \infty$. 
\end{thm}

The Theorem  \ref{stable} can be applied to a manifold that is diffeomorphic to  a product $\Sigma \times \Sigma'$ of closed orientable surfaces $\Sigma,\Sigma'$, if $\Sigma$ or $\Sigma'$ is not a two-sphere. The product metric $\overline{g}+ \overline{g}$ on $S^2 \times S^2$, where $\overline{g}$ is the constant curvature one metric, has positive bi-Ricci curvature.  By \cite{chu-lee-zhu}, smooth closed two-sided stable minimal hypersurfaces have bounded area for metrics in a neighborhood of $\overline{g}+ \overline{g}$.

The product metric $dt^2+\overline{g}$ on $S^1\times S^n$ has positive bi-Ricci curvature for $n\geq 2$, hence a connected sum  
$M$ of $i$ copies of $S^1\times S^n$ has a  metric $g$ of positive bi-Ricci curvature by \cite{hoelzel, shen-ye}. If  $\alpha,\beta \in H^1(M,\mathbb{Z})=\mathbb{Z}^i$, then $\alpha \smile \beta=0$. Suppose $3\leq (n+1)\leq 7$.  The two-sided stable closed  minimal hypersurfaces for the metric $g$ have bounded area by \cite{chu-lee-zhu}.

\begin{proof}

Let $\Lambda$ be the set of homology classes $\sigma\in H_n(M,\mathbb{Z})$ such that there is a  homologically area-minimizing, connected, closed,  two-sided, smoothly embedded outside a set of codimension seven, multiplicity one, minimal 
hypersurface $\Sigma$ with $[\Sigma]=\sigma$.

  Suppose that the set $\Lambda$ is finite. Let $\Lambda=\{\sigma_1,\dots,\sigma_k\}$. Define $\tilde{\Lambda}$ to be the set of 
 $h=\{\sigma_{i_1},\dots,\sigma_{i_j}\}\subset \Lambda$  such that there is $\Sigma_{i_r}$, area-minimizing, connected, closed,  two-sided, smoothly embedded outside a set of codimension seven, multiplicity one, minimal 
hypersurface, with $[\Sigma_{i_r}]=\sigma_{i_r}$, $1\leq r\leq j$, such that $\{\Sigma_{i_1},\dots, \Sigma_{i_j}\}$ is mutually disjoint.
Let $\alpha_i\in H^1(M,\mathbb{Z})$ be the Poincar\'{e} dual of $\sigma_i$, $1\leq i\leq k$.

The cup product is dual to the intersection of homology classes. Hence, if $h=\{\sigma_{i_1},\dots,\sigma_{i_j}\}\in \tilde{\Lambda}$ then
$\alpha_{i_r}\smile \alpha_{i_{r'}}=0$ for any $1\leq r,r'\leq j$. This is because the cup product of  cohomology classes of degree one is antisymmetric and   $\{\Sigma_{i_1},\dots, \Sigma_{i_j}\}$ is mutually disjoint.  Define $V_h\subset H^1(M,\mathbb{R})$ to be the real vector space
$$
V_h=\{\sum_{r=1}^j v_r \alpha_{i_r}: v_r\in \mathbb{R}\}.
$$
Then, for any $\alpha,\beta \in V_h$,  $\alpha \smile \beta =0$. Since the cup product on $H^1(M,\mathbb{R})$ is nontrivial,  it follows that
$V_h \neq H^1(M,\mathbb{R})$.

Hence, since the set $\tilde{\Lambda}$ is finite and $V_h$ is a vector space for $h\in \tilde{\Lambda}$, 
$$
(\cup_{h\in \tilde{\Lambda}} V_h) \neq H^1(M,\mathbb{R}).
$$
If $H^1(M,\mathbb{R})=\mathbb{R}^c$, $c=b_1(M,\mathbb{Z})$,  there is $(r_1,\dots,r_c) \in H^1(M,\mathbb{R})\setminus (\cup_{h\in \tilde{\Lambda}} V_h)$. By density, we can suppose that the numbers $r_i$ are rational since $\tilde{\Lambda}$ is finite. 
It follows that there is $q\in \mathbb{N}$ such that $q \cdot (r_1,\dots,r_c) \in H^1(M,\mathbb{Z})\setminus (\cup_{h\in \tilde{\Lambda}} V_h)$. Let $\sigma\in H_n(M,\mathbb{Z})$ be the Poincar\'{e} dual of $q \cdot (r_1,\dots,r_c)$.

By area-minimization, there is $T\in \sigma$ an integral current such that ${\bf M}(T)\leq {\bf M}(S)$ for any integral current $S\in \sigma$. The support ${\rm spt}(T)$ is a closed set that is a smoothly embedded outside a set of codimension seven minimal hypersurface.  Let $\{\Sigma_1,\dots,\Sigma_q\}$ be the connected components of ${\rm spt}(T)$. By \cite{ilmanen}, the regular set of $\Sigma_{r}$ is connected for $1\leq r\leq q$. It follows by the constancy theorem for integral cycles that
$$
T=m_1\Sigma_1+\dots + m_q\Sigma_q,
$$
where $\{m_1,\dots,m_q\}\subset \mathbb{Z}$, and that $\Sigma_r$ is  homologically area-minimizing and two-sided.  Therefore $[\Sigma_r]\in \Lambda$ for any $1\leq r \leq q$, and $h=\{[\Sigma_1],\dots,[\Sigma_q]\}\in \tilde{\Lambda}$
since $\{\Sigma_1,\dots,\Sigma_q\}$ is mutually disjoint. It follows that $q \cdot (r_1,\dots,r_c)\in V_h$, which is a contradiction.

Hence the set $\Lambda$ is infinite, and let $\{\sigma_i\}\subset \Lambda$ be a sequence with $\sigma_i\neq \sigma_j$ for $i\neq j$.
It follows by the definition  that  there is a connected,  two-sided, homologically area-minimizing,   smoothly embedded outside a set of codimension seven, closed  minimal 
hypersurface $\Sigma_i$ with $[\Sigma_i]=\sigma_i$ for any $i$.

 
 Suppose that there is $\lambda>0$ such that ${\bf M}_g(\Sigma_i)\leq \lambda$ for any $i$. Then, by compactness of integral currents with mass bounds,  there is a sequence $\{j\}\subset \{i\}$ such that the hypersurfaces $\Sigma_j$ converge in the flat topology. This is a contradiction since  $[\Sigma_j]\neq [\Sigma_{j'}]$ for $j\neq j'$. This  finishes the proof of the theorem.
 

\end{proof}

\begin{thm}\label{cohomology.index}
Let $M^{n+1}$, $3\leq (n+1)\leq 7$, be a closed manifold such that the cup product 
$$
\smile\,: H^1(M,\mathbb{Z})\times H^1(M,\mathbb{Z})\rightarrow H^2(M,\mathbb{Z})
$$
is nontrivial. Then, for any smooth bumpy Riemannian metric $g$ on $M$, there is a sequence of smooth, two-sided,  connected closed embedded minimal 
hypersurfaces  $\Sigma_i\subset (M,g)$, of Morse index one, such that  ${\rm vol}_g(\Sigma_i)\rightarrow \infty$. 
\end{thm}

\begin{proof}
By Theorem \ref{stable},  there is a sequence of smooth, two-sided, homologically area-minimizing, connected closed embedded minimal 
hypersurfaces  $\Sigma_i'\subset (M,g)$, such that   ${\rm vol}_g(\Sigma_i')\rightarrow \infty$. 

The theorem is proved by using the Proposition \ref{stable.index}, since the hypersurfaces $\Sigma_i'$ are stable.
\end{proof}

An interesting case is that of torus bundles with hyperbolic monodromy. The gluing map $\phi:T^n\rightarrow T^n$ induces an isomorphism
$\phi_*\in GL(n,\mathbb{Z})$ of $H_1(T^n)=\mathbb{Z}^n$, the monodromy of $\phi$,  and
$$
b_1(M_\phi)=1 + {\rm dim} \, {\rm Ker}(\phi_*-I).
$$ 
Hence $b_1(M_\phi)=1$ if $\phi_*$ has no eigenvalue equal to 1. 

 If the linear map  $\phi_*$ is  hyperbolic, in the sense that it has no complex eigenvalue with modulus 1, then 1 is not an eigenvalue of  $(\phi_*)^i$ for any $i\in \mathbb{Z}$, $i\neq 0$.
  Since a finite covering of $M_\phi$ can be covered by a torus bundle with monodromy $(\phi_*)^i$ for an integer $i\neq 0$, the manifold
$M_\phi$ has first virtual Betti number $vb_1(M_\phi)=1$. This suggests that at most one stable closed minimal hypersurface modulo coverings can be produced by area minimization in homology classes.

We discuss the case of stable minimal surfaces in three-manifolds.
  
\begin{prop}\label{minimal.spheres}
Let $(M,g)$ be a closed orientable three-dimensional Riemannian manifold. There is a disjoint  finite set $\tilde{\Lambda}$ of surfaces in $M$ such that  a surface $\Sigma\in \tilde{\Lambda}$ is either a stable minimal embedded sphere or a minimal embedded projective plane with stable two-sided cover, and if $\Omega$ is a connected component of $M \setminus (\cup_{\Sigma\in \tilde{\Lambda}}\Sigma)$ then the closed manifold $\tilde{\Omega}_1$ obtained by gluing three-balls to the metric closure $\tilde{\Omega}$ of $\Omega$ with boundary identification is irreducible.
\end{prop}

\begin{proof}
Let $\Lambda'$ be the set of embedded minimal projective planes with stable oriented two-cover.  A neighborhood of an
embedded projective plane is diffeomorphic to the complement of a ball in $\mathbb{RP}^3$. It follows by the prime decomposition theorem that there is a finite set $\Lambda\subset \Lambda'$ of mutually disjoint surfaces  that is not contained in any other set with this property.  Let $\Omega=M\setminus (\cup_{\Sigma\in \Lambda} \Sigma)$. Then $\Omega$
  is  connected, and is such that   there is no embedded minimal projective plane with stable oriented two-cover in $\Omega$ and  the metric closure $\tilde{\Omega}$ is a compact manifold with $\partial \tilde{\Omega}$ a union of stable minimal spheres.
  
  \medskip
  
{\bf Claim:}  Let $(\Omega',h)$ be a compact Riemannian manifold possibly with a boundary $\partial\Omega'$. Suppose that $\partial \Omega'$ if nontrivial is a union of stable minimal spheres, and let $\Omega'_1$  be the closed manifold obtained by gluing three-balls to $\Omega'$ with boundary identification. Suppose that there are no embedded minimal projective planes with stable two-sided cover in $\Omega'$.  If $\Omega_1'$ is not irreducible, then there is a stable minimal embedded two-sphere $\tilde{\Sigma}\subset {\rm int}(\Omega')$ such that $\tilde{\Sigma}$  does not bound a three-ball in $\Omega_1'$.

\medskip

Since $\Omega_1'$ is not irreducible, there is an embedded sphere $\Sigma\subset \Omega_1'$ that does not bound a three-ball. By using an isotopy we can suppose that $\Sigma\subset {\rm int}(\Omega')$. Let $\mathcal{I}(\Sigma)$ be the isotopy class of the surface $\Sigma$ in $\Omega'$, and $\{\Gamma_i\}$ in $\mathcal{I}(\Sigma)$ be a minimizing sequence. Then it follows from \cite{meeks-simon-yau} that there is a sequence  $\{\tilde{\Gamma}_i\}_i$, after passing to a subsequence, such that $\tilde{\Gamma}_i$ is obtained from $\Gamma_i$ by performing  $\gamma$-reductions and so that $\tilde{\Gamma}_i$ is isotopic to 
$$
\Gamma_i' \cup \left(\cup_{k=1}^r \Gamma_{i,k}\right),
$$ 
 $\Gamma_{i,k}$ is a disjoint union of  surfaces parallel to an embedded stable minimal sphere $\Sigma_k\subset \Omega'$ for any $1\leq k \leq r$ and $\lim_{i\rightarrow \infty} {\rm area}_g(\Gamma_i')=0.$ 

If $\Sigma_k$ is a component of $\partial \Omega'$, then the surface $\Sigma_k$ is the boundary of   a three-ball in $\Omega'_1$.  The components
of $\Gamma_i'$ bound three-balls by \cite{meeks-simon-yau}. Hence, since $\Sigma$ does not bound a three-ball,   by the topological claim in the proof of Proposition \ref{area.surface} it follows that there is $1\leq k \leq r$ such that $\Sigma_k\subset {\rm int}(\Omega')$ and $\Sigma_k$ is not the boundary of a three-ball.  This finishes the proof of the claim.

Let $\tilde{\Omega}_1$ be the closed manifold obtained by gluing three-balls to $\tilde{\Omega}$ with boundary identification. Then $M$ is the connected sum of $\tilde{\Omega}_1$ and $\#(\Lambda)$ copies of $\mathbb{RP}^3$. If $\tilde{\Omega}_1$ is irreducible, the proposition is proved.

Suppose $\tilde{\Omega}_1$ is not irreducible.  It follows by the claim that there is a stable minimal embedded sphere $\tilde{\Sigma}\subset {\rm int}(\tilde{\Omega})$ that does not bound a three-ball in $\tilde{\Omega}_1$. If the surface $\tilde{\Sigma}$ is  separating in $\tilde{\Omega}$, then $\tilde{\Omega}\setminus \tilde{\Sigma}$ has two connected components $\Omega^{(1)}$, $\Omega^{(2)}$, such that the boundary of the metric closure $\tilde{\Omega}^{(i)}$ of $\Omega^{(i)}$ is a union of stable minimal spheres. If $\tilde{\Omega}^{(i)}_1$ is the manifold obtained by gluing balls with boundary identification to $\tilde{\Omega}^{(i)}$, then $\tilde{\Omega}_1$ is the connected sum of $\tilde{\Omega}^{(i)}_1$, $i=1,2$. Since $\tilde{\Sigma}$ does not bound a three-ball in $\tilde{\Omega}_1$, the closed manifolds  $\tilde{\Omega}^{(i)}_1$, $i=1,2$, are not diffeomorphic to the three-sphere. Hence $\tilde{\Omega}^{(1)}_1 \# \tilde{\Omega}^{(2)}_1$ is a nontrivial connected sum.  In this case we continue by applying the claim to $\tilde{\Omega}^{(i)}$ if $\tilde{\Omega}^{(i)}_1$ is not irreducible. If the surface $\tilde{\Sigma}$ is not  separating in $\tilde{\Omega}$, then $\Omega^{(1)}=\tilde{\Omega}\setminus \tilde{\Sigma}$ is connected  and $\partial \tilde{\Omega}^{(1)}$ is a union of stable minimal spheres. The manifold $\tilde{\Omega}_1$ is the connected sum of $\tilde{\Omega}^{(1)}_1$ with $S^2\times S^1$. We continue by applying the claim to $\tilde{\Omega}^{(1)}$ if $\tilde{\Omega}^{(1)}_1$ is not irreducible. It follows by the prime decomposition theorem that this repeated application of the claim stops after finitely many steps. This proves the proposition.

\end{proof}

There is an open set of metrics for any closed three-manifold such that there are stable embedded minimal tori     with unbounded area (\cite{colding-minicozzi}, \cite{colding-hingston}). The extension to surfaces of any genus follows by \cite{dean}, \cite{kramer}.

We prove for the stable case:

\begin{thm}\label{stable.case}
Let $M$ be a closed orientable three-dimensional manifold and $g$ be a smooth Riemannian metric on $M$.  Suppose that $M$ is not diffeomorphic to a connected sum of spherical quotients, copies of  $S^2\times S^1$ and torus bundles or semibundles that are not Seifert fibered. Then there is a sequence of  connected,  smooth, two-sided, virtually embedded, closed stable minimal surfaces $\psi_i:\Sigma_i\rightarrow M$, such that 
 ${\rm area}_g(\psi_i(\Sigma_i))\rightarrow \infty.$
\end{thm}

\begin{proof}

Let  $\tilde{\Lambda}$ be a set  as in  Proposition \ref{minimal.spheres}. 
We denote by   $\Omega_1,\dots, \Omega_p$  the connected components of the complement of $\cup_{\Sigma\in \tilde{\Lambda}}\Sigma$.
Define $\tilde{\Omega}_j$ to be the metric closure of $\Omega_j$, and by $\tilde{\Omega}_{j,1}$  the topological closed manifold obtained by gluing  three-balls with identification of the boundary to each connected component
 of $\partial \tilde{\Omega}_j$. 
 The manifold  $\tilde{\Omega}_{j,1}$ is irreducible for any $1\leq j\leq p$.
 
 Then $M$ is a connected sum of $\{\tilde{\Omega}_{j,1}\}_{1\leq j \leq p}$ and copies of $\mathbb{RP}^3$ and $S^2 \times S^1$.
 By hypothesis  it follows that  there is $1\leq i\leq p$ such that $\tilde{\Omega}_{i,1}$ has infinite fundamental group and is not diffeomorphic to a torus bundle or semibundle that is not Seifert fibered.
 
 

Let $\{\Sigma_s\}_{s=1}^t$ be the embedded tori  of the canonical JSJ decomposition of $\tilde{\Omega}_{i,1}$. We can suppose that $\Sigma_s\subset {\rm int}(\tilde{\Omega}_i)$ for any $1\leq s\leq t$.  If any connected component of the complement of $\cup_{s=1}^t \Sigma_s$ in $\tilde{\Omega}_{i,1}$ is diffeomorphic to $T^2\times I$ or to the twisted $I$-bundle with base a Klein bottle, then $\tilde{\Omega}_{i,1}$ is a torus bundle or semibundle. In that case by hypothesis $\tilde{\Omega}_{i,1}$ is Seifert fibered and hence there is no torus in the JSJ decomposition. Therefore we can suppose that there is a component of the complement of $\cup_{s=1}^t \Sigma_s$ in $\tilde{\Omega}_{i,1}$ that is not diffeomorphic to $T^2\times I$ or to the twisted $I$-bundle with base a Klein bottle.
 

The surface  $\cup_{s=1}^t \Sigma_s$ 
 is isotopic to $\cup_{\eta=1}^q \Gamma_{\eta}'\subset {\rm int}(\tilde{\Omega}_i)$ in ${\rm int}(\tilde{\Omega}_i)$, where $\Gamma_\eta'$ is a union of parallel surfaces to a stable minimal embedded torus $T_\eta$ or to a minimal Klein bottle $T_\eta$ with stable oriented two-cover depending on $\eta$. 
 Hence there is a connected component $R$ of $\tilde{\Omega}_i\setminus (\cup_{\eta=1}^q T_\eta)$,  the boundary  $\partial\tilde{R}$ of the metric closure of $R$ a union of minimal tori and minimal spheres, such that the manifold $R'$ obtained by gluing balls with boundary identification to $\tilde{R}$ is diffeomorphic to a JSJ component of $\tilde{\Omega}_{i,1}$ that is not   $T^2\times I$ or the twisted $I$-bundle with base a Klein bottle. Then $R'$ is either  Seifert fibered or atoroidal.

 In the case that $R'$ is atoroidal, there is a complete finite volume hyperbolic metric on ${\rm int}(R')$. If $R'$ is a closed manifold, there are immersed incompressible surfaces by \cite{kahn-markovic, kahn-markovic2}. In the case there is  a boundary $\partial R'$, there are  immersed closed incompressible surfaces in ${\rm int}(R')$ by \cite{kahn-wright}. These surfaces are homotopic to closed area-minimizing surfaces in the hyperbolic metric $\overline{g}$, and there is a  sequence of these area-minimizing surfaces $\psi_i:\Sigma_i\rightarrow {\rm int}(R')$ that are unique in their homotopy class and that become dense in the Grassmannian of  ${\rm int}(R')$  (\cite{assal-lowe}). Hence  ${\rm area}_{\overline{g}}(\psi_i(\Sigma_i))\rightarrow \infty$.  
  
  Since $R'$ is the manifold that is obtained by gluing balls with boundary identification to $\tilde{R}$, the immersion 
  $\psi_i:\Sigma_i\rightarrow {\rm int}(R')$ is homotopic to a map $\psi_i':\Sigma_i\rightarrow {\rm int}(\tilde{R})$.
  The boundary $\partial \tilde{R}$ is a minimal  surface for the metric $g$. By area-minimization in $\tilde{R}$, there is an immersed $g$-minimal surface $\tilde{\psi_i}:\Sigma_i\rightarrow \tilde{R}$ in the homotopy class of $\psi_i':\Sigma_i\rightarrow \tilde{R}$ (\cite{sacks-uhlenbeck}, \cite{schoen-yau}). Since the components of $\partial\tilde{R}$ are diffeomorphic to spheres and tori, and ${\rm genus}(\Sigma_i)\geq 2$, it follows by the maximum principle that $\tilde{\psi}_i(\Sigma_i)\subset {\rm int}(\tilde{R})$. Hence $\tilde{\psi_i}(\Sigma_i)\subset M$.  There are $\delta,C>0$ such that the map $\tilde{\psi}_i$ is homotopic to a map $\phi_i$ such that
  $\phi_i(\Sigma_i)$ is disjoint from $\{x\in \tilde{R}:d_g(x,\partial\tilde{R})\leq \delta\}$ and ${\rm area}_g(\phi_i(\Sigma_i))\leq C{\rm area}_g(\tilde{\psi}_i(\Sigma_i))$. Since ${\rm area}_{\overline{g}}(\psi_i(\Sigma_i))\leq  {\rm area}_{\overline{g}}(\phi_i(\Sigma_i))$ and the metrics $g$ and $\overline{g}$ are uniformly equivalent on $\{x\in \tilde{R}:d_g(x,\partial\tilde{R})\geq \delta\}$, it follows that ${\rm area}_g(\tilde{\psi}_i(\Sigma_i))\rightarrow \infty$.

  The surface $\tilde{\psi}_i:\Sigma_i\rightarrow {\rm int}(R')$ is in the homotopy class of  $\psi_i:\Sigma_i\rightarrow {\rm int}(R')$. 
  The principal curvatures of $\psi_i(\Sigma_i)$ in the hyperbolic metric are bounded by one, hence the map $\psi_i$ lifts to an embedding in the covering 
  $\pi_i:R'_{\Sigma_i}\rightarrow {\rm int}(R')$ inducing the group $(\psi_i)_{\#}(\pi_1(\Sigma_i))$ (\cite{uhlenbeck}). Therefore $\tilde{\psi}_i$ lifts to 
  a map $\tilde{\psi}_i':\Sigma_i\rightarrow R'_{\Sigma_i}$ that minimizes area in the metric $\pi_i^*(g)$. The map $\tilde{\psi}_i'$ is an embedding by \cite{freedman-hass-scott}. Therefore, by \cite{scott}, 
  the surfaces $\tilde{\psi}_i:\Sigma_i\rightarrow M$ are virtually embedded in $\tilde{R}$ because $\pi_1(\tilde{R})=\pi_1(R')$ is LERF (\cite{agol}, \cite{wise}). Hence the surfaces  $\tilde{\psi}_i(\Sigma_i)$ are virtually embedded in $M$ as in the proof of Theorem \ref{index.surfaces}. 
  This finishes the proof of the theorem if $R'$ is atoroidal.

  Suppose that $R'$ is Seifert fibered.  Then there is a finite covering $\pi:R''\rightarrow R'$ that is an oriented circle bundle  with base an orientable compact surface $X$.
    
  Let $\sigma$ be an embedded loop in $X$, and $\Sigma$ be the torus in $R''$ that projects to $\sigma$.  If the loop $\sigma$ is homotopically nontrivial in $X$, $\Sigma$ is an incompressible embedded torus in $R''$. The manifold $R'\setminus \tilde{R}$ is a union of disjoint open balls. Hence $\pi^{-1}(R'\setminus \tilde{R})$ is a union of disjoint balls.  Hence the surface $\Sigma$ is isotopic to a surface $\Sigma'$ in $R''$ that does not intersect $\pi^{-1}(R'\setminus \tilde{R})$. Therefore $\pi(\Sigma')\subset \tilde{R}$.
  

  By area-minimization, since $\partial (\pi^{-1}(\tilde{R}))$ is a minimal surface for the metric $\pi^*(g)$, the surface $\Sigma'$ is isotopic in 
  $\pi^{-1}(\tilde{R})$ to either a minimal   torus $\Sigma''$ or to the boundary of a tubular neighborhood of a minimal   Klein bottle. Since there is no incompressible Klein bottle in $R''$, $\Sigma'$ is isotopic to a  torus $\Sigma''$.
  
  Therefore, if $\sigma$ is a homotopically nontrivial embedded loop  in $X$ the torus that projects to $\sigma$ is isotopic to a stable minimal embedded torus $\Sigma''_\sigma$ for the metric $\pi^*(g)$ such that $\pi(\Sigma''_\sigma)\subset \tilde{R}$.
  If $\Sigma''_{\tilde{\sigma}}$ and $\Sigma''_{\sigma}$ are isotopic, then $\tilde{\sigma}$ is conjugate to either $\sigma$ or $-\sigma$ in $\pi_1(X)$. 
  
If ${\rm genus}(X)\geq 1$, then there is a sequence $\{\sigma_i\}\subset \pi_1(X)$ of embedded loops such that $\sigma_i$ is not conjugate to $\sigma_{i'}$ or $\sigma_{i'}^{-1}$ for $i\neq i'$. If $X$ is the complement of $h\geq 4$ disks in the sphere, then there is such a sequence (\cite{colding-minicozzi}). Hence in these cases there is a sequence of stable minimal embedded tori $\{\Sigma_i\}$ for the metric $\pi^*(g)$ such that $\Sigma_i$ and  $\Sigma_{i'}$ are not isotopic for $i\neq i'$.  By the curvature estimates for stable surfaces, ${\rm area}_{\pi^*(g)}(\Sigma_i)\rightarrow \infty$.  The surfaces are such that $\pi(\Sigma_i)\subset {\rm int}(\tilde{R})$ for $i\geq i'$. Since ${\rm area}_g(\pi(\Sigma_i))\geq {\rm deg}(\pi)^{-1}{\rm area}_{\pi^*(g)}(\Sigma_i)$, it follows that ${\rm area}_g(\pi(\Sigma_i))\rightarrow \infty$. The surfaces $\pi(\Sigma_i)$ are virtually embedded in $\tilde{R}$ and hence in $M$ as in the proof of  of Theorem \ref{index.surfaces}.

Suppose that the surface $X$ is planar and has three boundary components. Then there  is a two-cover $X'$ that is diffeomorphic to  the complement of four disks in the sphere. Hence the incompressible surfaces can be constructed as in the previous paragraph by using the Seifert manifold with base $X'$ that is a covering of $R''$. 
This proves the theorem if $X\neq S^1\times I$. 

If $X=S^1\times I$, then $R''=T^2\times I$ and $R'$ is a Seifert fibered space with base a quotient of $S^1\times I$.
It follows that $R'$ is diffeomorphic to either $T^2\times I$ or to the twisted $I$-bundle with base a Klein bottle. This is a contradiction, which finishes the proof of the theorem.

\end{proof}

 By \cite{lima}, a manifold that is not diffeomorphic to a spherical quotient has a stable, closed,     two-sided,  immersed  minimal surface for any metric. There is a construction  of closed, irreducible  manifolds  in \cite{lima} satisfying the hypotheses of Theorem \ref{stable.case} which  do not have closed, embedded, smooth
minimal surfaces with stable oriented two-cover.

The case of stable surfaces in a hyperbolic mapping torus is not covered by Theorem \ref{stable} and Theorem \ref{stable.case}.
The hypotheses of Theorem \ref{nonspherical-introduction}
and of  the Theorem \ref{fiber.bundle-introduction} are satisfied by these spaces.


\begin{thebibliography}{99}

\bibitem{agol}
Agol, I.,
\textit{The virtual Haken conjecture, with an Appendix by Ian Agol, Daniel Groves,
and Jason Manning,}
Documenta Mathematica 18 (2013), 1045--1087


\bibitem{almgren} 
Almgren, F., \textit{The homotopy groups of the integral cycle groups,} Topology  (1962), 257--299. 

\bibitem{almgren-varifolds}
Almgren, F., \textit{The theory of varifolds,} Mimeographed notes, Princeton (1965).

\bibitem{ambrozio-marques-neves-rigidity}
Ambrozio, L., Marques, F. C., and Neves, A.,
\textit{Rigidity theorems for the area widths of Riemannian manifolds,}
arXiv:2408.14375 [math.DG] (2024)

\bibitem{aschenbrenner-friedl-wilton}
Aschenbrenner, M.,  Friedl, S., and Wilton, H.,
\textit{Three-manifold groups,}
Series of Lectures in Mathematics, Eur. Math. Soc. (2015)

\bibitem{assal-lowe}
Assal, F. A., and Lowe, B.,
\textit{Asymptotically geodesic surfaces,}
arXiv:2502.17303 [math.DG]  (2025)




\bibitem{brooks}
Brooks, R.,
\textit{The fundamental group and the spectrum of the Laplacian,}
Comment. Math. Helv. 56 (1981), 581--598




\bibitem{carlotto-li}
Carlotto, A., and Li, C.,
\textit{Constrained deformations of positive scalar curvature metrics,}
 J.  Differ. Geom. 126 (2), (2024), 475--554.
 
 \bibitem{catino-mastrolia-roncoroni}
 Catino, G., Mastrolia, P., and  Roncoroni, A.,
 \textit{Two rigidity results for stable minimal hypersurfaces,}
  Geom. and Funct. Anal. 34 (1), (2024), 1--18.
 
 \bibitem{cecchini-hanke-schick}
 Cecchini, S.,  Hanke, B., and  Schick, T.,
 \textit{Lipschitz rigidity for scalar curvature,}
 	arXiv:2206.11796 (2022)
	
\bibitem{chodosh-ketover-maximo}
Chodosh,  O., Ketover, D., and Maximo, D.,
\textit{Minimal hypersurfaces with bounded index,}
Invent. Math. 209 (3), (2017), 617--664
 
 

 
 \bibitem{chu-lee-zhu}
 Chu, J.,  Lee, M.-C., and Zhu, J.,
 \textit{Homological $n$-systole in $(n+1)$-manifolds and bi-Ricci curvature,}
 Adv. Math. 467 (2025), 110187
 

 
 \bibitem{colding-hingston}
 Colding, T., and  Hingston, N.,
 \textit{Metrics without Morse index bounds,}
 Duke Math. J. 119 (2003), 345--365.
 
 \bibitem{colding-minicozzi}
 Colding, T., and  Minicozzi, W.,
 \textit{Examples of embedded minimal tori without area bounds,}
 Int. Math. Res. Not. 20 (1999), 1097--1100
 
 \bibitem{dean}
 Dean, B.,
 \textit{Compact embedded minimal surfaces of positive genus without area bounds,}
  Geom. Dedicata 102 (2003), 45--52. 
  
   \bibitem{fraser-schoen}
  Fraser, A., and Schoen, R.,
  \textit{Stability and largeness properties of minimal surfaces in higher codimension,}
 	arXiv:2303.07423 [math.DG] (2023)
  
 
 \bibitem{freedman-hass-scott}
 Freedman, M.,  Hass, J., and  Scott, P.,
 \textit{Least area incompressible surfaces in three-manifolds,} 
 Invent. Math.  71 (3) (1983), 609-642.
 

 
 \bibitem{gromov-lawson-scalar}
 Gromov, M., and   Lawson, H. B.,
 \textit{Spin and scalar curvature in the presence of a fundamental group I}
  Ann. of Math. 111 (2) (1980), 209--230
 
 \bibitem{gromov-lawson}
Gromov, M.,  and  Lawson, H. B.,
 \textit{Positive scalar curvature and the Dirac operator on complete Riemannian manifolds,}
 Publ. Math. IHES 58 (1983), 295--408.
 
 \bibitem{hatcher_three-manifold}
Hatcher, A.,
\textit{Notes on basic 3-manifold topology,}
online book (2007)
 
 \bibitem{hoelzel}
 Hoelzel, S.,
 \textit{Surgery stable curvature conditions,}
 Math. Ann. 365 (1) (2016), 13--47. 
 

 
 \bibitem{ilmanen}
Ilmanen, T.,
\textit{A strong maximum principle for singular minimal hypersurfaces,}
Calc. Var. Partial Differential Equations 4 (5) (1996), 443--467.

 
 \bibitem{kahn-markovic} 
 Kahn, J.,  and Markovic,  V.,
\textit{Immersing almost geodesic surfaces in a closed hyperbolic three manifold,}
Ann. of Math.  175 (2012), 1127--1190. 


\bibitem{kahn-markovic2} 
Kahn, J. and Markovic,  V.,
\textit{Counting essential surfaces in a closed hyperbolic three-manifold,} Geom. Topol. 16 (2012), 601--624. 

\bibitem{kahn-wright}
Kahn, J.  and Wright, A.,
\textit{Nearly Fuchsian surface subgroups of finite covolume Kleinian groups,}
Duke Math. J. 170 (3) (2021), 503--573

\bibitem{kazhdan}
Kazhdan, D.,
\textit{Connection of the dual space of a group with the
structure of its closed subgroups,}
 Funct. Anal. Appl. 1  (1) (1967), 63--65.

 
 
 \bibitem{ketover-liokumovich-song}
 Ketover, D., Liokumovich, Y.  and Song, A.,
 \textit{On the existence of minimal Heegaard surfaces,}
 	arXiv:1911.07161 [math.DG] (2019)
	
	\bibitem{ketover-marques-neves}
Ketover, D.,   Marques, F. C., and  Neves, A.,
 \textit{The catenoid estimate and its geometric applications,} 
J. Differ. Geom. 115 1 (2020), 1-26.



\bibitem{kramer}
Kramer, J.,
\textit{Examples of stable embedded minimal spheres without area bounds,}
	arXiv:0812.3841 [math.DG] (2008)


 
 \bibitem{lima}
 Lima, V.,
 \textit{On a conjecture of Meeks, P\'{e}rez and Ros,}
 J. Differ. Geom. 125 (3) (2023), 613--622.

\bibitem{llarull}
Llarull, M.,
\textit{Sharp estimates and the Dirac operator,}
Math. Ann. 310 (1998),   55--71.




\bibitem{marques-neves-rigidity-spheres}
Marques, F. C., and Neves, A.,
\textit{Rigidity of min-max minimal spheres in three-manifolds,}
Duke Math. J. 161 (14) (2012), 2725--2752

\bibitem{marques-neves-index}
Marques, F. C.,  and  Neves, A.,
\textit{Morse index and multiplicity of min-max minimal hypersurfaces,}
Camb. J. Math. 4 (4) (2016),  463--511



\bibitem{marques-neves-lower-bound}
Marques, F. C., and Neves, A., 
\textit{Morse index of multiplicity one min-max minimal hypersurfaces,}
 Adv. Math. 378 (2021), 107527
 
 \bibitem{mazet}
 Mazet, L.,
 \textit{Stable minimal hypersurfaces in $\mathbb{R}^6$,}
 	arXiv:2405.14676 [math.DG] (2024)
 
 \bibitem{mazet-rosenberg}
 Mazet, L., and Rosenberg, H.,
 \textit{Minimal hypersurfaces of least area,}
 J. Differ. Geom. 106 2 (2017), 283--316.
 
 \bibitem{meeks-simon-yau} 
Meeks, W., Simon, L., and Yau, S.-T., 
 \textit{Embedded minimal surfaces, exotic spheres, and manifolds with positive Ricci curvature,} 
Ann. of Math.  116 (1982), 621--659. 

\bibitem{montezuma}
 Montezuma,  R.,
 \textit{Metrics of positive scalar curvature and unbounded min-max widths,}
  Calc. of Var. and PDEs, 55 (6) (2016), p. 139.
 
 
 \bibitem{perelman}
 Perelman, G.,
  \textit{Ricci flow with surgery on three-manifolds,} arXiv:math/0303109 (2003)
 
 \bibitem{pitts} 
Pitts, J.,
 \textit{Existence and regularity of minimal surfaces on Riemannian manifolds,} Mathematical Notes 27, Princeton University Press, Princeton, (1981).

\bibitem{sacks-uhlenbeck}
Sacks, J. and Uhlenbeck, K.,  
\textit{Minimal immersions of closed Riemann surfaces,}
 Trans. of the Amer. Math. Soc., 271 (2) (1982), 639--652.

\bibitem{schoen} 
Schoen, R.,
 \textit{Estimates for stable minimal surfaces in three-dimensional manifolds,}  Seminar on minimal submanifolds, 111--126,
Ann. of Math. Stud. 103,  Princeton Univ. Press, 1983.


\bibitem{schoen-simon} 
Schoen, R., and Simon, L.,
 \textit{Regularity of stable minimal hypersurfaces,} 
Comm. Pure Appl. Math. 34 (1981), 741--797. 

\bibitem{schoen-yau}
Schoen, R., and Yau, S.T.,
 \textit{Existence of incompressible minimal surfaces and the topology of three dimensional manifolds of non-negative scalar curvature,} 
 Ann. of Math. 110 (1979), 127--142.



\bibitem{scott}
Scott, P.,
\textit{Subgroups of surface groups are almost geometric,}
J. London Math. Soc. 17 (2) (1978), 555--565. 



\bibitem{sharp}
Sharp, B.,
\textit{Compactness of minimal hypersurfaces with bounded index,} 
J. Differential Geom. 106 (2) (2017),  317--339. 

\bibitem{shen-ye}
Shen, Y., and  Ye, R.,
\textit{On the geometry and topology of manifolds of positive bi-Ricci
curvature,}
arXiv:dg-ga/9708014 (1997)

\bibitem{song}
Song, A.,
 \textit{Embeddedness of least area minimal hypersurfaces,}
J. Differ. Geom. 110 (2) (2018), 345--377.

\bibitem{song-existence}
Song, A.,
\textit{Existence of infinitely many minimal hypersurfaces in closed manifolds,}
 Ann. of Math. 197 (3) (2023), 859--895.
 
 \bibitem{uhlenbeck} 
 Uhlenbeck,  K.,
 \textit{Closed minimal surfaces in hyperbolic three-manifolds,}  Seminar on minimal submanifolds,  Ann. of Math. Stud. 103 (1983), 147--168.  

\bibitem{wang-zhou}
Wang, Z., and  Zhou, X., 
 \textit{Existence of four minimal spheres in $S^3$ with a bumpy metric,}
arXiv:2305.08755 [math.DG] (2023)



\bibitem{white-bumpy}
White, B.,
\textit{On the bumpy metrics theorem for minimal submanifolds,}
Amer. J. Math. 139 (4) (2017),  1149--1155.

\bibitem{wilton-zalesskii}
Wilton, H., and Zalesskii, P.,
\textit{Profinite properties of graph manifolds,}
Geom. Dedicata 147 (2010), 29--45

\bibitem{wise}
Wise, D.,
\textit{The structure of groups with quasiconvex hierarchy,}
Ann. of Math. Studies 366 (2021)

\bibitem{zhou-multiplicity}
Zhou, X., 
\textit{On the multiplicity one conjecture in min-max theory,}
Ann. of Math. 192 (3) (2020), 767-820.





\end{thebibliography}
\end{document}